\documentclass[12pt, onesided, leqno]{amsart}
\usepackage{enumerate}
\usepackage[all]{xy}
\usepackage{mathtools}
\usepackage{amsmath,amssymb,mathrsfs,geometry}
\newcommand{\doublespace}
   {\addtolength{\baselineskip}{0.25\baselineskip}}

\newtheorem{thm}{Theorem}[section]
\newtheorem{cor}[thm]{Corollary}
\newtheorem{lem}[thm]{Lemma}

\newtheorem{prop}[thm]{Proposition}

\theoremstyle{definition}
\newtheorem{pdef}[thm]{Definition}
\newtheorem{remark}[thm]{Remark}
\newtheorem{exam}[thm]{Example}
\newtheorem{cond}[thm]{Condition}

\numberwithin{equation}{section}

\def\bitimes{{\boxtimes\boxtimes}}  
\def\bitimet{{\hspace{1pt}\boxtimes\boxtimes\hspace{1pt}}}  

\def\origin{{\boldsymbol x_0}}

\def\biplus{{\boxplus\boxplus}}
\def\biplut{{\hspace{1pt}\boxplus\boxplus\hspace{1pt}}}

\title{Limit Theorems and Wrapping Transforms in Bi-free Probability Theory}
\author{Takahiro Hasebe and Hao-Wei Huang}
\address{Department of Mathematics, Hokkaido University, North 10 West 8, Kita-Ku, Sapporo 060-0810, Japan.}
\email{thasebe@math.sci.hokudai.ac.jp}
\address{Department of Applied Mathematics, National Sun Yat-sen University,
No. 70, Lienhai Road, Kaohsiung 80424, Taiwan, R.O.C.}
\email{hwhuang@math.nsysu.edu.tw}
\date{\today}
\RequirePackage[normalem]{ulem} 
\RequirePackage{color}\definecolor{RED}{rgb}{1,0,0}\definecolor{BLUE}{rgb}{0,0,1} 
\providecommand{\DIFadd}[1]{{\protect\color{blue}\uwave{#1}}} 
\providecommand{\DIFaddbegin}{} 
\providecommand{\DIFaddend}{} 

\begin{document}

\maketitle
\doublespace
\pagestyle{myheadings}
%

\begin{abstract} In this paper, we characterize idempotent distributions with respect to the bi-free multiplicative convolution on the bi-torus. Also, the bi-free analogous L\'{e}vy triplet of an infinitely divisible distribution on the bi-torus without non-trivial idempotent factors is obtained. This triplet is unique and generates a homomorphism from the bi-free multiplicative semigroup of infinitely divisible distributions to the classical one. The relevances of the limit theorems associated with four convolutions, classical and bi-free additive convolutions and classical and bi-free multiplicative convolutions, are analyzed. The analysis relies on the convergence criteria for limit theorems and the use of push-forward measures induced by the wrapping map from the plane to the bi-torus. Different from the bi-free circumstance, the classical multiplicative L\'{e}vy triplet is not always unique. Due to this, some conditions are furnished to ensure uniqueness.
\end{abstract}

\section{Introduction} The main aim of the present paper is to build the association among various limit theorems and their convergence criteria in classical and bi-free probability theories.

\emph{Bi-free probability theory}, introduced by Voiculescu in \cite{V14}, is an outspread research field of free probability theory, which grew out to intend to simultaneously study the left and right actions of algebras over reduced free product spaces. Since its creation, a great deal of research work has been conducted to better understand this theory and its connections to other parts of mathematics \cite{Sko16}\cite{Sko18}\cite{V15}\cite{V15ST}. Aside from the combinatorial means, the utilization of analytic functions as transformations and the bond to classical probability theory also play crucial roles in the study and comprehension of this theory \cite{HW1}\cite{HW2}. Especially, recent developments of bi-free harmonic analysis enable one to investigate bi-free limit theorems and other related topics from the probabilistic point of view \cite{HHW}.

To work in the probabilistic framework, we thereby consider the family $\mathscr{P}_\mathbf{X}$ of Borel probability measures on a complete separable metric space $\mathbf{X}$ and endow this family with a commutative and associative binary operation $\lozenge$. Classical and \emph{bi-free convolutions}, respectively denoted by $*$ and $\biplus$, are two examples of such operations performed on $\mathscr{P}_{\mathbb{R}^2}$. In probabilistic terms, $\mu_1*\mu_2$ is the probability distribution of the sum of two independent bivariate random vectors respectively having distributions $\mu_1$ and $\mu_2$. When restricted to compactly supported measures in $\mathscr{P}_{\mathbb{R}^2}$, $\mu_1\biplut\mu_2$ is the distribution of the sum of two bi-free bipartite self-adjoint pairs with distributions $\mu_1$ and $\mu_2$, respectively \cite{V14}. This new notion of convolution was later extended, without any limitation, to the
whole class $\mathscr{P}_{\mathbb{R}^2}$ by the continuity theorem of transforms \cite{HHW}. The product of two independent random vectors having distributions on the bi-torus $\mathbb{T}^2$ gives rise to the classical multiplicative convolution $\circledast$, and the bi-free analog of multiplicative convolution $\bitimes$ is defined in a similar manner \cite{V15ST}.

In (non-commutative) probability theory, the limit theorem and its related subject, the notion of \emph{infinite divisibility} of distributions, have attracted much attention. By saying that a distribution in $(\mathscr{P}_\mathbf{X},\lozenge)$ is infinitely divisible we mean that it can be expressed as the operation $\lozenge$ of an arbitrary number of copies of identical distributions from $\mathscr{P}_\mathbf{X}$.
The collection of measures having this infinitely divisible feature forms a semigroup and will be denoted by $\mathcal{ID}(\mathbf{X},\lozenge)$, or simply by $\mathcal{ID}(\lozenge)$ if the identification of the metric space is unnecessary. Any measure satisfying $\mu=\mu\lozenge\mu$, known as \emph{idempotent}, is an instance of infinitely divisible distributions. In the case of $\mathbf{X}=\mathbb{R}$, these topics have been thoroughly studied in classical probability by the efforts of de Finetti, Kolmogorov, L\'{e}vy, and Khintchine \cite{Par67}, and the same themes in the free contexts have also been deeply explored in the literature \cite{BerVoicu93}.

Bi-free probability, as expected, also parallels perfectly aspects of classical and free probability theories \cite{BerPata99}. For example, the theory of bi-freely infinitely divisible distributions generalizes \emph{bi-free central limit theorem} as they also serve as the limit laws for sums of bi-freely independent and identically distributed faces. Specifically, it was shown in \cite{HHW} that for some infinitesimal triangular array $\{\mu_{n,k}\}_{n\geq1,1\leq k\leq n_k}\subset\mathscr{P}_{\mathbb{R}^2}$ and sequence $\{\boldsymbol v_n\}\subset\mathbb{R}^2$, the sequence
\begin{equation} \label{munclass}
\delta_{\boldsymbol v_n}*\mu_{n1}*\cdots*\mu_{nk_n}
\end{equation} converges weakly
if and only if so does the sequence
\begin{equation} \label{munbifree}
\delta_{\boldsymbol v_n}\biplut\mu_{n1}\biplut\cdots\biplut\mu_{nk_n}.
\end{equation} The limiting distributions in (\ref{munclass}) and (\ref{munbifree}) respectively belong to the semigroups $\mathcal{ID}(*)$ and $\mathcal{ID}(\biplus)$, and their classical and bi-free \emph{L\'{e}vy triplets} agree. This conformity consequently brings out an isomorphism $\Lambda$ between these two semigroups.

Same tasks are performed in the case of bi-free multiplicative convolution in this paper. We
determine $\bitimes$-idempotent elements, which happen to be the distributions of \emph{bi-free $\mathrm{R}$-diagonal pairs} of operators in a bipartite system addressed in \cite{diagonal}, and
identify measures in $\mathscr{P}_{\mathbb{T}^2}$ bearing no non-trivial $\bitimes$-idempotent factors. Specifically, we demonstrate that $\nu\in\mathcal{ID}(\bitimes)$
has no non-trivial $\bitimes$-idempotent factor if and only if it belongs to $\mathscr{P}_{\mathbb{T}^2}^\times$, the sub-collection of $\mathscr{P}_{\mathbb{T}^2}$ with the attributes
\[\int_{\mathbb{T}^2}s_j\;d\nu(s_1,s_2)\neq0,\;\;\;\;\;j=1,2.\]

Fix an infinitesimal triangular array $\{\nu_{nk}\}_{n\geq1,1\leq k\leq k_n}\subset\mathscr{P}_{\mathbb{T}^2}$ and a sequence $\{\boldsymbol\xi_n\}\subset\mathbb{T}^2$.
We also manifest that the weak convergence of the sequence
\begin{equation} \label{bifreemul}
\delta_{\boldsymbol\xi_n}\bitimet\nu_{n1}\bitimet\cdots\bitimet\nu_{nk_n}
\end{equation} to some element in $\mathscr{P}_{\mathbb{T}^2}^\times$ yields the same property of the sequence
\begin{equation} \label{classicalmul}
\delta_{\boldsymbol\xi_n}\circledast\nu_{n1}\circledast\cdots\circledast\nu_{nk_n},
\end{equation}
and that their limiting distributions are both infinitely divisible. This is done by distinct types of equivalent convergence criteria offered in the present paper. As in the case of addition, there exists a triplet concurrently serving as
the classical and bi-free multiplicative L\'{e}vy triplets of the limiting distributions in (\ref{classicalmul}) and (\ref{bifreemul}). The consistency of their L\'{e}vy triplets, together with the description of $\mathcal{ID}(\bitimes)\backslash\mathscr{P}_{\mathbb{T}^2}^\times$, consequently produces a homomorphism $\Gamma$ from $\mathcal{ID}(\bitimes)$ to
$\mathcal{ID}(\circledast)$.

Because of the nature of $\mathcal{ID}(\bitimes)\backslash\mathscr{P}_{\mathbb{T}^2}^\times$ and that
the limit in (\ref{classicalmul}) may generally not have a unique L\'{e}vy measure, the homomorphism stated above is neither surjective nor injective. However, postulating the uniqueness of the L\'{e}vy measure, the weak convergence of (\ref{classicalmul}) derives that of (\ref{bifreemul}). The previously presented issue is addressed in the last section, where the uniqueness of L\'{e}vy measures is obtained under certain circumstances.

In addition to the previously mentioned conjunctions, what we would like to point out is that measures in $\mathscr{P}_{\mathbb{R}^2}$ and $\mathscr{P}_{\mathbb{T}^2}$ can be linked through the wrapping map $W\colon\mathbb{R}^2\to\mathbb{T}^2$, $(x,y)\mapsto(e^{ix},e^{iy})$. This wrapping map induces a map $W_\ast\colon \mathscr{P}_{\mathbb{R}^2} \to \mathscr{P}_{\mathbb{T}^2}$ so that the measure $\nu_{nk}=W_\ast(\mu_{nk})=\mu_{nk}W^{-1}$ enjoys the property: the weak convergence of (\ref{munclass}) or (\ref{munbifree}) yields the weak convergence of (\ref{bifreemul}) and (\ref{classicalmul}) with $\boldsymbol\xi_n=W(\boldsymbol v_n)$. Furthermore, the synchronous convergence allows one to construct a homomorphism $W_{\biplus}\colon\mathcal{ID}(\biplus)\to\mathcal{ID}(\bitimes)$ making the following diagram commute:
\begin{equation}\label{diagram}
\xymatrix@=40pt{
 \mathcal{ID}(\ast) \ar[r]^-{W_\ast} \ar[d]^-\Lambda & \mathcal{ID}(\circledast)  \\
 \mathcal{ID}(\biplus) \ar[r]^-{W_{\biplus}} & \mathcal{ID}(\bitimes) \ar[u]^-\Gamma &
}\end{equation}
This diagram is a two-dimensional analog of \cite[Theorem 1]{Ceb16}.

The rest of the paper is organized as follows. In Section \ref{sec2} we provide the necessary background in classical and non-commutative probability theories. In Section \ref{sec3} we characterize $\bitimes$-idempotent distributions. In Section \ref{sec4} we make comparisons of the convergence criteria of limit theorems, as well as those through wrapping transforms. Section \ref{sec5} is devoted to offering bi-free multiplicative L\'{e}vy triplets of infinitely divisible distributions and investigating the relationships among limit theorems in additive and multiplicative cases. Section \ref{sec6} covers the classical limit theorem on higher dimensional tori. Section \ref{sec7} provides the derivation of the diagram (\ref{diagram}). In the last section we examine the uniqueness of L\'{e}vy measures of infinitely divisible distributions on the circle.

\section{Preliminary}\label{sec2}
\subsection{Convergence of measures} Let $\mathscr{B}_\mathbf{X}$ be the collection of Borel sets on a complete separable metric space $(\mathbf{X},\mathrm{d})$. A point is selected from $\mathbf{X}$ and fixed, named the \emph{origin} and denoted by $\origin$ in the following. In the present paper, we will be mostly concerned with the abelian groups $\mathbf{X}=\mathbb R^d$ and $\mathbf{X}=\mathbb T^d$ endowed with the relative topology from $\mathbb{C}^d$, where the origin is chosen to be the unit. They are respectively the $d$-dimensional Euclidean metric space and the $d$-dimensional torus (or the \emph{$d$-torus} for short). The $1$-torus is just the unit circle $\mathbb{T}$ on the complex plane. A set contained in $\{\boldsymbol x\in\mathbf{X}:\mathrm{d}(\boldsymbol x,\origin)\geq r\}$ for some $r>0$ is colloquially said to be \emph{bounded away from the origin}.

Next, let us introduce several types of measures on $\mathbf{X}$ that will be discussed in the present paper. The first one is the collection $\mathscr{M}_\mathbf{X}$ of finite positive Borel measures on $\mathbf{X}$. We shall also consider the set $\mathscr{M}_\mathbf{X}^\origin$ of positive Borel measures whose each member confined on any Borel set bounded away from the origin is a finite measure. Clearly, we have $\mathscr{M}_\mathbf{X}\subset\mathscr{M}_\mathbf{X}^\origin$. Another assortment concerned herein is the collection $\mathscr{P}_\mathbf{X}$ of elements in $\mathscr{M}_\mathbf{X}$ having unit total mass.

The set $C_b(\mathbf{X})$ of bounded continuous functions on $\mathbf{X}$ induces the weak topology on $\mathscr{M}_\mathbf{X}$.
Likewise, $\mathscr{M}_\mathbf{X}^\origin$ is equipped with the
topology generated by $C_b^\origin(\mathbf{X})$, bounded continuous functions having support bounded away from the origin. Concretely, basic neighborhoods of a $\tau\in\mathscr{M}_\mathbf{X}^\origin$ are of the form
\[\bigcap_{j=1,\ldots,n}\left\{\widetilde{\tau}\in\mathscr{M}_\mathbf{X}^\origin:\left|\int f_j\,d\widetilde{\tau}-\int f_j\,d\tau\right|<\epsilon\right\},\] where $\epsilon>0$ and each $f_j\in C_b^\origin(\mathbf{X})$. Putting it differently, a sequence $\{\tau_n\}\subset\mathscr{M}_\mathbf{X}^\origin$ converges to some $\tau$ in $\mathscr{M}_\mathbf{X}^\origin$, written as $\tau_n\Rightarrow_\origin\tau$, if and only if
\[\lim_{n\to\infty}\int f\,d\tau_n=\int f\,d\tau,\;\;\;\;\;f\in C_b^\origin(\mathbf{X}).\] We remark that $\tau$ is not unique as it may assign arbitrary mass to the origin. Nevertheless, any weak limit in $\mathscr{M}_\mathbf{X}^\origin$ that comes across in our discussions will serve as the so-called \emph{L\'{e}vy measure}, which does not charge the origin.

\emph{Portmanteau theorem} and \emph{continuous mapping theorem} in the framework of $\mathscr{M}_\mathbf{X}^\origin$ are presented below (see \cite{Portmanteau06,mapping}). Recall that the push-forward measure $\tau h^{-1}\colon\mathscr{B}_{\mathbf{X}'}\to[0,+\infty]$ of  $\tau\in\mathscr{M}_\mathbf{X}^\origin$ provoked by a measurable mapping $h\colon(\mathbf{X},\mathscr{B}_\mathbf{X})\to(\mathbf{X}',\mathscr{B}_{\mathbf{X}'})$
is defined as
\begin{equation} \label{pushforward}
(\tau h^{-1})(B')=\tau(\{\boldsymbol x\in\mathbf{X}:h(\boldsymbol x)\in B'\}),\;\;\;\;\;B'\in\mathscr{B}_{\mathbf{X}'}.
\end{equation}

\begin{prop} \label{Portman}
The following are equivalent for $\{\tau_n\}$ and $\tau$ in $\mathscr{M}_\mathbf{X}^\origin$.
\begin{enumerate} [$\quad(1)$]
\item {$\tau_n\Rightarrow_\origin\tau$;}
\item {for any $f\in C_b(\mathbf{X})$ and any $B\subset\mathscr{B}_\mathbf{X}$,
which is bounded away from the origin and satisfies
$\tau(\partial B)=0$,
\[\lim_{n\to\infty}\int_Bf\,d\tau_n=\int_Bf\,d\tau;\]}
\item {for every closed set $F$ and open set $G$ of $\mathbf{X}$ that are bounded away from
$\origin$, $\limsup_{n\to\infty}\tau_n(F)\leq\tau(F)$ and $\liminf_{n\to\infty}\tau_n(G)\geq\tau(G)$.}
\end{enumerate}
If $h\colon (\mathbf{X},\mathrm{d})\to(\mathbf{X}',\mathrm{d}')$ is measurable
so that $h$ is continuous at $\origin$, $h(\origin)=\boldsymbol{x}'_0$, and the set of discontinuities of $h$ has $\tau$-measure zero, then $\tau_n\Rightarrow_\origin\tau$ implies $\tau_nh^{-1}\Rightarrow_{\boldsymbol{x}'_0}\tau h^{-1}$.
\end{prop}

Finally, let us introduce the subset $\tilde{\mathscr{M}}_\mathbf{X}^\origin$ consisting of measures in $\mathscr{M}_\mathbf{X}^\origin$ that do not charge the origin $\boldsymbol x_0$. This set is metrizable and becomes a separable complete metric space \cite[Theorem 2.2]{mapping}. In particular, the relative compactness of a subset $Y$ of $\tilde{\mathscr{M}}_\mathbf{X}^\origin$ is equivalent to that any sequence of $Y$ has a subsequence convergent in $\tilde{\mathscr{M}}_\mathbf{X}^\origin$. We refer the reader to \cite[Theorem 2.7]{mapping} for an analog of Prokhorov's theorem, which characterizes the relative compactness of subsets in $\tilde{\mathscr{M}}_\mathbf{X}^\origin$.

\subsection{Notations} Below, we collect notations that will be commonly used in the sequel. The customary symbol $\arg s\in(-\pi,\pi]$ stands for the principal argument of a point $s\in\mathbb{T}$, while $\Re s$ and $\Im s$ respectively represent the real and imaginary parts of $s$. Here and elsewhere, points in a multi-dimensional space will be written in bold letters, for instance, ${\boldsymbol s}=(s_1,\ldots,s_d)\in\mathbb{T}^d$ and
$\boldsymbol p=(p_1,\ldots,p_d)\in\mathbb{Z}^d$ with each $s_j\in\mathbb{T}$ and $p_j\in\mathbb{Z}$.
For any $\epsilon>0$, we shall use
\[\mathscr{V}_\epsilon=\{\boldsymbol x\in\mathbb{R}^d:\|\boldsymbol x\|<\epsilon\}\;\;\;\;\;\mathrm{and}\;\;\;\;\;\mathscr{U}_\epsilon=\{{\boldsymbol s}\in\mathbb{T}^d:\|\arg{\boldsymbol s}\|<\epsilon\}\] to respectively express open neighborhoods of origins $\mathbf{0}\in\mathbb{R}^d$ and $\mathbf{1}\in\mathbb{T}^d$, where we use the notation  $\arg{\boldsymbol s}=(\arg s_1,\ldots,\arg s_d)\in\mathbb{R}^d$. Analogous expressions also apply to vectors $\Re\boldsymbol s=(\Re s_1,\dots,\Re s_d)$ and $\Im\boldsymbol s=(\Im s_1,\dots,\Im s_d)$. Besides, we adopt the operational conventions in multi-dimensional spaces in the sequel, such as ${\boldsymbol s}^{\boldsymbol p}=s_1^{p_1} \cdots s_d^{p_d}$, $\boldsymbol s\boldsymbol t=(s_1t_1,\ldots,s_dt_d)$, $\boldsymbol s^{-1}=(1/s_1,\ldots,1/s_d)$, and $e^{i{\boldsymbol s}}=(e^{is_1},\ldots,e^{is_d})$.

The push-forward probabilities $\mu^{(j)}=\mu\pi_j^{-1}$, $j=1,\ldots,d$, on the real line
induced by projections $\pi_j:\mathbb{R}^d\to\mathbb{R}$, $\boldsymbol x\mapsto x_j$, are called \emph{marginals} of $\mu\in\mathscr{P}_{\mathbb{R}^d}$.
Marginals of probability measures on $\mathbb{T}^d$ are defined and displayed in the same way.
On a high-dimensional torus, especial on $\mathbb{T}^2$, we shall also consider the (right) coordinate-flip transform $h_\mathrm{op}:\mathbb{T}^2\to\mathbb{T}^2$ defined as $h_\mathrm{op}(\boldsymbol s)=(s_1,1/s_2)$. Denote by $\boldsymbol s^\star=h_\mathrm{op}(\boldsymbol s)$ and $B^\star=\{\boldsymbol s^\star:\boldsymbol s\in B\}$ if $\boldsymbol s\in\mathbb{T}^2$ and $B\subset\mathbb{T}^2$. By the (right) coordinate-flip measure of $\rho\in\mathscr{M}_{\mathbb{T}^2}^\mathbf{1}$, we mean the push-forward measure $\rho^\star=\rho h_\mathrm{op}^{-1}$, alternatively defined as
\[\rho^\star(B)=\rho(B^\star),\;\;\;\;\;B\in\mathscr{B}_{\mathbb{T}^2}.\]

\subsection{Free probability and bi-free probability} \label{sec2.3}
 Aside from the classical convolution on $\mathscr{P}_{\mathbb{R}^2}$, we shall also consider the bi-free convolution $\biplus$, where the \emph{bi-free $\phi$-transform} takes the place of Fourier transform \cite{HHW}: for $\mu_1,\mu_2\in\mathscr{P}_{\mathbb{R}^2}$,
\[\phi_{\mu_1\biplut\mu_2}=\phi_{\mu_1}+\phi_{\mu_2}.\]
All information about marginals of the bi-free convolution is carried over to the free convolution: $(\mu_1\biplut\mu_2)^{(j)}=\mu_1^{(j)}\boxplus\mu_2^{(j)}$ for $j=1,2$.

Now, we turn to probability measures on the $d$-torus. The sequence
\[m_{\boldsymbol p}(\nu)=\int_{\mathbb{T}^d}\boldsymbol s^{\boldsymbol p}\,d\nu(\boldsymbol s),\;\;\;\;\;\boldsymbol p\in\mathbb{Z}^d,\] is called the \emph{$d$-moment sequence} of $\nu\in\mathscr{P}_{\mathbb{T}^d}$. In some circumstances, \emph{characteristic function} and $\widehat{\nu}(\boldsymbol p)$
are the precise terminology and notation used for this sequence. Owing to Stone-Weierstrass theorem, every measure in $\mathscr{P}_{\mathbb{T}^d}$ is uniquely determined by its $d$-moment sequence, namely, $m_{\boldsymbol p}(\nu)\equiv m_{\boldsymbol p}(\nu')$ only when $\nu=\nu'$. The classical convolution $\circledast$ of distributions on $\mathbb{T}^d$ is characterized by
\[m_{\boldsymbol p}(\nu_1\circledast\nu_2)=m_{\boldsymbol p}(\nu_1)m_{\boldsymbol p}(\nu_2),\;\;\;\;\;\nu_1,\nu_2\in\mathscr{P}_{\mathbb{T}^d}.\]

The \emph{bi-free multiplicative convolution} of $\nu_1,\nu_2\in\mathscr{P}_{\mathbb{T}^2}^\times$ is determined by its marginals ${(\nu_1\bitimet \nu_2)^{(j)}}=\nu_1^{(j)}\boxtimes\nu_2^{(j)}$ and the bi-free multiplicative formula
\[\Sigma_{\nu_1\bitimet\nu_2}(z,w)=\Sigma_{\nu_1}(z,w)\cdot\Sigma_{\nu_2}(z,w)\] for points $(z,w)\in\mathbb{C}^2$ in a neighborhood of $(0,0)$ and $(0,\infty)$.
Here the free multiplicative convolution can be rephrased by means of
the free $\Sigma$-transform $\Sigma_{\nu_1^{(j)}\boxtimes\nu_2^{(j)}}=\Sigma_{\nu_1^{(j)}}\cdot\Sigma_{\nu_2^{(j)}}$ valid in a neighborhood of the origin of the complex plane. The reader is referred to \cite{BerVoicu92,BerVoicu93,HW1,HW2,Sko16,Sko18,V15,V15ST} for more details along with properties of the transforms in (bi)-free probability theory.

Fix $\nu_1,\nu_2\in\mathscr{P}_{\mathbb{T}^2}$, and let $\nu=\nu_1\bitimet\nu_2$. In order to analyze $\nu$ in-depth here and elsewhere, it will be convenient to treat it as the distribution of a certain bipartite pair $(u_1u_2,v_1v_2)$, where
$(u_1,v_1)$ and $(u_2,v_2)$ are bi-free bipartite unitary pairs in some $C^*$-probability space having distributions $\nu_1$ and $\nu_2$,
respectively. Below, we briefly introduce the construction of such pairs carrying the mentioned properties. For more information, we refer the reader to \cite{HW2,V14,V15ST}.

Associating each $\nu_j$ with the Hilbert space $\mathcal{H}_j=L^2(\nu_j)$ with specified unit vector $\xi_j$, the constant function one in $\mathcal{H}_j$, one considers the Hilbert space free product $(\mathcal{H},\xi)=*_{j=1,2}(\mathcal{H}_j,\xi_j)$. The left
and right factorizations of $\mathcal{H}_j$ from $\mathcal{H}$ can be respectively done via isomorphisms $V_j:\mathcal{H}_j\otimes\mathcal{H}(\ell,j)\to\mathcal{H}$ and $W_j:\mathcal{H}(r,j)\otimes\mathcal{H}_j\to\mathcal{H}$. Then for any $T\in B(\mathcal{H}_j)$, these isomorphisms induce the so-called \emph{left} and \emph{right operators} $\lambda_j(T)=V_j(T\otimes I_{\mathcal{H}(\ell,j)})V_j^{-1}$ and $\rho_j(T)=W_j(I_{\mathcal{H}(r,j)}\otimes T)W_j^{-1}$ on $\mathcal{H}$. For any $S_j,T_j\in B(\mathcal{H}_j)$, pairs $(\lambda_1(S_1),\rho_1(T_1))$ and $(\lambda_2(S_2),\rho_2(T_2))$ are, by definition, bi-free in the $C^*$-probability space $(B(\mathcal{H}),\varphi_\xi)$, where $\varphi_\xi(\cdot)=\langle\cdot\xi,\xi\rangle$. Particularly, the multiplication operators $(S_jf)(s,t)=sf(s,t)$ and $(T_jf)(s,t)=tf(s,t)$ for $f\in\mathcal{H}_j$ furnish the desired pairs $(u_1,v_1)$ and $(u_2,v_2)$, where $u_j=\lambda_j(S_j)$ and $v_j=\rho_j(T_j)$.

Recall from \cite{HW2} that one can perform the \emph{opposite bi-free multiplicative convolution} of $\nu_1$ and $\nu_2$:
\begin{equation} \label{opdef}
\nu_1\bitimet^{\mathrm{op}}\nu_2=(\nu_1^\star\bitimet\nu_2^\star)^\star.
\end{equation} It turns out that $\nu_1\bitimet^{\mathrm{op}}\nu_2$ is the distribution of $(u_1u_2,v_2v_1)$, which is the pair obtained by performing the opposite multiplication on the right face $(u_1,v_1)\cdot^{\mathrm{op}}(u_2,v_2)=(u_1u_2,v_2v_1)$. The coordinate-flip map $h_\mathrm{op}$ gives rise to a homeomorphism
from the semigroup $(\mathscr{P}_{\mathbb{T}^2},\bitimes)$ to another $(\mathscr{P}_{\mathbb{T}^2},\bitimes^{\mathrm{op}})$ satisfying \[(\nu_1\bitimet\nu_2)h_\mathrm{op}^{-1}=(\nu_1h_\mathrm{op}^{-1})
\bitimet^{\mathrm{op}}(\nu_2h_\mathrm{op}^{-1}),\] which is the distribution of
\[h_\mathrm{op}((u_1,v_1)(u_2,v_2))=(u_1u_2,v_2^{-1}v_1^{-1})=h_\mathrm{op}((u_1,v_1))\cdot^{\mathrm{op}}
h_\mathrm{op}((u_2,v_2)).\]
Passing to the transform
\[\Sigma_\nu^{\mathrm{op}}(z,w)=\Sigma_{\nu^{\star}}(z,1/w),\] the equation (\ref{opdef}) is translated into
\[\Sigma_{\nu_1\bitimes^{\mathrm{op}}\nu_2}^{\mathrm{op}}(z,w)
=\Sigma_{\nu_1}^{\mathrm{op}}(z,w)\cdot\Sigma_{\nu_1}^{\mathrm{op}}(z,w).\]

\subsection{Limit Theorem} \label{LTcond}
Either in classical or in (bi-)free probability theory, one is concerned with the asymptotic behavior of the sequence
\begin{equation} \label{asymptotic}
\delta_{\boldsymbol x_n}\lozenge\mu_{n1}\lozenge\cdots\lozenge\mu_{nk_n},\;\;\;\;\;n=1,2,\ldots,
\end{equation}
where $\delta_{\boldsymbol x}$ is the Dirac measure concentrated at $\boldsymbol x\in\mathbf{X}$ and
$\{\mu_{nk_n}\}_{n\geq1,1\leq k\leq k_n}$ is an infinitesimal triangular array in $\mathscr{P}_\mathbf{X}$. The infinitesimality of $\{\mu_{nk}\}$, by definition, means that $k_1<k_2<\cdots$ and that for any $\epsilon>0$, we have
\[\lim_{n\to\infty}\max_{1\leq k\leq k_n}\mu_{nk}(\{\boldsymbol x\in\mathbf{X}:\mathrm{d}(\boldsymbol x,\boldsymbol x_0)\geq\epsilon\})=0.\]
One phenomenon related to (\ref{asymptotic}) is the concept of infinite divisibility: $\mu\in(\mathscr{P}_\mathbf{X},\lozenge)$ is said to be infinitely divisible if for any $n\in\mathbb{N}$, it coincides with the $n$-fold $\lozenge$-operation $\mu_n^{\lozenge n}$ of some $\mu_n\in\mathscr{P}_\mathbf{X}$.

Commutative and associative binary operations to be considered throughout the paper are classical convolutions $*$ and $\circledast$ on $\mathscr{P}_{\mathbb{R}^d}$ and
$\mathscr{P}_{\mathbb{T}^d}$, respectively, and bi-free additive and multiplicative convolutions $\biplus$ and $\bitimes$ on $\mathscr{P}_{\mathbb{R}^2}$ and $\mathscr{P}_{\mathbb{T}^2}$, respectively.

The following convergence criteria play an essential role in the asymptotic analysis of limit theorems of  $\mathscr{P}_{\mathbb{R}^d}$.

\begin{cond} \label{cond+1}
Let $\{\tau_n\}$ be a sequence in $\mathscr{M}_{\mathbb{R}^d}^\mathbf{0}$.

\begin{enumerate}[\rm(I)]
\item\label{item:condI} For $j=1,\ldots,d$, the sequence $\{\sigma_{nj}\}_{n\geq1}$ defined as
\[d\sigma_{nj}(\boldsymbol x)=\frac{x_j^2}{1+x_j^2}\,d\tau_n(\boldsymbol x)\] belongs to $\mathscr{M}_{\mathbb{R}^d}$ and converges weakly to some $\sigma_j\in\mathscr{M}_{\mathbb{R}^d}$.

\item\label{item:condII} For $j,\ell=1,\ldots,d$, the following limit exists in $\mathbb{R}$:
\[L_{j\ell}=\lim_{n\to\infty}\int_{\mathbb{R}^2}\frac{x_jx_\ell}{(1+x_j^2)(1+x_\ell^2)}\,d\tau_n(\boldsymbol x).\]
\end{enumerate}
\end{cond}

\begin{cond} \label{cond+2}
Let $\{\tau_n\}$ be a sequence in $\mathscr{M}_{\mathbb{R}^d}^\mathbf{0}$.

\begin{enumerate}[\rm(I)]
\setcounter{enumi}{2}

\item\label{item:condIII} There is some $\tau\in\mathscr{M}_{\mathbb{R}^d}^\mathbf{0}$ with $\tau(\{\mathbf{0}\})=0$ so that $\tau_n\Rightarrow_\mathbf{0}\tau$. \\

\item\label{item:condIV} For any vector $\boldsymbol u\in\mathbb{R}^d$, the following limits exist in $\mathbb{R}$:
\[\lim_{\epsilon\to0}\limsup_{n\to\infty}\int_{\mathscr{V}_\epsilon}\langle\boldsymbol u,\boldsymbol x\rangle^2\,d\tau_n(\boldsymbol x)=Q(\boldsymbol u)=
\lim_{\epsilon\to0}\liminf_{n\to\infty}\int_{\mathscr{V}_\epsilon}\langle\boldsymbol u,\boldsymbol x\rangle^2\,d\tau_n(\boldsymbol x).\]
\end{enumerate}
\end{cond}

Although we describe the conditions in a higher dimension setup, the reader can effortlessly mimic the proof in \cite{HHW} to obtain the equivalence of Condition \ref{cond+1} and Condition \ref{cond+2}, and draw the following consequences:
\begin{enumerate}[\rm\hspace{5mm}(1)]

\item The function $Q(\cdot)=\langle\mathbf{A}\cdot,\cdot\rangle$ in \eqref{item:condIV}
defines a non-negative quadratic form on $\mathbb{R}^d$, where the matrix $\mathbf{A}=(a_{j\ell})$ is given by
\[a_{j\ell}=L_{j\ell}-\int_{\mathbb{R}^d}\frac{x_jx_\ell}{(1+x_j^2)(1+x_\ell^2)}\,d\tau(\boldsymbol x),\;\;\;\;\;j,\ell=1,\ldots,d.\] In particular, $a_{jj}=\sigma_j(\{\mathbf{0}\})$ for $j=1,\ldots,d$.

\item Measures $\tau$ and $\sigma_1,\ldots,\sigma_d$ are uniquely determined by the relations
\[d\sigma_j(\boldsymbol x)=\frac{x_j^2}{1+x_j^2}\,d\tau(\boldsymbol x)+Q(\boldsymbol e_j)\delta_\mathbf{0}(d\boldsymbol x),\] where $\{\boldsymbol e_j\}$ is the standard basis of $\mathbb{R}^d$.
\item The function $\boldsymbol x\mapsto\min\{1,\|\boldsymbol x\|^2\}$ is $\tau$-integrable.
\end{enumerate}

Now, let us briefly introduce the limit theorem. Throughout our discussions in the paper,
\begin{equation} \label{thetarange}
\theta\in(0,1)
\end{equation}
is an arbitrary but fixed quantity. To meet the purpose, consider the shifted triangular array \[\mathring\mu_{nk}(B)=\mu_{nk}(B+{\boldsymbol v}_{nk}),\;\;\;\;\;B\in\mathscr{B}_{\mathbb{R}^d},\] associated with an infinitesimal triangular array $\{\mu_{nk}\}_{n\geq1,1\leq k\leq k_n}\subset\mathscr{P}_{\mathbb{R}^d}$ and the vector
\[{\boldsymbol v}_{nk}=\int_{\mathscr{V}_\theta}\boldsymbol x\;d\mu_{nk}(\boldsymbol x).\] Due to $\lim_{n\to\infty}\max_{k\leq k_n}\|{\boldsymbol v}_{nk}\|=0$, $\{\mathring\mu_{nk}\}$ so obtained is also infinitesimal. In conjunction with this centered triangular array, we focus on the positive measures
\begin{equation} \label{taundef}
\tau_n=\sum_{k=1}^{k_n}\mathring\mu_{nk}.
\end{equation}

It turns out that the sequence in (\ref{munclass}) converges weakly to a certain $\mu_*\in\mathscr{P}_{\mathbb{R}^d}$ if and only if $\tau_n$ defined in (\ref{taundef}) meets Condition \ref{cond+1} (or, equivalently, Condition \ref{cond+2}) and
the limit
\begin{equation} \label{limitv}
\boldsymbol v=\lim_{n\to\infty}\left[\boldsymbol v_n+\sum_{k=1}^{k_n}\left({\boldsymbol v}_{nk}
+\int_{\mathbb{R}^d}\frac{\boldsymbol x}{1+\|\boldsymbol x\|^2}\,d\mathring\mu_{nk}(\boldsymbol x)\right)\right]
\end{equation}
exists in $\mathbb{R}^d$. Additionally, $\mu_*$ is $*$-infinitely divisible and possesses the characteristic function $\widehat{\mu}_*(\boldsymbol u)$ read as
\begin{equation} \label{LKrepre}
\exp\left[i\langle\boldsymbol u,\boldsymbol v\rangle-\frac{1}{2}\langle\mathbf{A}\boldsymbol u,\boldsymbol u\rangle+\int_{\mathbb{R}^d}\left(e^{i\langle\boldsymbol u,\boldsymbol x\rangle}-1-\frac{i\langle\boldsymbol u,\boldsymbol x\rangle}{1+\|\boldsymbol x\|^2}\right)d\tau(\boldsymbol x)\right],
\end{equation} which is known as the
\emph{L\'{e}vy-Khintchine representation}. The limiting distribution is uniquely determined by this formula and denoted by $\mu_*^{(\boldsymbol v,\mathbf{A},\tau)}$, and $(\boldsymbol v,\mathbf{A},\tau)$ is referred to as its \emph{L\'{e}vy triplet}. Those triplets $(\boldsymbol v,\mathbf{A},\tau)$, where
\begin{equation}\label{Levy_triplet}
\begin{split}
&\text{$\boldsymbol v\in\mathbb R^d$, $\mathbf{A}\in M_d(\mathbb{R})$ is positive semi-definite, and $\tau$ is a positive measure} \\
&\text{on $\mathbb{R}^d$ satisfying $\tau(\{\mathbf0\})=0$ and $\min\{1,\|\boldsymbol x\|^2\} \in L^1(\tau)$,}
\end{split}
\end{equation}
completely parameterize $\mathcal{ID}(\ast)$.

As a matter of fact, when $d=2$, the same convergence criteria are also necessary and sufficient to assure the weak convergence of (\ref{munbifree}). Paralleling to the classical case, the limiting distribution of (\ref{munbifree}) is $\biplus$-infinitely divisible and owns the bi-free $\phi$-transform, called \emph{bi-free L\'{e}vy-Khintchine representation}, of the form
\[\phi(z,w)=\frac{v_1}{z}+\frac{v_2}{w}+\left(\frac{a_{11}}{z^2}+\frac{a_{12}}{zw}+\frac{a_{22}}{w^2}\right)+
\int_{\mathbb{R}^2}\left[\frac{zw}{(z-x_1)(w-x_2)}-1 - \frac{x_1z^{-1} +x_2 w^{-1}}{1+\|\boldsymbol x\|^2}\right]d\tau(\boldsymbol x).\]
Analogically, this limiting distribution is always expressed as $\mu_{\biplus}^{(\boldsymbol v,\mathbf{A},\tau)}$ and said to own the \emph{bi-free L\'{e}vy triplet} $(\boldsymbol v,\mathbf{A},\tau)$. The bi-freely infinitely divisible distribution $\mu_{\biplus}^{(\boldsymbol0,\mathbf{I},0)}$
is known as the standard bi-free Gaussian distribution \cite{HW1}.

Those triplets $(\boldsymbol v,\mathbf{A},\tau)$ satisfying \eqref{Levy_triplet} give a complete parametrization of the set $\mathcal {ID}(\biplus)$ as well as of $\mathcal{ID}(\ast)$, and therefore output a bijective homomorphism $\Lambda$ from $\mathcal{ID}(*)$ onto $\mathcal{ID}(\biplus)$, sending an element $\mu_*^{(\boldsymbol v,\mathbf{A},\tau)}$ in the first set to the distribution $\mu_{\biplus}^{(\boldsymbol v,\mathbf{A},\tau)}$ in the second one. No matter in the classical or bi-free probability, $*$- and $\biplus$-infinitely divisible distributions both appear as limiting distributions in the limit theorem.

Next, we turn our attention to the limit theorem on the $d$-torus, on which the
Borel probability measures of interest are sometimes imposed the mean conditions:
\[\int_{\mathbb{T}^d}s_j\;d\nu({\boldsymbol s})\neq0,\;\;\;\;\;j=1,\ldots,d.\]
For convenience, we adopt the symbol $\mathscr{P}_{\mathbb{T}^d}^\times$ to signify the collection of probability measures carrying such features. As will be shown in Theorem \ref{IDmoment}, when $d=2$, these conditions turn out to be necessary and sufficient for a $\bitimes$-infinitely divisible distribution to contain no non-trivial $\bitimes$-idempotent factor. We would also like to remind the reader that the symbol $\mathscr{P}_{\mathbb{T}^2}^\times$ introduced here is distinct from that in \cite{HW2} as Theorem \ref{characterID} of the present paper designates that the requirement $m_{1,1}(\nu)\neq0$ in the limit theorem is redundant.

Given an infinitesimal triangular array $\{\nu_{nk}\}_{n\geq1,1\leq k\leq k_n}$ in $\mathscr{P}_{\mathbb{T}^d}$, one works with
the rotated probability measures
\[d\mathring\nu_{nk}({\boldsymbol s})=d\nu_{nk}({\boldsymbol b}_{nk}{\boldsymbol s}),\] where
\begin{equation} \label{bnk}
\boldsymbol b_{nk}=\exp\left[i\int_{\mathscr{U}_\theta}(\arg{\boldsymbol s})\,d\nu_{nk}({\boldsymbol s})\right].
\end{equation} Once again, $\{\mathring\nu_{nk}\}$ is infinitesimal
because of $\lim_n\max_k\|\arg{\boldsymbol b}_{nk}\|=0$. Given a sequence $\{\boldsymbol\xi_n\}\subset\mathbb{T}^d$, further define vectors in $\mathbb{T}^d$:
\begin{equation} \label{gamman}
\boldsymbol\gamma_n=\boldsymbol\xi_n\exp\left[i\sum_{k=1}^{k_n}\left(\arg{\boldsymbol b}_{nk}+\int_{\mathbb{T}^2}(\Im\boldsymbol s)
\,d\mathring\nu_{nk}({\boldsymbol s})\right)\right].
\end{equation}

The bi-free multiplicative limit theorem on the bi-torus has been shown in \cite[Theorem 3.4]{HW2}:

\begin{thm} \label{limitthmX}
The necessary and sufficient condition for
the sequence \emph{(\ref{bifreemul})} to converge weakly to a certain $\nu_{\boxtimes\boxtimes}\in\mathscr{P}_{\mathbb{T}^2}^\times$ is that the limit
\begin{equation} \label{gamma}
\lim_{n\to\infty}\boldsymbol\gamma_n=\boldsymbol\gamma
\end{equation} exists and the following positive measures satisfy \emph{Conditions \ref{condX1}} for $d=2$:
\begin{equation} \label{rhon}
\rho_n=\sum_{k=1}^{k_n}\mathring\nu_{nk}.
\end{equation}
\end{thm}

\begin{cond} \label{condX1}
Let $\{\rho_n\}$ be a sequence in $\mathscr{M}_{\mathbb{T}^d}^\mathbf{1}$.

\begin{enumerate}[(i)]
\item\label{item:condi} For $j=1,\ldots,d$, the sequence $\{\lambda_{nj}\}_{n\geq1}$ defined as
\[d\lambda_{nj}({\boldsymbol s})=(1-\Re s_j)\,d\rho_n({\boldsymbol s})\] belongs to $\mathscr{M}_{\mathbb{T}^d}$ and converges weakly to some $\lambda_j\in\mathscr{M}_{\mathbb{T}^d}$.

\item\label{item:condii} For $1\leq j,\ell\leq d$, the following limit exists in $\mathbb{R}$:
\[L_{j\ell}=\lim_{n\to\infty}\int_{\mathbb{T}^d}(\Im s_j)(\Im s_\ell)\,d\rho_n({\boldsymbol s}).\]
\end{enumerate}
\end{cond}

The limiting distribution $\nu=\nu_\bitimes$ in Theorem \ref{limitthmX} is $\bitimes$-infinitely divisible, as expected, and uniquely determined by the formulas
\begin{equation} \label{bi-freeLK}
\Sigma_{\nu^{(j)}}(\xi)=\exp[u_j(\xi)]\;\;\;\;\;\mathrm{and}\;\;\;\;\;
\Sigma_\nu(z,w)=\exp[u(z,w)],
\end{equation}
where
\[
u_j(\xi)=-i\arg\gamma_j+\int_{\mathbb{T}^2}\frac{1+\xi s_j}{1-\xi s_j}\,d\lambda_j({\boldsymbol s}),\;\;\;\;\;\xi\in\mathbb{D},\]
for $j=1,2$, and for $(z,w)\in(\mathbb{C}\backslash\mathbb{D})^2$,
\begin{align*}
\frac{(1-z)(1-w)}{1-zw}&u(z,w)=\int_{\mathbb{T}^2}\frac{1+zs_1}{1-zs_1}\frac{1+ws_2}{1-ws_2}(1-\Re s_2)\,d\lambda_1({\boldsymbol s}) \\
&-i\int_{\mathbb{T}^2}\frac{1+zs_1}{1-zs_1}(\Im s_2)\,d\lambda_1({\boldsymbol s})-
i\int_{\mathbb{T}^2}\frac{1+ws_2}{1-ws_2}(\Im s_1)\,d\lambda_2({\boldsymbol s})-L_{12}.
\end{align*}
In turn, any measure in $\mathcal{ID}(\bitimes)\cap\mathscr{P}_{\mathbb{T}^2}^\times$ truly arises as a weak-limit point of (\ref{bifreemul}).

\begin{remark} \label{urepre}
Suppose that $\nu\in\mathcal{ID}(\bitimes)\backslash\mathscr{P}_{\mathbb{T}^2}^\times$, and let $m_j=\int s_jd\nu^{(j)}$ for $j=1,2$. Then $\Sigma_{\nu^{(j)}}(0)=1/m_j$, $\arg\gamma_j=\arg m_j$, and $\lambda_j(\mathbb{T}^2)=-\log|m_j|\in[0,\infty)$. We remind the reader that the parameter $\gamma_j$ in $u_j(\xi)$
and that appearing in \cite{HW2} are conjugate complex numbers. With the help of the equation
\begin{equation} \label{decomeq}
\frac{1+\xi s}{1-\xi s}(1-\Re s)=i\Im s+\frac{(1-\xi)(1-s)}{1-\xi s},\;\;\;\;(\xi,s)\in\mathbb{D}\times\mathbb{T},
\end{equation} one can see that
\[u_j(\xi)=-i\arg\gamma_j+\lim_{n\to\infty}\int_{\mathbb{T}^2}\frac{1+\xi s_j}{1-\xi s_j}(1-\Re s_j)\,d\rho_n({\boldsymbol s})\] and
\[u(z,w)=\lim_{n\to\infty}\int_{\mathbb{T}^2}\frac{(1-zw)(1-s_1)(1-s_2)}{(1-zs_1)(1-ws_2)}\,d\rho_n({\boldsymbol s})\] for some sequence $\{\rho_n\}\subset\mathscr{M}_{\mathbb{T}^2}^\mathbf{1}$ satisfying Condition \ref{condX1}.
\end{remark}

\section{$\bitimes$-Idempotent Distributions} \label{sec3}
Let $\mu\in\mathscr{P}_\mathbf{X}$. A measure $\mu'\in\mathscr{P}_\mathbf{X}$ is called a $\lozenge$-factor of $\mu$ if $\mu=\mu'\lozenge\mu''$ for some $\mu''\in\mathscr{P}_\mathbf{X}$. Particularly, $\mu$ is said to be $\lozenge$-idempotent when $\mu'=\mu=\mu''$. Idempotent distributions and other related subjects in classical probability have been extensively studied \cite{Par67}. It is to questions of these sorts in the bi-free probability theory that the present section is devoted.

The normalized Lebesgue measure $\mathrm{m}=d\theta/(2\pi)$ on $\mathbb{T}$ is the only $\boxtimes$-idempotent element except for the trivial one, the Dirac measure at $1$. On $\mathbb{T}^2$, the probability measure
\[P(B)=\mathrm{m}\big(\{s\in\mathbb{T}:(s,\bar{s})\in B\}\big),\;\;\;\;\;B\in\mathscr{B}_{\mathbb{T}^2},\] is $\circledast$-idempotent because $m_{p,q}(P)=1$ for $p=q\in\mathbb{Z}$ and zero otherwise. As a matter of fact, this singularly continuous measure is also $\bitimes$-idempotent:

\begin{prop} \label{idempotnet1}
A $\bitimes$-idempotent distribution in $\mathscr{P}_{\mathbb{T}^2}$ is one of five types $\delta_{(1,1)}$, $\mathrm{m}\times\delta_1$, $\delta_1\times\mathrm{m}$, $\mathrm{m}\times\mathrm{m}$, and $P$. A measure in $\mathscr{P}_{\mathbb{T}^2}$ is $\bitimes^{\mathrm{op}}$-idempotent if and only if it is $\delta_{(1,1)}$, $\mathrm{m}\times\delta_1$, $\delta_1\times\mathrm{m}$, $\mathrm{m}\times\mathrm{m}$, or $P^\star$.
\end{prop}

\begin{proof} Suppose that $\nu$ is $\bitimes$-idempotent. Since each marginal satisfies $\nu^{(j)}=\nu^{(j)}\boxtimes\nu^{(j)}$, it follows that $\nu^{(j)}$ is $\boxtimes$-infinitely divisible. If $\nu^{(j)}$ has non-zero mean, then $\Sigma_{\nu^{(j)}}(0)=1$, which yields $\nu^{(j)}=\delta_1$ by \cite[Lemma 2.7]{BerVoicu92}. Otherwise, we can infer from \cite[Lemma 6.1]{BerVoicu92} that $\nu^{(j)}=\mathrm{m}$. Thus, consideration given to the case $\nu^{(1)}=\mathrm{m}=\nu^{(2)}$ is sufficient to complete the proof.
To continue the proof, we realize $\nu=\nu_1\bitimet\nu_2$ as the distribution of $(u,v)=(u_1u_2,v_1v_2)$, where $(u_j,v_j)=(\lambda_j(S_j),\rho_j(T_j))$, $j=1,2$, are bi-free unitary faces respectively having distributions $\nu_j=\nu$ in the $C^*$-probability space $(B(\mathcal{H}),\varphi_\xi)$, as constructed in Section \ref{sec2.3}.

From $\varphi_\xi(u_j)=0$ for $j=1,2$, it follows that $S_j^{\pm1}\xi_j\in\mathcal{\mathring H}_j=\mathcal{H}_j\ominus\mathbb{C}\xi_j$, which supplies a simplistic representation for $u^p\xi$ for any $p\in\mathbb{N}$, namely,
\begin{equation} \label{up}u^p\xi=\big((S_1\xi_1)\otimes(S_2\xi_2)\big)^{\otimes p}\;\;\;\;\;\mathrm{and}\;\;\;\;\;
u^{-p}\xi=\big((S_2^{-1}\xi_2)\otimes(S_1^{-1}\xi_1)\big)^{\otimes p}
\end{equation} respectively lying in spaces $(\mathcal{\mathring H}_1\otimes\mathcal{\mathring H}_2)^{\otimes p}$ and $(\mathcal{\mathring H}_2\otimes\mathcal{\mathring H}_1)^{\otimes p}$.
Similarly, $\varphi_\xi(v_1)=0=\varphi_\xi(v_2)$ implies that
\begin{equation} \label{vp}
v^q\xi=\big((T_2\xi_2)\otimes(T_1\xi_1)\big)^{\otimes q}\in(\mathcal{\mathring H}_2\otimes\mathcal{\mathring H}_1)^{\otimes q},\;\;\;\;\;q\in\mathbb{N}.
\end{equation} We consequently arrive at that for $(p,q)\in(\mathbb{Z}\backslash\{0\})\times(\mathbb N \cup\{0\})$,
\[m_{p,q}(\nu)=\varphi_\xi(u^pv^q)=\langle v^q\xi,u^{-p}\xi\rangle=\delta_{p,q}\big[\varphi_\xi(u_1v_1)\varphi_\xi(u_2v_2)\big]^p=
\delta_{p,q}m_{1,1}(\nu)^{2p},\] and that $m_{0,q}(\nu)=\varphi_\xi(v^q)=\delta_{0,q}$ for $q\in\mathbb N \cup\{0\}$. If $m_{1,1}(\nu)=0$,
then $m_{p,q}(\nu)=0$ for any $(p,q)\in\mathbb{Z}^2\backslash\{(0,0)\}$, which occurs only when $\nu=\mathrm{m}\times\mathrm{m}$. If $m_{1,1}(\nu)\neq0$, then the equation $m_{1,1}(\nu)=m_{1,1}(\nu)^2$ results in $m_{1,1}(\nu)=1$, yielding $\nu=P$ as they have a common $2$-moment sequence.

The $\bitimes^{\mathrm{op}}$-idempotent elements can be easily ascertained by formula (\ref{opdef}) and established results.
This finishes the proof.
\end{proof}

It is known that for any $\nu_1,\nu_2\in\mathscr{P}_{\mathbb{T}^2}$, $m_{p,q}(\nu_1\bitimet\nu_2)=m_{p,q}(\nu_1)m_{p,q}(\nu_2)$ holds when $(p,q)=(0,1),(1,0)$. We also have:

\begin{lem} \label{basiclem}
For any $\nu_1,\nu_2\in\mathscr{P}_{\mathbb{T}^2}$, it holds that
\[m_{1,1}(\nu_1\bitimet\nu_2)=m_{1,1}(\nu_1)m_{1,1}(\nu_2) \quad\text{and}\quad m_{1,-1}(\nu_1\bitimet^{\mathrm{op}}\nu_2)=m_{1,-1}(\nu_1)m_{1,-1}(\nu_2).
\]
\end{lem}

\begin{proof} Following the notations in Section \ref{sec2.3}, let
$\alpha_j=\langle S_j^{-1}\xi_j,\xi_j\rangle$, $\beta_j=\langle T_j\xi_j,\xi_j\rangle$, $h_j=S_j^{-1}\xi_j-\alpha_j\xi_j$, and $k_j=T_j\xi_j-\beta_j\xi_j$ for $j=1,2$.
Then
\[m_{1,1}(\nu_j)=\langle T_j\xi_j,S_j^{-1}\xi_j\rangle=\langle \beta_j\xi_j+k_j,\alpha_j\xi_j+h_j\rangle=\overline{\alpha_j}\beta_j+\langle k_j,h_j\rangle.\] On the other hand, we have $u_2^{-1}u_1^{-1}\xi=\alpha_1\alpha_2\xi+\alpha_2h_1+\alpha_1h_2+h_2\otimes h_1$ and $v_1v_2\xi=\beta_1\beta_2\xi+\beta_2k_1+\beta_1k_2+k_2\otimes k_1$.
Thus, the first desired result follows from the representation of $m_{1,1}(\nu_j)$ given above and the following computations
\begin{align*}
m_{1,1}(\nu_1\bitimet\nu_2)&=\langle v_1v_2\xi,u_2^{-1}u_1^{-1}\xi\rangle \\
&=
\overline{\alpha_1\alpha_2}\beta_1\beta_2+\overline{\alpha_2}\beta_2\langle k_1,h_1\rangle+\overline{\alpha_1}\beta_1\langle k_2,h_2\rangle+\langle k_1,h_1\rangle\langle k_2,h_2\rangle.
\end{align*}

Thanks to (\ref{opdef}) and the first conclusion, we have
\[
m_{1,-1}(\nu_1\bitimes^{\mathrm{op}}\nu_2)
=m_{1,1}(\nu_1^\star\bitimet\nu_2^\star)=m_{1,1}(\nu_1^\star)m_{1,1}(\nu_2^\star)
=m_{1,-1}(\nu_1)m_{1,-1}(\nu_2),
\]
 as desired.
\end{proof}

In the  sequel, except for $\delta_{(1,1)}$, the other four $\bitimes$-idempotent distributions are called non-trivial. The abusing notation $0^0=1$ is used in the following proposition and elsewhere.

\begin{prop} \label{idempotnet2}
Let $\nu\in\mathscr{P}_{\mathbb{T}^2}$.
\begin{enumerate} [\hspace{5mm}\rm(1)]
\item\label{item:idem1} {$\nu$ has the $\bitimes$-factor $\mathrm{m}\times\delta_1$ if and only if
$\nu=\mathrm{m}\times\nu^{(2)}$.}

\item {$\nu$ has the $\bitimes$-factor $\delta_1\times\mathrm{m}$ if and only if
$\nu=\nu^{(1)}\times\mathrm{m}$.}

\item\label{item:idem3} {$\nu$ has the $\bitimes$-factor $\mathrm{m}\times\mathrm{m}$
if and only if $\nu=\mathrm{m}\times\mathrm{m}$.}

\item {$P$ is a $\bitimes$-factor of $\nu$ if and only if
\begin{equation} \label{mixedm}
m_{p,q}(\nu)=\delta_{p,q}m_{1,1}(\nu)^p,\;\;\;\;\;(p,q)\in\mathbb{Z} \times (\mathbb N \cup\{0\}),
\end{equation} where $\delta_{p,q}$ is the Kronecker function of $p$ and $q$.}
\end{enumerate}
Statements \eqref{item:idem1}--\eqref{item:idem3} remain true if the convolution $\bitimes$ is replaced with $\bitimes^{\mathrm{op}}$. Moreover, $P^\star$ is a $\bitimes^{\mathrm{op}}$-factor of $\nu$ if and only if
\begin{equation} \label{opmixedm}
m_{p,q}(\nu)=\delta_{p,-q}m_{1,-1}(\nu)^p,\;\;\;\;\;(p,q)\in\mathbb{Z} \times (-\mathbb N \cup\{0\}).
\end{equation}
\end{prop}

\begin{remark}
For negative integers $q$, the formula \eqref{mixedm} becomes $m_{p,q}(\nu)=\delta_{p,q}m_{-1,-1}(\nu)^{-p}$, which can be obtained by taking complex conjugate.
\end{remark}

\begin{proof} Write $\nu=\nu_1\bitimet\nu_2$, where neither $\nu_1$ nor $\nu_2$ is $\delta_{(1,1)}$. We shall stay employing the notations for $\nu_1,\nu_2$ in Section \ref{sec2.3} to accomplish the proof.

First, assume that $\nu_2=\mathrm{m}\times\delta_1$. In order to obtain $\nu=\mathrm{m}\times\nu^{(2)}$, it amounts to proving that $m_{p,q}(\nu)=0$ for any $p\in\mathbb{Z}^\times=\mathbb{Z}\backslash\{0\}$ and $q\in\mathbb{N}\cup\{0\}$ by the uniqueness of probability measures on the bi-torus.

To this end, denote $\mathcal{\mathring H}_j=\mathcal{H}_j\ominus\mathbb{C}\xi_j$, $g=S_1\xi_1$, $\alpha=\langle g,\xi_1\rangle$, $\mathring g=g-\alpha\xi_1\in\mathcal{\mathring H}_1$, and
$g_p=S_2^p\xi_2$ for $p\in\mathbb{Z}^\times$. Below, we record some useful observations regarding the moments of the left face $u=u_1u_2$ for the later use. First of all, observe that the set $Z=\{g_p:p\in\mathbb{Z}^\times\}$ is a subset of $\mathcal{\mathring H}_2$ as any non-zero order moment of $\mathrm{m}$ vanishes. Hence for any $p>0$, we have
\begin{align*}
ug_p&=u_1V_2(S_2\otimes I_{\mathcal{H}(\ell,2)})V_2^{-1}g_p \\
&=u_1V_2(S_2\otimes I_{\mathcal{H}(\ell,2)})g_p\otimes\xi \\
&=u_1g_{p+1} \\
&=V_1(S_1\otimes I_{\mathcal{H}(\ell,1)})\xi_1\otimes g_{p+1}=V_1g\otimes g_{p+1}=\alpha g_{p+1}+\mathring g\otimes g_{p+1}.
\end{align*} This demonstrates the inclusion $uZ^+\subset\mathrm{span}(Z^+\cup Z_1^+)$, where $uZ^+$ is the action of the operator $u$ on the set $Z^+=\{g_p:p>0\}$ and
\[Z_1^+=\{\mathring g\otimes z_1\otimes\mathring g\otimes z_2\otimes\cdots\otimes\mathring g\otimes z_n:z_j\in Z^+\;\mathrm{and}\;n\in\mathbb{N}\}.\]
One can also establish the inclusion $uZ_1^+\subset\mathrm{span}(Z_1^+\cup Z_2^+)$, where
\[Z_2^+=\{z\otimes h:z\in Z^+\,\,\mathrm{and}\,\,h\in Z_1^+\},\] by observing that
for any $h\in Z_1^+$,
\begin{align*}
uh=u_1V_2(S_2\otimes I_{\mathcal{H}(\ell,2)})\xi_2\otimes h=u_1g_1\otimes h=\alpha g_1\otimes h+\mathring g\otimes g_1\otimes h.
\end{align*}
For any $z=g_p\in Z^+$ and $h\in Z_1^+$, the identity
\[u(z\otimes h)=u_1g_{p+1}\otimes h=\alpha g_{p+1}\otimes h+\mathring g\otimes g_{p+1}\otimes h\] then says that $uZ_2^+\subset\mathrm{span}(Z_1^+\cup Z_2^+)$. These inclusions and the iteration $u^{p+1}\xi=u(u^p\xi)$ combining with the initial equation
$u\xi=\alpha g_1+\mathring g\otimes g_1$ result in
$\{u^p\xi:p\in\mathbb{N}\}\subset\mathrm{span}(Z^+\cup Z_1^+\cup Z_2^+)$.

In the same vein, one can ascertain that
$\{u^{-p}\xi:p\in\mathbb{N}\}\subset\mathrm{span}(Z^-\cup Z_1^-\cup Z_2^-)$, where $Z^-=\{g_{-p}:p\in\mathbb{N}\}$,
\[Z_1^-=\{z_1\otimes(S_1^{-1}\xi_1)^\circ\otimes z_2\otimes(S_1^{-1}\xi_1)^\circ\otimes\cdots\otimes z_n\otimes(S_1^{-1}\xi_1)^\circ:z_1,\ldots,z_n\in Z^-\},\]
and $Z_2^-=\{h\otimes z:z\in Z^-\;\mathrm{and}\;h\in Z_1^-\}$.

On the other hand, $\nu_2^{(2)}=\delta_1$ shows that $v_2=I_{B(\mathcal{H})}$, the identity in $B(\mathcal{H})$, and so
$(v_1v_2)^q\xi=v_1^q\xi\in\mathbb{C}\xi\oplus\mathcal{\mathring H}_1$, $q\in\mathbb{N}\cup\{0\}$. Thanks to the orthogonality $\{Z,Z_1^{\pm},Z_2^{\pm}\}\perp\mathbb{C}\xi\oplus\mathcal{\mathring H}_1$, we
conclude $m_{p,q}(\nu)=0$ for any $p\in\mathbb{Z}^\times$ and $q\in\mathbb{N}\cup\{0\}$. This verifies the ``only if'' part of (1). The ``if'' part of (1) is obvious.

By similar reasonings, (2) holds. If $\mathrm{m}\times\mathrm{m}$ is a $\bitimes$-factor of $\nu$, then so are distributions $\mathrm{m}\times\delta_1$ and $\delta_1\times\mathrm{m}$, from which we see that (3) holds by (1) and (2).

Finally, we suppose $\nu_2=P$ and justify (4). In view of $P$ being $\bitimes$-idempotent, $\nu_1\bitimet P$ may take the place of $\nu_1$ without affecting the convolution $\nu=\nu_1\bitimet P$. Thus, in the remaining proof, $\nu_1=\nu_1\bitimet P$ stands for this replacement. Since $m_{p,q}(\nu_1)=0=m_{p,q}(\nu_2)$ for $(p,q)=(0,1),(1,0)$, formulas (\ref{up}) and (\ref{vp}), together with Lemma \ref{basiclem}, allow one to see that $m_{p,q}(\nu)=\langle v^q\xi,u^{-p}\xi\rangle=\delta_{p,q}\langle S_1T_1\xi_1,\xi_1\rangle^p\langle S_2T_2\xi_2,\xi_2\rangle^p=\delta_{p,q}m_{1,1}(\nu)^p$ for $(p,q)\in\mathbb{Z} \times (\mathbb N \cup\{0\})$. This furnishes all mixed moments (\ref{mixedm}) of $\nu$.

That $\nu\bitimet P$
has the $\bitimes$-factor $P$ and the established result implies that for any $(p,q)\in\mathbb{Z} \times (\mathbb N \cup\{0\})$, $m_{p,q}(\nu\bitimet P)=\delta_{p,q}m_{1,1}(\nu\bitimet P)^p=\delta_{p,q}m_{1,1}(\nu)^p$ by Lemma \ref{basiclem}. Thus
$m_{p,q}(\nu \bitimet P)=m_{p,q}(\nu)$ or, equivalently, $\nu\bitimet P=\nu$ if (\ref{mixedm}) holds, proving the converse of (4).

All assertions regarding $\bitimes^{\mathrm{op}}$-idempotent factors are direct consequences of (1)-(4), (\ref{opdef}), and the formula $m_{p,q}(\nu^\star)=m_{p,-q}(\nu)$.
\end{proof}

\begin{remark} From $m_{1,1}(\mathrm{m}\times\mathrm{m})=0$, assertion (4) of Proposition \ref{idempotnet2} can be strengthened as that $P$ is the only non-trivial $\bitimes$-idempotent factor of $\nu\in\mathscr{P}_{\mathbb{T}^2}$ if and only if $m_{1,1}(\nu)\neq0$ and (\ref{mixedm}) holds.
\end{remark}

\begin{remark} Recall that given two pairs $(a_1,b_1)$ and $(a_2,b_2)$, one can perform two different types of multiplications on them, one of which is the usual multiplication $(a_1a_2,b_1b_2)$ and the other is the opposite multiplication $(a_1a_2,b_2b_1)$ on the right face of the pairs.
In \cite{diagonal}, the author studied \emph{bi-$\mathrm{R}$-diagonal pairs} and introduced the \emph{bi-Haar unitary pair}, the bipartite pair having distribution $P^\star$. Measures $\nu\in\mathscr{P}_{\mathbb{T}^2}$ satisfying (\ref{opmixedm}) are bi-$\mathrm{R}$-diagonal because of $\nu=\nu\bitimet^{\mathrm{op}}P^\star$ according to Proposition \ref{idempotnet2}. For bi-$\mathrm{R}$-diagonal distributions in the non-bipartite system, we refer the reader to \cite{diagonal}, especially Theorem 4.4 therein.
\end{remark}

For any $c\in\mathbb{D}$, define
\[d\kappa_c(s)=\frac{1-|c|^2}{|1-\bar{c}s|^2}\,d\mathrm{m}(s),\] which is the probability measure on $\mathbb{T}$ induced by the Poisson kernel. It is the normalized Haar measure on $\mathbb{T}$ in case $c=0$. By taking the weak limit we define $\kappa_c=\delta_c$ for $c\in\mathbb T$. Alternatively, $\kappa_c$ with $c\in\mathbb{D}\cup\mathbb{T}$ is the unique probability measure on $\mathbb{T}$ determined by the requirement
\[\int_\mathbb{T}s^p\,d\kappa_c(s)=c^p,\;\;\;\;\;p\in\mathbb{N}.\]
Moments of negative orders of $\kappa_c$ can be acquired by taking the complex conjugate:
$$
\int_\mathbb{T}s^p\,d\kappa_c(s)=\overline{c}^{|p|},\;\;\;\;\;p\in-\mathbb{N}.
$$

Observe that for any $c,d\in\mathbb{D}\cup\mathbb T$, we have
\begin{equation} \label{kappa1}
\nu\bitimet(\kappa_c\times\kappa_d)=\nu\circledast(\kappa_c\times\kappa_d),
\;\;\;\;\;\nu\in\mathscr{P}_{\mathbb{T}^2}.
\end{equation}
To see this, first consider $\nu$ and $\kappa_c\times\kappa_d$ as the distributions of two bi-free commuting unitary faces $(u_1,v_1)$ and $(u_2,v_2)$, respectively, in some $C^*$-probability space $(B(\mathcal{H}),\varphi_\xi)$. We next use the fact that both $(u_2,v_2)$ and $(cI_{B(\mathcal{H})},dI_{B(\mathcal{H})})$ have commuting faces, are bi-free from $(u_1,v_1)$, and  have the same $(p,q)$-moments $c^p d^q$ for $(p,q) \in (\mathbb{N}\cup\{0\})^2$. In view of the universal calculation formula for mixed moments \cite[Lemma 5.2]{V14}, we may replace $(u_2,v_2)$ with $(cI_{B(\mathcal{H})},dI_{B(\mathcal{H})})$ in the computation of
$\varphi_\xi((u_1u_2)^p(v_1v_2)^q)$. This entails $\varphi_\xi((u_1u_2)^p(v_1v_2)^q)
=c^pd^q\varphi_\xi(u_1^pv_1^q)$, and hence
\begin{equation}\label{eq:poisson_kernel}
m_{p,q}(\nu\bitimet(\kappa_c\times\kappa_d))= m_{p,q}(\nu\circledast(\kappa_c\times\kappa_d))
\end{equation}
for $(p,q) \in (\mathbb{N}\cup\{0\})^2$.
Similarly, for $(p,q) \in (\mathbb{N}\cup\{0\})\times (-\mathbb{N}\cup\{0\})$ in \eqref{eq:poisson_kernel}, one can show that the pairs of faces $(u_2,v_2)$ and $(cI_{B(\mathcal{H})},(1/\overline{d})I_{B(\mathcal{H})})$ have the same $(p,q)$-moments $c^p \overline{d}^{|q|}$. The other cases in \eqref{eq:poisson_kernel} can be established by taking complex conjugate. Therefore, the identity \eqref{eq:poisson_kernel} holds for all $(p,q) \in \mathbb Z^2$, justifying (\ref{kappa1}).

A special case of (\ref{kappa1}) is that
\[(\kappa_{c_1}\times\kappa_{d_1})\bitimet
(\kappa_{c_2}\times\kappa_{d_2})=\kappa_{c_1c_2}\times\kappa_{d_1d_2}=
(\kappa_{c_1}\times\kappa_{d_1})\circledast(\kappa_{c_2}\times\kappa_{d_2})\] is valid for any $c_1,c_2,d_1,d_2\in\mathbb{D}\cup \mathbb T$, which implies the following results.

\begin{prop} \label{kappainfdiv}
The measure $\kappa_c\times\kappa_d$ is both $\circledast$- and $\boxtimes\boxtimes$-infinitely divisible for any $c,d\in\mathbb{D}\cup\mathbb{T}$.
\end{prop}

\begin{prop} \label{deltaID}
Any probability measure $\nu$ on $\mathbb{T}^2$ with moments satisfying \emph{(\ref{mixedm})} is expressed as $P\circledast(\kappa_c\times\delta_1)$, where $c=m_{1,1}(\nu)$. In particular, $\nu$ is both $\circledast$- and $\bitimes$-infinitely divisible.
\end{prop}

\begin{proof} The fact that $m_p(\kappa_c)=c^p$ for $p\in \mathbb N$ and $m_{p,q}(P)=\delta_{p,q}$ for $(p,q) \in \mathbb Z^2 $ implies $m_{p,q}(P\circledast(\kappa_c\times\delta_1))=\delta_{p,q}c^{p}$ for $(p,q)\in\mathbb{Z}\times (\mathbb N \cup\{0\})$, and hence $\nu =P\circledast(\kappa_c\times\delta_1)$.

The measure $P$ is $\circledast$-infinitely divisible as it is $\circledast$-idempotent. Combining this with Proposition \ref{kappainfdiv} yields that $P\circledast(\kappa_c\times\delta_1)$ is $\circledast$-infinitely divisible. Thanks to the identity $P\circledast(\kappa_c\times\delta_1)=P\bitimet(\kappa_c\times\delta_1)$, which is a consequence of (\ref{kappa1}), similar reasoning works for proving the $\bitimes$-infinite divisibility of $\nu$. This completes the proof.
\end{proof}

A consequence of (\ref{kappa1}) and Proposition \ref{deltaID} is that the following identity holds for every $\nu \in\mathscr{P}_{\mathbb T^2}$:
\begin{equation} \label{eq:identityP}
P\bitimet \nu = P \bitimet (\kappa_{m_{1,1}(\nu)}\times \delta_1) = P \circledast (\kappa_{m_{1,1}(\nu)}\times \delta_1).
\end{equation}

The following is a bi-free multiplicative analog of the classical multiplicative limit theorem.

\begin{thm} \label{characterID}
Let $\{\nu_{nk_n}\}_{n\geq,1\leq k\leq k_n}$ be an infinitesimal triangular array in $\mathscr{P}_{\mathbb{T}^2}$ and $\{\boldsymbol\xi_n\}$ be a sequence in $\mathbb{T}^2$. If the sequence
\begin{equation} \label{ITbicircle}
\delta_{\boldsymbol\xi_n}\bitimet\nu_{n1}\bitimet\cdots
\bitimet\nu_{nk_n}
\end{equation}
has a weak limit $\nu$, then $\nu$ is $\bitimes$-infinitely divisible.
\end{thm}

\begin{proof} We separately consider three possible statuses (i) $m_{1,0}(\nu)\neq0\neq m_{0,1}(\nu)$, (ii) $m_{1,0}(\nu)=0\neq m_{0,1}(\nu)$ (the case $m_{1,0}(\nu)\neq0=m_{0,1}(\nu)$ is treated similarly to (ii)), and (iii) $m_{1,0}(\nu)=0=m_{0,1}(\nu)$.

(i) Once we can prove that $m_{1,1}(\nu)\neq0$, then the $\bitimes$-infinite divisibility of $\nu$ will follow from \cite[Theorem 4.2]{HW2}. Assume to the contrary that $m_{1,1}(\nu)=0$, which together with Lemma \ref{basiclem} implies that as $n\to\infty$,
\[m_{1,1}(\delta_{\boldsymbol\xi_n})m_{1,1}(\nu_{n1})\cdots m_{1,1}(\nu_{nk_n})=m_{1,1}(\delta_{\boldsymbol\xi_n}\bitimet\nu_{n1}
\bitimet\cdots\bitimet\nu_{nk_n})\to0.\] Then there exists a sequence $\{\ell_n\}\subset\mathbb{N}$ so that as $n\to\infty$,
$m_{1,1}(\delta_{\boldsymbol\xi_n})m_{1,1}(\nu_{n1})\cdots m_{1,1}(\nu_{n\ell_n})\to0$ and $m_{1,1}(\nu_{n,\ell_n+1})\cdots m_{1,1}(\nu_{nk_n})\to0$, namely, one has
$m_{1,1}(\delta_{\boldsymbol\xi_n}\bitimet\nu_{n1}
\bitimet\cdots
\bitimet\nu_{n\ell_n})\to0$ and $m_{1,1}(\nu_{n,\ell_n+1}\bitimet\cdots
\bitimet\nu_{nk_n})\to0$ by Lemma \ref{basiclem} again. We can select, for example,
\[\ell_n=\min\big\{1 \leq k \leq k_n: |m_{1,1}(\delta_{\boldsymbol\xi_n}\bitimet\nu_{n1}
\bitimet\cdots
\bitimet\nu_{nk})| \leq |m_{1,1}(\delta_{\boldsymbol\xi_n}\bitimet\nu_{n1}
\bitimet\cdots
\bitimet\nu_{nk_n})|^{1/2}\big\}.\]

One may assume, by passing to a subsequence if necessary, that $\delta_{\boldsymbol\xi_n}\bitimet\nu_{n1}\bitimet\cdots
\bitimet\nu_{n\ell_n}\Rightarrow\nu_1'\in\mathscr{P}_{\mathbb{T}^2}$ and $\nu_{n,\ell_n+1}\bitimet\cdots
\bitimet\nu_{nk_n}\Rightarrow\nu_1''\in\mathscr{P}_{\mathbb{T}^2}$. Then we have $\nu=\nu_1'\bitimet\nu_1''$ and $m_{1,1}(\nu_1')=0=m_{1,1}(\nu_1'')$. Also, the formula $m_{1,0}(\nu)=m_{1,0}(\nu_1')m_{1,0}(\nu_1'')$ indicates that either $|m_{1,0}(\nu_1')|\geq|m_{1,0}(\nu)|^{1/2}$ or $|m_{1,0}(\nu_1'')|\geq|m_{1,0}(\nu)|^{1/2}$ must occur; assume, without loss of generality, that the first inequality is valid. Carrying out the same arguments on $\nu_1'$ allows us to obtain measures $\nu_2',\nu_2''\in\mathscr{P}_{\mathbb{T}^2}$ fulfilling requirements $\nu_1'=\nu_2'\bitimet\nu_2''$, $m_{1,1}(\nu_2')=0=m_{1,1}(\nu_2'')$, and $|m_{1,0}(\nu_2')|\geq|m_{1,0}(\nu_1')|^{1/2}$. Continuing this process then results in the existence of sequences $\{\nu_n'\},\{\nu_n''\}\subset\mathscr{P}_{\mathbb{T}^2}$ for which $\nu_n'=\nu_{n+1}'\bitimet\nu_{n+1}''$, $m_{1,1}(\nu_n')=0=m_{1,1}(\nu_n'')$, and $|m_{1,0}(\nu_{n+1}')|\geq|m_{1,0}(\nu_n')|^{1/2}$ hold.

As noted above, one has $\nu=\nu_n'\bitimet\nu_n'''$ for some $\nu_n'''\in\mathscr{P}_{\mathbb{T}^2}$ and $|m_{1,0}(\nu_n')|\geq|m_{1,0}(\nu)|^{1/2^n}$.
Passing to subsequences if needed again, assume that $\nu_n'\Rightarrow\nu_1\in\mathscr{P}_{\mathbb{T}^2}$ and $\nu_n'''\Rightarrow\nu_2\in\mathscr{P}_{\mathbb{T}^2}$, and so $\nu=\nu_1\bitimet\nu_2$, $m_{1,1}(\nu_1)=0$, and $|m_{1,0}(\nu_1)|=1$. The last identity reveals that $\nu_1=\delta_{\alpha}\times\nu_1^{(2)}$, where $\alpha=m_{1,0}(\nu_1)\in\mathbb{T}$. In addition, making use of $0\neq m_{0,1}(\nu)=m_{0,1}(\nu_1)m_{0,1}(\nu_2)$ yields $m_{0,1}(\nu_1)\neq0$.
However, these discussions would lead to $m_{1,1}(\nu_1)=\alpha m_{0,1}(\nu_1)\neq0$, a contradiction. Hence we must have $m_{1,1}(\nu)\neq0$, as desired.

(ii) Note that the marginal $\nu^{(2)}$ is $\boxtimes$-infinitely divisible by \cite[Theorem 2.1]{BB08}.
The $\bitimes$-infinite divisibility of $\nu$ will follow immediately if one can argue that $\nu=\mathrm{m}\times\nu^{(2)}$. The proof, presented below, is basically similar to that of (1).

First, applying the strategy employed in the first paragraph of (1) to $m_{1,0}(\nu)=0$ indicates the presence of $\ell_n\in\mathbb{N}$ satisfying
$m_{1,0}(\delta_{\boldsymbol\xi_n}\bitimet\nu_{n1}\bitimet\cdots
\bitimet\nu_{n\ell_n})\to0$ and $m_{1,0}(\nu_{n,\ell_n+1}\bitimet\cdots
\bitimet\nu_{nk_n})\to0$ as $n\to\infty$. Assume, dropping a subsequence if necessary, that $\delta_{\boldsymbol\xi_n}\bitimet\nu_{n1}\bitimet\cdots
\bitimet\nu_{n\ell_n}\Rightarrow\nu_1'\in\mathscr{P}_{\mathbb{T}^2}$ and $\nu_{n,\ell_n+1}\bitimet\cdots
\bitimet\nu_{nk_n}\Rightarrow\nu_1''\in\mathscr{P}_{\mathbb{T}^2}$. Thus we have $\nu=\nu_1'\bitimet\nu_1''$ and $m_{1,0}(\nu_1')=0=m_{1,0}(\nu_1'')$.
We may further assume $|m_{0,1}(\nu_1')|\geq|m_{0,1}(\nu)|^{1/2}$. Mimicking the arguments in (1) constructs sequences $\{\nu_n'\},\{\nu_n''\}\subset\mathscr{P}_{\mathbb{T}^2}$ meeting conditions
$\nu=\nu_n'\bitimet\nu_n''$, $m_{1,0}(\nu_n')=0=m_{1,0}(\nu_n'')$, and $|m_{0,1}(\nu_n')|\geq|m_{0,1}(\nu)|^{1/2^n}$.
Passing to subsequences if needed again, assume $\nu_n'\Rightarrow\nu_1\in\mathscr{P}_{\mathbb{T}^2}$ and $\nu_n''\Rightarrow\nu_2\in\mathscr{P}_{\mathbb{T}^2}$. Then we come to that $\nu=\nu_1\bitimet\nu_2$, $m_{1,0}(\nu_1)=0=m_{1,0}(\nu_2)$, and $\nu_1^{(2)}=\delta_\alpha$
for some $\alpha\in\mathbb{T}$.
To proceed the proof, we shall use the notations introduced in Section \ref{sec2.3}.
Since $v_1=\alpha I_{B(\mathcal{H})}$, it follows that $v^q\xi=\alpha^qv_2^q\xi\in\mathbb{C}\xi\oplus\mathcal{\mathring H}_2$ for any $q\in\mathbb{Z}$. Thus, the formula in (\ref{up}) implies that $m_{p,q}(\nu)=\langle v^q\xi,u^{-p}\xi\rangle=0$ for any $(p,q)\in\mathbb{N}\times\mathbb{Z}$, proving $\nu=\mathrm{m}\times\nu^{(2)}$.

(iii) In this case, we have $\nu^{(1)}=\mathrm{m}=\nu^{(2)}$ by \cite[Lemma 6.1]{BerVoicu92} and \cite[Theorem 2.1]{BB08}. Further, one can employ the proof in (ii) to show that there are $\nu_1,\nu_2\in\mathscr{P}_{\mathbb{T}^2}$, which are limiting distributions of sequences of the form in (\ref{ITbicircle}), so that $\nu=\nu_1\bitimet\nu_2$ and $m_{1,0}(\nu_1)=0=m_{1,0}(\nu_2)$. Then $m_{0,1}(\nu_1)m_{0,1}(\nu_2)=m_{0,1}(\nu)=0$. If $m_{0,1}(\nu_1)=0=m_{0,1}(\nu_2)$, then (\ref{up}) and (\ref{vp}) yield that $m_{p,q}(\nu)=\delta_{p,q}m_{1,1}(\nu)^p$ for $(p,q) \in \mathbb Z \times (\mathbb N \cup\{0\})$, whence $\nu$ is $\bitimes$-infinitely divisible by Proposition \ref{deltaID}. For the other case, say $m_{0,1}(\nu_1)\neq0$, the established conclusion in (ii) then shows that $\nu_1=\mathrm{m}\times\nu_1^{(2)}$. In such a situation, the measure $\nu_1$, as well as $\nu$, has the $\bitimes$-factor $\mathrm{m}\times\delta_1$. Thus, Proposition \ref{idempotnet2} says that $\nu=\mathrm{m}\times\mathrm{m}$, which is $\bitimes$-infinitely divisible by Proposition \ref{deltaID}.
\end{proof}

\begin{cor} The set $\mathcal{ID}(\boxtimes\boxtimes)$ is weakly closed.
\end{cor}

We are now in a position to characterize distributions in $\mathcal{ID}(\bitimes)$ carrying no non-trivial $\bitimes$-idempotent factors.

\begin{thm} \label{IDmoment}
In order that a measure $\nu\in\mathcal{ID}(\bitimes)$ contains no non-trivial $\bitimes$-idempotent factor, it is
necessary and sufficient that $m_{1,0}(\nu)\neq0\neq m_{0,1}(\nu)$, in which case $m_{1,1}(\nu)\neq0$.
\end{thm}

\begin{proof} According to Proposition \ref{idempotnet2}, only the necessity requires a proof.
We merely prove that $\nu$ has a non-trivial $\bitimes$-idempotent factor when $m_{1,0}(\nu)=0$, because the case $m_{0,1}(\nu)=0$ can be handled in the same way. To do so, let $m_{1,0}(\nu)=0$, and consider two possible cases (i) $m_{0,1}(\nu)=0$ and (ii) $m_{0,1}(\nu)\neq0$, which are discussed separately below. Note that $m_{p,0}(\nu)=0$ for all $p\in\mathbb{N}$ since $\nu^{(1)}=m$.

Case (i): Since $\nu^{(j)}=\mathrm{m}$ for $j=1,2$, one can mimic the proof of Proposition \ref{idempotnet1}, especially employ equations (\ref{up}) and (\ref{vp}), to obtain $m_{p,q}(\nu)=\delta_{p,q}m_{1,1}(\nu)^p$ for $(p,q)\in\mathbb{Z}\times (\mathbb N \cup\{0\})$. Hence $P$ is a $\bitimes$-factor of $\nu$ by Proposition \ref{idempotnet2}.

Case (ii): To treat this case, let $\nu_n'\in\mathscr{P}_{\mathbb{T}^2}$ be an $n$-th $\bitimes$-convolution root of $\nu$ for any $n\in \mathbb N$, i.e., $(\nu_n')^{\bitimes n}=\nu$. Then we have
$\nu=\nu_n'\bitimet\nu_n''$, where $\nu_n''=(\nu_n')^{\bitimes(n-1)}$, $m_{1,0}(\nu_n')=0=m_{1,0}(\nu_n'')$ and $|m_{0,1}(\nu_n')|=|m_{0,1}(\nu)|^{1/n}$. If $\nu'$ and $\nu''$ are any weak limits of $\{\nu_n'\}$ and $\{\nu_n''\}$, respectively, then we further obtain
$\nu= \nu' \bitimet \nu''$, $m_{1,0}(\nu')=0=m_{1,0}(\nu'')$, and $|m_{0,1}(\nu')|=1$. This leads to $(\nu')^{(2)}=\delta_\alpha$ for some $\alpha \in \mathbb T$, which is exactly the situation dealt in the last part of the proof (ii) of Theorem \ref{characterID}. Thus we conclude that $\nu=\mathrm{m} \times \nu^{(2)}$, which has the $\bitimes$-idempotent factor $\mathrm{m}\times \delta_1$ by Proposition \ref{idempotnet2}.

Lastly, we turn to argue that $m_{1,1}(\nu)\neq0$ if $m_{1,0}(\nu)\neq0\neq m_{0,1}(\nu)$. Note that any sequence $\{\nu_n\}$
satisfying $\nu=\nu_n^{\bitimes 2^n}$ has a subsequence $\{\nu_{n_\ell}\}$ converging weakly to $\delta_{\boldsymbol\xi}$ for some $\boldsymbol\xi\in\mathbb{T}^2$ (cf. the proof (i) of Theorem \ref{characterID}). Then Lemma \ref{basiclem} implies that $|m_{1,1}(\nu)|^{2^{-n_\ell}}=|m_{1,1}(\nu_{n_\ell})|\to|m_{1,1}(\delta_{\boldsymbol\xi})|=1$, which leads to the desired result.
\end{proof}

Proposition \ref{idempotnet2} and Theorem \ref{IDmoment} readily imply the following.

\begin{cor} \label{ID-exception}
Any measure $\nu$ in $\mathcal{ID}(\bitimes)\backslash\mathscr{P}_{\mathbb{T}^2}^\times$ is either $\nu^{(1)} \times \mathrm{m}$, $\mathrm{m} \times \nu^{(2)}$, $\mathrm{m} \times \mathrm{m}$ or $P\circledast (\kappa_c \times \delta_1)$, where $\nu^{(1)}$ and $\nu^{(2)}$ are in $\mathcal{ID}(\boxtimes)$ with nonzero mean and $c \in (\mathbb D\cup \mathbb T)\setminus\{0\}$.
\end{cor}

\section{Equivalent Conditions on Limit Theorems} \label{sec4}
This section is devoted to exploring the associations among the conditions introduced in Section \ref{LTcond} and the following one.

\begin{cond} \label{condX2}
Let $\{\rho_n\}$ be a sequence in $\mathscr{M}_{\mathbb{T}^d}^\mathbf{1}$.

\begin{enumerate}[(i)]
\setcounter{enumi}{2}
\item\label{item:condiii} There exists some $\rho\in\mathscr{M}_{\mathbb{T}^d}^\mathbf{1}$ with
$\rho(\{\mathbf{1}\})=0$ so that
$\rho_n\Rightarrow_\mathbf{1}\rho$.

\item\label{item:condiv} The following limits exist in $\mathbb{R}$ for any $\boldsymbol p\in\mathbb{Z}^d$:
\[\lim_{\epsilon\to0}\limsup_{n\to\infty}\int_{\mathscr{U}_\epsilon}
\big\langle\boldsymbol p,\Im\boldsymbol s\big\rangle^2d\rho_n({\boldsymbol s})=Q(\boldsymbol p)=\lim_{\epsilon\to0}\liminf_{n\to\infty}\int_{\mathscr{U}_\epsilon} \big\langle\boldsymbol p,\Im\boldsymbol s\big\rangle^2d\rho_n({\boldsymbol s}).\]
\end{enumerate}
\end{cond}

Condition \ref{condX1} with $d=2$ was used in \cite[Theorem 3.4]{HW2} to prove the limit theorem for the bi-free multiplicative convolution, while Condition \ref{condX2} is beneficial for the corresponding classical limit theorem (see Section \ref{sec6} for more details). The connection and properties of these two conditions are given below.

\begin{prop} \label{equicondX}
{\rm Condition \ref{condX1}} is equivalent to {\rm Condition \ref{condX2}}, in which
\begin{equation} \label{lambdarho}
d\lambda_j({\boldsymbol s})=(1-\Re s_j)d\rho({\boldsymbol s})+\frac{Q(\boldsymbol e_j)}{2}\delta_{\mathbf{1}}(d{\boldsymbol s}),\;\;\;\;\;j=1,\ldots,d,
\end{equation}
\begin{equation} \label{LevyTcond}
\int_{\mathbb{T}^d}\|\mathbf{1}-\Re\boldsymbol s\|\,d\rho({\boldsymbol s})<\infty,
\end{equation}
and the quadratic form $Q(\cdot)=\langle\mathbf{A}\cdot,\cdot\rangle$ on $\mathbb{Z}^d$ is determined by
a positive semi-definitive matrix $\mathbf{A}=(a_{j\ell})$
whose entries are
\begin{equation} \label{constc}
a_{j\ell}=L_{j\ell}-\int_{\mathbb{T}^d}(\Im s_j)(\Im s_\ell)\,d\rho({\boldsymbol s})\in\mathbb{R},\;\;\;\;\;j,\ell=1,\ldots,d.
\end{equation}
Moreover, $a_{jj}=2\lambda_j(\{\mathbf{1}\})$ for $j=1,\ldots,d$.
\end{prop}

\begin{proof} Suppose first that Condition \ref{condX1} is satisfied. Then the relation $(1-\Re s_j)d\lambda_\ell=(1-\Re s_\ell)d\lambda_j$, $j,\ell=1,\ldots,d$, guaranteed by the item \eqref{item:condi} of Condition \ref{condX1} ensures that the measure
\begin{equation} \label{rhodef1}
d\rho({\boldsymbol s})=\frac{1_{\{s_j\neq1\}}({\boldsymbol s})}{1-\Re s_j}\,d\lambda_j({\boldsymbol s})
\end{equation} is unambiguous. Also, it satisfies requirements $\rho(\mathbb{T}^d\backslash\mathscr{U}_\epsilon)<\infty$ for any $\epsilon>0$ and (\ref{LevyTcond}). Hence the measure $\rho$ that we just constructed belongs to $\mathscr{M}_{\mathbb{T}^d}^\mathbf{1}$.

To see $\rho_n\Rightarrow_\mathbf{1}\rho$, pick a continuous function $f$ on $\mathbb{T}^d$ with support contained within $\mathbb{T}^d\backslash\mathscr{U}_\delta$ for some $\delta>0$. This $f$ produces $d$ continuous
functions on $\mathbb{T}^d$, which are
\[f_j(\boldsymbol s)=\frac{\mathrm{dist}(U_j,\boldsymbol s)}{\mathrm{dist}(U_1,\boldsymbol s)+\cdots+\mathrm{dist}(U_d,\boldsymbol s)}\cdot 1_{\mathbb{T}^d\backslash\mathscr{U}_\delta}(\boldsymbol s) f(\boldsymbol s),\] where $U_j=\{{\boldsymbol u}\in\mathbb{T}^d:|\arg u_j|<\delta/d\}$ and $\mathrm{dist}(U_j,\boldsymbol s)=\inf\{\|\arg\boldsymbol s-\arg \boldsymbol u\|:\boldsymbol u\in U_j\}$ for $j=1,\ldots,d$. Obviously, the relation $f=f_1+\cdots+f_d$ holds and each $f_j/(1-\Re s_j)$ is continuous on $\mathbb{T}^d$. These observations and the weak convergence $\lambda_{nj}\Rightarrow\lambda_j$ then yield that
\begin{align*}
\int_{\mathbb{T}^d}f({\boldsymbol s})\,d\rho_n({\boldsymbol s})&=\int_{\mathbb{T}^d}\sum_{j=1}^df_j({\boldsymbol s})\,d\rho_n({\boldsymbol s}) \\
&=\sum_{j=1}^d\int_{\mathbb{T}^d}\frac{f_j({\boldsymbol s})}{1-\Re s_j}\,d\lambda_{nj}({\boldsymbol s})\mathop{\longrightarrow}\limits_{n\to\infty}\sum_{j=1}^d\int_{\mathbb{T}^d}\frac{f_j({\boldsymbol s})}{1-\Re s_j}\,d\lambda_j({\boldsymbol s})=\int_{\mathbb{T}^d}f({\boldsymbol s})\,d\rho({\boldsymbol s}).
\end{align*}
Therefore, we have completed the verification of the item \eqref{item:condiii} of Condition \ref{condX2}.

We next demonstrate the validity of the following identities for $1\leq j,\ell\leq d$,
\begin{equation} \label{a}
\lim_{\epsilon\to0}\limsup_{n\to\infty}\int_{\mathscr{U}_\epsilon}(\Im s_j)(\Im s_\ell)\,d\rho_n=
\lim_{\epsilon\to0}\liminf_{n\to\infty}\int_{\mathscr{U}_\epsilon}(\Im s_j)(\Im s_\ell)\,d\rho_n,
\end{equation} which confirms that of Condition \ref{condX2}\eqref{item:condiv}.
To continue, observe that the mapping $s\mapsto(\Im s)^2/(1-\Re s)$ is continuous on $\mathbb{T}$ and at the origin, it takes value
\begin{equation} \label{ImRe}
\lim_{\arg s\to0}\frac{(\Im s)^2}{1-\Re s}=2.
\end{equation} Then (\ref{LevyTcond}) and (\ref{ImRe}) imply that for $j,\ell=1,\ldots,d$,
\[\int_{\mathbb{T}^d}|(\Im s_j)(\Im s_\ell)|\,d\rho\leq\left(\int_{\mathbb{T}^d}(\Im s_j)^2\,d\rho\right)^{1/2}
\left(\int_{\mathbb{T}^d}(\Im s_\ell)^2\,d\rho\right)^{1/2}<\infty.\] In order to get results (\ref{constc}) and (\ref{a}), we examine the following differences which are related to them:
\[D_n(\epsilon)=\int_{\mathbb{T}^d}(\Im s_j)(\Im s_\ell)\,d\rho_n-
\int_{\mathbb{T}^d}(\Im s_j)(\Im s_\ell)\,d\rho-\int_{\mathscr{U}_\epsilon}(\Im s_j)(\Im s_\ell)\,d\rho_n,\] which further split into the sum of
\[I_{n1}(\epsilon)=\int_{\mathbb{T}^d\backslash\mathscr{U}_\epsilon}(\Im s_j)(\Im s_\ell)\,d\rho_n-
\int_{\mathbb{T}^d\backslash\mathscr{U}_\epsilon}(\Im s_j)(\Im s_\ell)\,d\rho\] and
\[I_2(\epsilon)=-\int_{\mathscr{U}_\epsilon}(\Im s_j)(\Im s_\ell)\,d\rho.\]

Apparently, we have $\lim_{\epsilon\to0}I_2(\epsilon)=0$ owing to $(\Im s_j)(\Im s_\ell)\in L^1(\rho)$. Next, take an $\epsilon'\in(\epsilon,2\epsilon]$ with the attribute that $\rho(\partial\mathscr{U}_{\epsilon'})=0$, the presence of which is insured by the finiteness of
the measure $1_{\mathbb{T}^d\backslash\mathscr{U}_\epsilon}\rho$ on $\mathbb{T}^d$. Then applying Proposition \ref{Portman}
to the established result $\rho_n\Rightarrow_\mathbf{1}\rho$ results in
\[\lim_{n\to\infty}\int_{\mathbb{T}^d\backslash\mathscr{U}_{\epsilon'}}(\Im s_j)(\Im s_\ell)\,d\rho_n=
\int_{\mathbb{T}^d\backslash\mathscr{U}_{\epsilon'}}(\Im s_j)(\Im s_\ell)\,d\rho.\]
On the other hand, employing (\ref{ImRe}) and working with the closed subset $F_\epsilon=\{{\boldsymbol s}\in\mathbb{T}^d:\epsilon\leq\|\arg{\boldsymbol s}\|\leq\epsilon'\}$, we come to
\begin{eqnarray} \label{ImsImt}
\left(\limsup_{n\to\infty}\int_{F_\epsilon}|\Im s_j||\Im s_\ell|\,d\rho_n\right)^2 \nonumber&\leq&
\limsup_{n\to\infty}\int_{F_\epsilon}(\Im s_j)^2\,d\rho_n\cdot\int_{F_\epsilon}(\Im s_\ell)^2\,d\rho_n \nonumber \\
&=&\limsup_{n\to\infty}\int_{F_\epsilon}\frac{(\Im s_j)^2}{1-\Re s_j}\,d\lambda_{nj}\cdot\int_{F_\epsilon}\frac{(\Im s_\ell)^2}{1-\Re s_\ell}\,d\lambda_{n\ell} \\
&\leq&8\lambda_j(F_\epsilon)\lambda_\ell(F_\epsilon)\to0\;\;\;\mathrm{as}\;\;\epsilon\to0 \nonumber.
\end{eqnarray} With the help of the facts $(\Im s_j)(\Im s_\ell)\in L^1(\rho)$ and $\mathbb{T}^d\backslash\mathscr{U}_\epsilon=(\mathbb{T}^d\backslash\mathscr{U}_{\epsilon'})\cup F_\epsilon$, we can conclude that $\lim_{\epsilon\to0}\limsup_{n\to\infty}|I_{n1}(\epsilon)|=0$. Consequently, we have shown that $\lim_{\epsilon\to0}\limsup_{n\to\infty}|D_n(\epsilon)|=0$, which also accounts for equations (\ref{constc}) and (\ref{a}) with any indices $j$ and $\ell$.

If $\epsilon'$ is also chosen so that $\lambda_j(\partial\mathscr{U}_{\epsilon'})=0$, then we draw once again from (\ref{ImRe}) that $a_{jj}=2\lambda_j(\{\mathbf{1}\})$ because
\begin{align*}
\limsup_{n\to\infty}\int_{\mathscr{U}_\epsilon}\big|2(1-\Re s_j)-(\Im s_j)^2\big|\,d\rho_n&\leq\limsup_{n\to\infty}\int_{\mathscr{U}_{\epsilon'}}\left|2-\frac{(\Im s)^2}{1-\Re s}\right|d\lambda_{nj} \\
&=\int_{\mathscr{U}_{\epsilon'}}\left|2-\frac{(\Im s_j)^2}{1-\Re s_j}\right|d\lambda_j\mathop{\longrightarrow}\limits_{\epsilon\to0}0.
\end{align*} This conclusion and (\ref{rhodef1}) give (\ref{lambdarho}).

As discussed above, the limits in \eqref{item:condiv} of Condition \ref{condX2} are equal to $\langle\mathbf{A}\boldsymbol p,\boldsymbol p\rangle$ with $\mathbf{A}=(a_{j\ell})$ for any $\boldsymbol p\in\mathbb{Z}^d$. That $\mathbf{A}\geq0$ can be gained by the positivity of the quadratic form $Q$.

Next, we elaborate that Condition \ref{condX2} implies Condition \ref{condX1}. Define $\lambda_j$'s as in (\ref{lambdarho}). These measures thus obtained are all in $\mathscr{M}_{\mathbb{T}^d}$, and the arguments for this go as follows. Select a sequence $\{\epsilon_m\}$ such that $\epsilon_m\downarrow0$ as $m\to\infty$ and
$\rho(\{\|\arg{\boldsymbol s}\|=\epsilon_m\})=0$ for each $m$. Then \eqref{item:condiv}, along with Proposition \ref{Portman}, indicates that for any numbers $m<m'$ both large enough, one has
\[\int_{\{\epsilon_{m'}<\|\arg{\boldsymbol s}\|<\epsilon_m\}}(\Im s_j)^2\,d\rho({\boldsymbol s})\leq1+Q(\boldsymbol e_j).\] Thanks to monotone convergence theorem, (\ref{ImRe}), and the assumption $\rho(\{\mathbf{1}\})=0$, one further gets that for $m$ large enough,
\[\int_{\mathscr{U}_{\epsilon_m}}(1-\Re s_j)\,d\rho({\boldsymbol s})<\infty.\] This proves that $\lambda_j(\mathbb{T}^d)<\infty$ and $(\Im s_j)^2\in L^1(\rho)$ for any $j$.

After the previous preparations, we are in a position to justify the weak convergence $\lambda_{nj}\Rightarrow\lambda_j$. Given a continuous function $f$ on $\mathbb{T}^d$, we shall analyze the following estimate
\[\left|\int_{\mathbb{T}^d}f\,d\lambda_{nj}-\int_{\mathbb{T}^d}f\,d\lambda_j\right|\leq D_{n1}(m)+D_{n2}(m)+D_{n3}(m)+D_{n4}(m),\] where
\begin{align*}
D_{n1}(m)&=\int_{\mathscr{U}_{\epsilon_m}}|f({\boldsymbol s})-f(\mathbf{1})|\,d\lambda_{nj}({\boldsymbol s}), \\
D_{n2}(m)&=|f(\mathbf{1})|\big|\lambda_{nj}(\mathscr{U}_{\epsilon_m})-Q(\boldsymbol e_j)/2\big|, \\
D_{n3}(m)&=\int_{\mathscr{U}_{\epsilon_m}\backslash\{\mathbf{1}\}}|f|\,d\lambda_j({\boldsymbol s}),\;\;\;\mathrm{and} \\
D_{n4}(m)&=\left|\int_{\mathbb{T}^d\backslash\mathscr{U}_{\epsilon_m}}f\,d\lambda_{nj}({\boldsymbol s})-
\int_{\mathbb{T}^d\backslash\mathscr{U}_{\epsilon_m}}f\,d\lambda_j({\boldsymbol s})\right|.
\end{align*} An application of (\ref{ImRe}) enables us to get
\begin{align*}
\lim_{m\to\infty}\limsup_{n\to\infty}|D_{n1}(m)|
\leq Q(\boldsymbol e_j)\cdot\lim_{m\to\infty}\sup_{{\boldsymbol s}\in\mathscr{U}_{\epsilon_m}}\big|f({\boldsymbol s})-
f(\mathbf{1})\big|=0.
\end{align*} Next, observe that one can express the difference $2\lambda_{nj}(\mathscr{U}_{\epsilon_m})-Q(\boldsymbol e_j)$ as the sum
\[\int_{\mathscr{U}_{\epsilon_m}}\big[2(1-\Re s_j)-(\Im s_j)^2\big]d\rho_n({\boldsymbol s})+
\int_{\mathscr{U}_{\epsilon_m}}(\Im s_j)^2\,d\rho_n({\boldsymbol s})-Q(\boldsymbol e_j).\] Since (\ref{ImRe}) and the item \eqref{item:condiv} imply that
\[\lim_{m\to\infty}\limsup_{n\to\infty}\int_{\mathscr{U}_{\epsilon_m}}\big|2(1-\Re s_j)-(\Im s_j)^2\big|\,d\rho_n({\boldsymbol s})=0,\] we come to  $\lim_{m\to\infty}\limsup_{n\to\infty}|D_{n2}(m)|=0$. On the other hand, the finiteness of $\lambda_j(\mathbb{T}^d)$ leads to
\[\lim_{m\to\infty}\limsup_{n\to\infty}D_{n3}(m)\leq\|f\|_\infty\cdot\lim_{m\to\infty}\lambda_j(\mathscr{U}_{\epsilon_m}\backslash\{\mathbf{1}\})=0.\] That we have $\lim_{n\to\infty}D_{n4}(m)=0$ for all $m$ evidently follows from Proposition \ref{Portman}. Putting all these observations together illustrates $\lambda_{nj}\Rightarrow\lambda_j$.

It remains to deal with the item \eqref{item:condii} of Condition \ref{condX1}, in which the integral is divided into the sum
\[\int_{\mathscr{U}_{\epsilon_m}}(\Im s_j)(\Im s_\ell)\,d\rho_n+\int_{\mathbb{T}^d\backslash\mathscr{U}_{\epsilon_m}}(\Im s_j)(\Im s_\ell)\,d\rho_n.\] When $j\neq\ell$, taking $\lim_{m\to\infty}\limsup_{n\to\infty}$ and $\lim_{m\to\infty}\liminf_{n\to\infty}$ of the first integral gives the same value $[Q(\boldsymbol e_j+\boldsymbol e_\ell)-Q(\boldsymbol e_j)-Q(\boldsymbol e_\ell)]/2$, while doing the same thing to the second integral yields
\[\int_{\mathbb{T}^d}(\Im s_j)(\Im s_\ell)\,d\rho\] because of $\rho_n\Rightarrow_\mathbf{1}\rho$ and $(\Im s_j)^2+(\Im s_\ell)^2\in L^1(\rho)$. The same technique also elaborates \eqref{item:condii} when $j=\ell$. This finishes the proof of the proposition.
\end{proof}

An intuitive thought is that measures on $\mathbb{T}^d$ obtained by rotating measures within controllable angles maintain the same structural properties, such as Condition \ref{condX2}, as the original ones. The statement and its rigorous proof are given below:

\begin{prop} \label{perturbthm}
Suppose that $\{\nu_{nk}\}\subset\mathscr{P}_{\mathbb{T}^d}$ is a triangular array for which the measure $\rho_n=\sum_{k=1}^{k_n}\nu_{nk}$ satisfies {\rm Condition \ref{condX2}}. If an array $\{\boldsymbol\theta_{nk}\}\subset(-\pi,\pi]^d$ fulfills the condition
\begin{equation} \label{perturbation}
\lim_{n\to\infty}\sum_{k=1}^{k_n}\big(\mathbf{1}-\cos(\boldsymbol\theta_{nk})\big)=\mathbf{0},
\end{equation} then {\rm Condition \ref{condX2}} is still applicable to measures $\widetilde{\rho}_n(\cdot)=\sum_{k=1}^{k_n}\nu_{nk}(e^{i\boldsymbol\theta_{nk}}\cdot)$. In addition, $\widetilde{\rho}_n\Rightarrow_\mathbf{1}\rho$ if $\rho_n\Rightarrow_\mathbf{1}\rho$, and $\rho_n$ and $\widetilde{\rho}_n$ define the same quadratic form in {\rm Condition \ref{condX2}}\eqref{item:condiv}.
\end{prop}

\begin{proof} First of all, (\ref{perturbation}) reveals that
$\lim_n\max_k\|\boldsymbol\theta_{nk}\|=0$.
We now argue that $\widetilde{\rho}_n\Rightarrow_\mathbf{1}\rho$ as well by using Proposition \ref{Portman}. To do so, pick a closed subset $F\subset\mathbb{T}^d\backslash\mathscr{U}_r$ for some $r>0$. Since $\rho(F)<\infty$, it follows that given any $\delta>0$, there exists a closed set $F'\subset\mathbb{T}^d\backslash\mathscr{U}_{r/2}$ such that $e^{i\boldsymbol\theta_{nk}}F\subset F'$ for all sufficiently large $n$ and for all $1\leq k\leq k_n$, and $\rho(F'\backslash F)<\delta$. Then
$\widetilde{\rho}_n(F)=\sum_k\nu_{nk}(e^{i\boldsymbol\theta_{nk}}F)\leq\sum_k\nu_{nk}(F')=
\rho_n(F')$ implies that $\limsup_{n\to\infty}\widetilde{\rho}_n(F)\leq\limsup_{n\to\infty}\rho_n(F')\leq
\rho(F')\leq\rho(F)+\delta$. Consequently, we arrive at the inequality $\limsup_{n\to\infty}\widetilde{\rho}_n(F)\leq\rho(F)$. In the same vein, one can show that $\liminf_{n\to\infty}\widetilde{\rho}_n(G)\geq\rho(G)$ for any set $G$ which is open and bounded away from $\mathbf{1}$. Hence $\widetilde{\rho}_n\Rightarrow_\mathbf{1}\rho$ by Proposition \ref{Portman}.

Next, we turn to demonstrate that both $\rho_n$ and $\widetilde{\rho}_n$ bring out the tantamount quantities in (\ref{a}), which asserts that the quadratic form in \eqref{item:condiv} output by them is unchanged on $\mathbb{Z}^d$. Any index $n$ considered below is always sufficiently large. In the case $j=\ell$, we have the estimate
\[
\int_{\mathscr{U}_\epsilon}(\Im s_j)^2\,d\widetilde{\rho}_n({\boldsymbol s})=\sum_{k=1}^{k_n}\int_{e^{i\boldsymbol\theta_{nk}}\mathscr{U}_\epsilon}\big(\Im(e^{-i\theta_{nkj}}s_j)\big)^2\,d\nu_{nk}({\boldsymbol s})\leq\sum_{k=1}^{k_n}\int_{\mathscr{U}_{2\epsilon}}\big(\Im(e^{-i\theta_{nkj}}s_j)\big)^2\,d\nu_{nk}({\boldsymbol s}),\] where we express $\boldsymbol\theta_{nk}=(\theta_{nk1},\ldots,\theta_{nkd})$. The inequality
\begin{equation} \label{mainineq}
\big(\Im(e^{-i\theta_{nkj}}s_j)\big)^2\leq(\Im s_j)^2+2|\sin\theta_{nkj}||\Im s_j|+\sin^2(\theta_{nkj})
\end{equation} will help us to continue with the arguments. Consideration given to the first term on the right-hand side of (\ref{mainineq}) gives
\[\lim_{\epsilon\to0}\limsup_{n\to\infty}\sum_{k=1}^{k_n}\int_{\mathscr{U}_{2\epsilon}}(\Im s_j)^2\,d\nu_{nk}({\boldsymbol s})=a_{jj}\] by the hypothesis, while analyzing the second term results in
\begin{align*}
\sum_{k=1}^{k_n}|\sin\theta_{nkj}|&\cdot\int_{\mathscr{U}_{2\epsilon}}|\Im s_j|\,d\nu_{nk}({\boldsymbol s}) \\
&\leq
\left(\sum_{k=1}^{k_n}\sin^2(\theta_{nkj})\right)^{1/2}\left(\sum_{k=1}^{k_n}\left(\int_{\mathscr{U}_{2\epsilon}}|\Im s_j|\,d\nu_{nk}({\boldsymbol s})\right)^2\right)^{1/2} \\
&\leq\left(\sum_{k=1}^{k_n}\sin^2(\theta_{nkj})\right)^{1/2}\left(\int_{\mathscr{U}_{2\epsilon}}(\Im s_j)^2\,d\rho_n({\boldsymbol s})\right)^{1/2}
\end{align*} by the Cauchy-Schwarz inequality. The simple fact $\sin^2(x)\leq2(1-\cos x)$ for $x\in\mathbb{R}$ and the assumption (\ref{perturbation}) immediately yield that
\[\lim_{\epsilon\to0}\limsup_{n\to\infty}\sum_{k=1}^{k_n}\int_{\mathscr{U}_{2\epsilon}}|\sin\theta_{nkj}||\Im s_j|\,d\nu_{nk}({\boldsymbol s})=0\] and
\[\lim_{\epsilon\to0}\limsup_{n\to\infty}\sum_{k=1}^{k_n}\int_{\mathscr{U}_{2\epsilon}}\sin^2(\theta_{nkj})\,d\nu_{nk}({\boldsymbol s})\leq
\limsup_{n\to\infty}\sum_{k=1}^{k_n}\sin^2(\theta_{nkj})=0.\] These estimates then lead to
\[\lim_{\epsilon\to0}\limsup_{n\to\infty}\int_{\mathscr{U}_\epsilon}(\Im s_j)^2\,d\widetilde{\rho}_n({\boldsymbol s})\leq a_{jj}.\] Employing the opposite inclusion
$\mathscr{U}_{\epsilon/2}\subset e^{-i\boldsymbol\theta_{nk}}\mathscr{U}_\epsilon$ and inequality
\[\big(\Im(e^{-i\theta_{nkj}}s_j)\big)^2\geq(\Im s_j)^2-2(1-\cos\theta_{nkj})-2|\sin\theta_{nkj}||\Im s_j|-\sin^2(\theta_{nkj})\] helps us to prove
\[\lim_{\epsilon\to0}\liminf_{n\to\infty}\int_{\mathscr{U}_\epsilon}(\Im s_j)^2\,d\rho_n({\boldsymbol s})\geq a_{jj}.\]

Up to this point, we have come to that $\widetilde{\rho}_n$ and $\rho_n$ give the same value in (\ref{a}) when $j=\ell$.
Now we deal with the situation $j\neq\ell$. After careful consideration of all available information, the focus is only needed on the summand
\[\sum_{k=1}^{k_n}\int_{e^{i\boldsymbol\theta_{nk}}\mathscr{U}_\epsilon}(\Im s_j)(\Im s_\ell)\,d\nu_{nk}(\boldsymbol s)\] and justifying that
\begin{equation} \label{symdiff}
\lim_{\epsilon\to0}\limsup_{n\to\infty}\sum_{k=1}^{k_n}\int_{(e^{i\boldsymbol\theta_{nk}}\mathscr{U}_\epsilon)\bigtriangleup\mathscr{U}_\epsilon}|\Im s_j||\Im s_\ell|\,d\nu_{nk}(\boldsymbol s)=0,
\end{equation} where $\bigtriangleup$ denotes the operation of symmetric difference on sets. Using the fact $(e^{i\boldsymbol\theta_{nk}}\mathscr{U}_\epsilon)\bigtriangleup\mathscr{U}_\epsilon\subset\{\epsilon/2\leq\|\arg\boldsymbol s\|\leq2\epsilon\}$ and mimicking the proof of (\ref{ImsImt}) allow us to get (\ref{symdiff}) done. This finishes the proof.
\end{proof}

Recall from (\ref{pushforward}) that the push-forward measure $\tau W^{-1}\in\mathscr{M}_{\mathbb{T}^d}^\mathbf{1}$ of a given $\tau\in\mathscr{M}_{\mathbb{R}^d}^\mathbf{0}$
via the wrapping map $W(\boldsymbol x)=e^{i\boldsymbol x}$ from $\mathbb{R}^d$ to $\mathbb{T}^d$ is defined as
\begin{equation} \label{wrapping1}
(\tau W^{-1})(B)=\tau(\{\boldsymbol x\in\mathbb{R}^d:e^{i\boldsymbol x}\in B\}),\;\;\;\;\;B\in\mathscr{B}_{\mathbb{T}^d}.
\end{equation}
A useful and frequently used result regarding $W$ is the change-of-variables formula stating that a Borel function $f$ on $\mathbb{T}^d$ belongs to $L^1(\tau W^{-1})$ if and only if the function $\boldsymbol x\mapsto f(e^{i\boldsymbol x})$ lies in $L^1(\tau)$, and the equation
\begin{equation} \label{wrapping2}
\int_{\mathbb{T}^d}f({\boldsymbol s})\,d(\tau W^{-1})({\boldsymbol s})=\int_{\mathbb{R}^d}f(e^{i\boldsymbol x})\,d\tau(\boldsymbol x)
\end{equation} holds in either case. In the following, we will translate conditions introduced in Section \ref{LTcond} accordingly via the wrapping map $W$.

\begin{prop} \label{+implyX}
Let $\{\tau_n\}$ and $\tau$ be in $\mathscr{M}_{\mathbb{R}^d}^\mathbf{0}$ satisfying {\rm Condition \ref{cond+1}} {\rm(}or {\rm Condition \ref{cond+2}}{\rm)}.
Then \emph{Condition \ref{condX2}}, as well as \emph{Condition \ref{condX1}}, applies to $\rho_n=\tau_nW^{-1}$ and $\rho=1_{\mathbb{T}^d\backslash\{\mathbf{1}\}}\tau W^{-1}$.
Moreover, $\tau_n$ and $\rho_n$ determine the same quadratic form on $\mathbb{Z}^d$, in particular, the same matrix in \eqref{item:condIV} and \eqref{item:condiv}, respectively.
\end{prop}

\begin{proof} Suppose that Condition \ref{cond+2} holds for $\tau_n$ and $\tau$, and let $\mathbf{A}=(a_{j\ell})$ represent the matrix produced by these measures in \eqref{item:condIV}. According to Proposition \ref{equicondX}, we shall only elaborate that Condition \ref{condX2} is applicable to $\rho_n$ and $\rho$.

That $\rho_n\Rightarrow_\mathbf{1}\rho$ is clearly valid according to the continuous mapping theorem, Proposition \ref{Portman}. It remains to argue that in Condition \ref{condX2}\eqref{item:condiv}, $\rho_n$ also outputs $\mathbf{A}$.
The simple observation that $e^{i\boldsymbol x}\in\mathscr{U}_\epsilon$ if and only if $\boldsymbol x$ belongs to the set
\begin{equation} \label{Vtilde}
\widetilde{\mathscr{V}}_\epsilon=\bigcup_{\boldsymbol p\in\mathbb{Z}^d}\big\{\boldsymbol x+2\pi\boldsymbol p:\boldsymbol x\in\mathscr{V}_\epsilon\big\}
\end{equation}
and formula (\ref{wrapping2}) help us to establish that for $j,\ell=1,\ldots,d$,
\begin{align*}
\int_{\mathscr{U}_\epsilon}(\Im s_j)(\Im s_\ell)\,d\rho_n({\boldsymbol s})&=\int_{\mathbb{T}^d}1_{\mathscr{U}_\epsilon}({\boldsymbol s})(\Im s_j)(\Im s_\ell)\,d\rho_n({\boldsymbol s}) \\
&=\int_{\mathbb{\mathbb{R}}^d}1_{\mathscr{U}_\epsilon}(e^{i\boldsymbol x})(\Im e^{ix_j})(\Im e^{ix_\ell})\,d\tau_n(\boldsymbol x) =\int_{\widetilde{\mathscr{V}}_\epsilon}\sin(x_j)\sin(x_\ell)\,d\tau_n(\boldsymbol x).
\end{align*}

Observe next that we have $\widetilde{\mathscr{V}}_\epsilon\backslash\mathscr{V}_\epsilon=\bigcup_{m=1}^d\mathscr{D}_{\epsilon m}$, where $\mathscr{D}_{\epsilon m}=\widetilde{\mathscr{V}}_\epsilon\cap\{|x_m|\geq\pi\}$, provided that $\epsilon<\pi$. If we temporarily impose the requirement $\sigma_m(\partial\mathscr{D}_{\epsilon m})=0$ for some $m\in\{1,\ldots,d\}$, then the weak convergence $\sigma_{nm}\Rightarrow\sigma_m$ implies that
\begin{align*}
\limsup_{n\to\infty}\int_{\mathscr{D}_{\epsilon m}}|\sin(x_j)\sin(x_\ell)|\,d\tau_n&=\limsup_{n\to\infty}
\int_{\mathscr{D}_{\epsilon m}}|\sin(x_j)\sin(x_\ell)|\cdot\frac{1+x_m^2}{x_m^2}\,d\sigma_{nm} \\
&=\int_{\mathscr{D}_{\epsilon m}}|\sin(x_j)\sin(x_\ell)|\cdot\frac{1+x_m^2}{x_m^2}\,d\sigma_m\mathop{\longrightarrow}\limits_{\epsilon\to0}0.
\end{align*} This, along with facts $x-\sin x=o(|x|^2)$ as $|x|\to0$ and $x_j^2\in L^1(\sigma_j)$,
leads to
\begin{align*}
\lim_{\epsilon\to0}\limsup_{n\to\infty}\int_{\mathscr{U}_\epsilon}(\Im s_j)(\Im s_\ell)\,d\rho_n({\boldsymbol s})&=\lim_{\epsilon\to0}\limsup_{n\to\infty}\int_{\mathscr{V}_\epsilon}\sin(x_j)\sin(x_\ell)\,d\tau_n(\boldsymbol x) \\
&=\lim_{\epsilon\to0}\limsup_{n\to\infty}\int_{\mathscr{V}_\epsilon}x_jx_\ell\,d\tau_n(\boldsymbol x).
\end{align*} The same arguments also elaborate the identity
\[\lim_{\epsilon\to0}\liminf_{n\to\infty}\int_{\mathscr{U}_\epsilon}(\Im s_j)(\Im s_\ell)\,d\rho_n({\boldsymbol s})=
\lim_{\epsilon\to0}\liminf_{n\to\infty}\int_{\mathscr{V}_\epsilon}x_jx_\ell\,d\tau_n(\boldsymbol x).\] Apparently, the selection of
$\epsilon$ does not vary the validity of these identities, and so we have established that $\rho_n$ generates the matrix $\mathbf{A}$ in \eqref{item:condiv} as well.
\end{proof}

Measures in $\mathscr{M}_{\mathbb{R}^d}^\mathbf{0}$ can be wrapped either clockwise or counterclockwise (cf. (\ref{wrapping1})) in all variables, and consequences, such as Proposition \ref{+implyX}, are not affected at all by this slight change. As a matter of fact, it is also the case when one wraps some variables counterclockwise and others clockwise. Without loss of generality, we shall use the simplest circumstance, the $2$-dimensional opposite wrapping map $W_2^\star\colon\mathbb{R}^2\to\mathbb{T}^2$, $(x_1,x_2)\mapsto(e^{ix_1},e^{-ix_2})$, to illustrate these features. The following result is merely an easy consequence of the continuous mapping theorem, the relations  $(\tau(W_2^\star)^{-1})(B)=(\tau W_2^{-1})(B^\star)=(\tau W_2^{-1})^\star(B)$ for any $B\in\mathscr{B}_{\mathbb{T}^2}$, and Proposition \ref{+implyX}.

\begin{prop} \label{opcond}
If $\{\rho_n\}$ and $\rho$ in $\mathscr{M}_{\mathbb{T}^2}^\mathbf{1}$ fulfill \emph{Condition \ref{condX2}}, then
\begin{enumerate} [$\qquad(1)$]
\item {$\rho_n^\star\Rightarrow_\mathbf{1}\rho^\star$, and}
\item {for any $\boldsymbol p=(p_1,p_2)\in\mathbb{Z}^2$, denoting
by $\boldsymbol p^\star=(p_1,-p_2)$, we have
\[\lim_{\epsilon\to0}\limsup_{n\to\infty}\int_{\mathscr{U}_\epsilon}
\big\langle\boldsymbol p,\Im\boldsymbol s\big\rangle^2d\rho_n^\star({\boldsymbol s})=Q(\boldsymbol p^\star)=\lim_{\epsilon\to0}\liminf_{n\to\infty}\int_{\mathscr{U}_\epsilon}
\big\langle\boldsymbol p,\Im\boldsymbol s\big\rangle^2d\rho_n^\star({\boldsymbol s}).\]}
\end{enumerate} Particularly, if $\{\tau_n\}$ and $\tau$ in $\mathscr{M}_{\mathbb{R}^2}^\mathbf{0}$ satisfy \emph{Condition \ref{cond+1}} \emph{(}or \emph{Condition \ref{cond+2}}\emph{)}, then statements \emph{(1)} and \emph{(2)} above apply to $\rho_n^\star=\tau_n(W_2^\star)^{-1}$ and $\rho^\star=\tau(W_2^\star)^{-1}$.
\end{prop}

We add one remark on the item (2) of the preceding proposition: if $Q(\boldsymbol p)=\langle\mathbf{A}\boldsymbol p,\boldsymbol p\rangle$, then $Q(\boldsymbol p^\star)=\langle\mathbf{A}\!^{\mathrm{op}}\boldsymbol p,\boldsymbol p\rangle$, where the $(i,j)$-entry of $\mathbf{A}\!^{\mathrm{op}}$ is $(-1)^{i+j}\mathbf{A}_{ij}$.

\section{Limit Theorems and Bi-free Multiplicative L\'{e}vy Triplet} \label{sec5}
\subsection{Bi-free multiplicative L\'evy-Khintchine representation} Thanks to Proposition \ref{equicondX}, one can correlate the quantity $L_{12}$ and measures $\lambda_j$ given in the formulas (\ref{bi-freeLK}) with the matrix $\mathbf{A}$ and measure $\rho\in\mathscr{M}_{\mathbb{T}^2}^\mathbf{1}$ determined by (\ref{lambdarho}), (\ref{LevyTcond}), and (\ref{constc}). Therefore, instead of working with the parametrization $(\boldsymbol \gamma, \lambda_1,\lambda_2,L_{12})$ for measures in $\mathcal{ID}(\bitimes)\cap\mathscr{P}_{\mathbb{T}^2}^\times$, one may take another parametrization $(\boldsymbol\gamma,\mathbf{A},\rho)$ (with the same $\boldsymbol \gamma$) having the following properties with $d=2$:
\begin{equation} \label{propTD}
\begin{split}
&\boldsymbol\gamma\in\mathbb{T}^d,\,\text{$\mathbf{A}$ is a positive semi-definite $d\times d$ symmetric matrix, and} \\
& \text{$\rho$ is a positive measure on $\mathbb{T}^d$ so that $\rho(\{\mathbf{1}\})=0$ and $\|\mathbf{1}-\Re\boldsymbol s\|\in L^1(\rho)$}.
\end{split}
\end{equation}

We shall refer to $(\boldsymbol\gamma,\mathbf{A},\rho)$ as the \emph{bi-free multiplicative L\'{e}vy triplet} of the corresponding measure in $\mathcal{ID}(\bitimes)\cap\mathscr{P}_{\mathbb{T}^2}^\times$ having (bi-)free $\Sigma$-transforms presented in (\ref{bi-freeLK}), and signify this measure by $\nu_{\bitimes}^{(\boldsymbol\gamma,\mathbf{A},\rho)}$ to comply with the correspondence. This triplet plays the role of the classical multiplicative L\'{e}vy triplet in \eqref{eq:classical_exponent}. We will clarify this in more details in Corollary \ref{cor:limit_free_classical}, where limit theorems between classical and bi-free multiplicative convolutions are examined and in Section \ref{sec7}, where the diagram \eqref{diagram} is shown.

Observe that a measure $\nu$ belongs to $\mathcal{ID}(\bitimes^\mathrm{op})\cap\mathscr{P}_{\mathbb{T}^2}^\times$ if and only if $\nu^\star\in\mathcal{ID}(\bitimes)\cap\mathscr{P}_{\mathbb{T}^2}^\times$ by (\ref{opdef}) and Theorem \ref{IDmoment}.
Thus, we shall denote by $\nu_{\bitimes^{\mathrm{op}}}^{(\boldsymbol\gamma,\mathbf{A},\rho)}$ the measure $\nu$ satisfying $\nu^\star=\nu_{\bitimes}^{(\boldsymbol\gamma^\star,\mathbf{A}\!^\mathrm{op},\rho^\star)}$ and refer to
$(\boldsymbol\gamma,\mathbf{A},\rho)$ as its \emph{opposite bi-free multiplicative L\'{e}vy triplet}. Passing to analytic transforms, we have
\[\Sigma_\nu^\mathrm{op}(z,w)=
\Sigma_{\nu_{\bitimes}^{(\boldsymbol\gamma^\star,\mathbf{A}\!^\mathrm{op},\rho^\star)}}(z,1/w),
\;\;\;\;(z,w)\in\mathbb{D}\times(\mathbb{T}\cup\{0\})^c.\]

In terms of notations introduced above, we reformulate the basic limit theorem \cite[Theorem 3.4]{HW2} on the bi-free multiplicative convolution, including statements for $\bitimes^{\rm op}$.

\begin{thm} \label{limitthmop}
Given an infinitesimal array $\{\nu_{nk}\}\subset\mathscr{P}_{\mathbb{T}^2}^\times$ and a sequence $\{\boldsymbol\xi_n\}\subset\mathbb{T}^2$, define $\boldsymbol\gamma_n$ as in \emph{(\ref{gamman})}.
The following are equivalent.
\begin{enumerate} [$\qquad(1)$]
\item {The sequence
\begin{equation} \label{limitmul}
\delta_{\boldsymbol\xi_n}\bitimet\nu_{n1}\bitimet\cdots\bitimet\nu_{nk}
\end{equation}
converges weakly to some $\nu_{\bitimes}\in\mathscr{P}_{\mathbb{T}^2}^\times$.}
\item {The sequence
\begin{equation} \label{limitmulop}
\delta_{\boldsymbol\xi_n}\bitimet^{\mathrm{op}}\nu_{n1}\bitimet^{\mathrm{op}}\cdots\bitimet^{\mathrm{op}}\nu_{nk}
\end{equation} converges weakly to some $\nu_{\bitimes^{\mathrm{op}}}\in\mathscr{P}_{\mathbb{T}^2}^\times$.}
\item {The measure $\rho_n=\sum_{k=1}^{k_n}\mathring\nu_{nk}$ satisfies
{\rm Condition \ref{condX2}}
{\rm(}or \emph{Condition \ref{condX1}}{\rm)} with $d=2$
and $\lim_{n\to\infty}\boldsymbol\gamma_n=\boldsymbol\gamma$ exists.}
\end{enumerate} If {\rm (1)--(3)} hold, then $\nu_{\bitimes}=\nu_{\bitimes}^{(\boldsymbol\gamma,\mathbf{A},\rho)}$ and $(\nu_{\bitimes^{\mathrm{op}}})^\star=\nu_{\bitimes}^{(\boldsymbol\gamma^\star,\mathbf{A}^{\!\!\mathrm{op}},\rho^\star)}$, where $\rho$ and $\mathbf{A}$ are as in {\rm Condition \ref{condX2}} and {\rm Proposition \ref{equicondX}}, respectively.
\end{thm}

\begin{proof} We only need to prove (2)$\Leftrightarrow$(3). With $\{\boldsymbol b_{nk}\}$ defined in (\ref{bnk}), the equality
\begin{align*}
\exp\left[i\int_{\mathscr{U}_\theta}(\arg\boldsymbol s)\,d\nu_{nk}^\star(\boldsymbol s)\right]=
\exp\left[i\int_{\mathscr{U}_\theta}(\arg s_1,-\arg s_2)\,d\nu_{nk}(\boldsymbol s)\right]=\boldsymbol b_{nk}^\star
\end{align*}
shows that for any Borel set $B$ on $\mathbb{T}^2$,
\[(\nu_{nk}^\star)^\circ(B)=\nu_{nk}^\star(\boldsymbol b_{nk}^\star B)=\nu_{nk}(\boldsymbol b_{nk}B^\star)=\mathring\nu_{nk}(B^\star)=(\mathring\nu_{nk})^\star(B).\] Since the operations $\star$ and $\circ$ acting on $\nu_{nk}$ are interchangeable in order, we adopt the notation $\mathring\nu_{nk}^\star$ instead of $(\nu_{nk}^\star)^\circ=(\mathring\nu_{nk})^\star$ if no confusion can arise.

According to (\ref{opdef}), item (2) holds if and only if
\[\delta_{\boldsymbol\xi_n^\star}\bitimet\nu_{n1}^\star\bitimet\cdots\bitimet\nu_{nk}^\star=
(\delta_{\boldsymbol\xi_n}\bitimet^{\mathrm{op}}\nu_{n1}\bitimet^{\mathrm{op}}\cdots
\bitimet^{\mathrm{op}}\nu_{nk})^\star\Rightarrow
(\nu_{\bitimes^{\mathrm{op}}})^\star.\] This happens if and only if Condition \ref{condX2} applies to the measure $\sum_{k=1}^n\mathring\nu_{nk}^\star=\left(\sum_{k=1}^n\mathring\nu_{nk}\right)^\star$ and the vector
\[\boldsymbol\gamma_n^\star=\boldsymbol\xi_n^\star\exp\left[i\sum_{k=1}^{k_n}\left(\arg\boldsymbol b_{nk}^\star+\int_{\mathbb{T}^2}(\Im\boldsymbol s)\,d
\mathring\nu_{nk}^\star(\boldsymbol s)\right)\right]\] has a limit by Theorem \ref{limitthmX}. Then Proposition \ref{opcond} yields the equivalent (2)$\Leftrightarrow$(3) and proves the last assertion.
\end{proof}

The goal of this section is to provide an alternative description for the $\Sigma$-transform of a measure in $\mathcal{ID}(\bitimes)\cap\mathscr{P}_{\mathbb{T}^2}^\times$ in terms of its bi-free multiplicative L\'evy triplets.  To achieve this, we need some basics. For any $p\in\mathbb{N}$, the function
\[\mathcal{K}_p(s)=\frac{s^p-1-ip\Im s}{1-\Re s}\]
is continuous on $\mathbb{T}$ and equal to $-p^2$ at $s=1$.

\begin{lem} \label{angleineq}
For any $p\in\mathbb{N}$, we have $\|\Im\mathcal{K}_p\|_\infty\leq p^3$ and
\[\int_{-\pi}^\pi\mathcal{K}_p(e^{i\theta})\,d\theta=-2p\pi.\]
\end{lem}

\begin{proof} In the following arguments, we shall make use of the basic formula
\begin{equation} \label{eq:cos_formula}
\frac{1-\cos(p\theta)}{1-\cos\theta}=e^{i(1-p)\theta}\sum_{j,k=0}^{p-1}e^{i(j+k)\theta}.
\end{equation}
Clearly, we have $\Im\mathcal{K}_1\equiv0$. If $\|\Im\mathcal{K}_p\|_\infty\leq p^3$ for some $p\geq2$, then for $s\neq1$, the inequality $|(1-\Re s^p)/(1-\Re s)|\leq p^2$ following from \eqref{eq:cos_formula} implies that
\[|\Im\mathcal{K}_{p+1}(s)|=\left|\Im s^p-\Im\mathcal{K}_p(s)+\frac{1-\Re s^p}{1-\Re s}\cdot\Im s\right|\leq1+p^3+p^2\leq(p+1)^3.\] By induction, this finishes the proof of the first assertion.
To prove the second assertion, it suffices to show
\[\int_{-\pi}^\pi\frac{1-\cos(p\theta)}{1-\cos\theta}\,d\theta=2p\pi,\] which can be easily obtained by using (\ref{eq:cos_formula}).
\end{proof}

Fix a measure $\nu\in\mathscr{P}_{\mathbb{T}^2}^\times\cap\mathcal{ID}(\boxtimes\boxtimes)$, and suppose that its (bi-)free $\Sigma$-transforms are given as in (\ref{bi-freeLK}). Due to the integral representations, both $u_1$ and $u_2$ are analytic in $\Omega=(\mathbb{C}\backslash\mathbb{T})\cup\{\infty\}$ and $u$ is analytic in $\Omega^2$. Hence the function
\[U_\nu(z,w)=\frac{zw}{1-zw}u(z,w)-\frac{z}{1-z}u_1(z)-\frac{w}{1-w}u_2(w)\] is analytic in $\Omega^2$. If $\nu\in\mathcal{ID}(\boxtimes\boxtimes^\mathrm{op})\cap\mathscr{P}_{\mathbb{T}^2}^\times$, then we define
\[U_\nu^\mathrm{op}(z,w)=U_{\nu^\star}(z,1/w),\] which is also an analytic function in $\Omega^2$.

When $\nu\in\mathcal{ID}(\boxtimes\boxtimes^\mathrm{op})\cap\mathscr{P}_{\mathbb{T}^2}^\times$, one can obtain an equivalent formula for $U_\nu$ in terms of the bi-free multiplicative L\'evy triplet, which we call the \emph{bi-free multiplicative L\'evy-Khintchine representation}. Note that  we acquire the following proof with the help of limit theorems, in spite of the algebraic nature of the statement. Also, it is simpler even though there exists an algebraic proof.

\begin{thm} \label{revisedthm}
Letting $\nu=\nu_{\bitimes}^{(\boldsymbol\gamma,\mathbf{A},\rho)}$, we have
\begin{equation} \label{revisedform}
U_\nu(z,w)=\frac{iz}{1-z}\arg\gamma_1+\frac{iw}{1-w}\arg\gamma_2-N_\nu(z,w)+P_\nu(z,w),
\end{equation}
where
\[N_\nu(z,w)=\frac{a_{11}}{2}\cdot\frac{z(1+z)}{(1-z)^2}+
\frac{a_{12}zw}{(1-z)(1-w)}+\frac{a_{22}}{2}\cdot\frac{w(1+w)}{(1-w)^2}\] and
\[P_\nu(z,w)=(1-z)(1-w)\sum_{\boldsymbol p=\boldsymbol0}^{\boldsymbol\infty}
\left[\int_{\mathbb{T}^2}\big({\boldsymbol s}^{\boldsymbol p}-1-i\langle\boldsymbol p,\Im{\boldsymbol s}\rangle\big)\,d\rho({\boldsymbol s})\right]\!z^{p_1}w^{p_2}.\] Further, letting $\widetilde{\nu}=\nu_{\bitimes^{\mathrm{op}}}^{(\boldsymbol\gamma,\mathbf{A},\rho)}$, we have
\[U_{\widetilde{\nu}}^{\mathrm{op}}(z,w)=U_{\nu_{\bitimes}^{(\boldsymbol\gamma^\star,\mathbf{A}\!^{\mathrm{op}},\rho^\star)}}(z,1/w).\]
\end{thm}

\begin{proof} First of all, using Remark \ref{urepre} and the function
\[f(z,w,\boldsymbol s)=\frac{zw(1-s_1)(1-s_2)}{(1-zs_1)(1-ws_2)}-\frac{z(1+zs_1)(1-\Re s_1)}{(1-z)(1-zs_1)}-\frac{w(1+ws_2)(1-\Re s_2)}{(1-w)(1-ws_2)},\] one can rewrite $U_\nu$ as
\[U_\nu(z,w)=\frac{iz}{1-z}\arg\gamma_1+\frac{iw}{1-w}\arg\gamma_2+\lim_{n\to\infty}\int_{\mathbb{T}^2}f(z,w,{\boldsymbol s})\,d\rho_n({\boldsymbol s}).\]

In the following, $r>0$ is taken so that $\rho(\partial\mathscr{U}_r)=0$. The continuity of $\boldsymbol s\mapsto f(z,w,\boldsymbol s)$ on $\mathbb{T}^2$ for any fixed $(z,w)\in\mathbb{D}^2$ and Proposition \ref{Portman} imply that
\[\lim_{n\to\infty}\int_{\mathbb{T}^2\backslash\mathscr{U}_r}f(z,w,{\boldsymbol s})\,d\rho_n({\boldsymbol s})=
\int_{\mathbb{T}^2\backslash\mathscr{U}_r}f(z,w,{\boldsymbol s})\,d\rho({\boldsymbol s}).\] Applying dominated convergence theorem to $f(z,w,\cdot)\in L^1(\rho)$, we arrive at
\[\lim_{r\to0}\lim_{n\to\infty}\int_{\mathbb{T}^2\backslash\mathscr{U}_r}f(z,w,{\boldsymbol s})\,d\rho_n({\boldsymbol s})=
\int_{\mathbb{T}^2}f(z,w,{\boldsymbol s})\,d\rho({\boldsymbol s}).\]

On the other hand, thanks to weak convergence $\lambda_{nj}=(1-\Re s_j)\rho_n\Rightarrow\lambda_j$, $j=1,2$, we see that for $\xi\in\mathbb{D}$,
\begin{align*}
\limsup_{n\to\infty}&\left|\int_{\mathscr{U}_r}\frac{1+\xi s_j}{1-\xi s_j}(1-\Re s_j)\,d\rho_n-\frac{a_{jj}}{2}\cdot\frac{1+\xi}{1-\xi}\right| \\
&\leq\limsup_{n\to\infty}
\left(\big|\lambda_{nj}(\mathscr{U}_r)-a_{jj}/2\big|\left|\frac{1+\xi}{1-\xi}\right|+
\int_{\mathscr{U}_r}\left|\frac{1+\xi s_j}{1-\xi s_j}-\frac{1+\xi}{1-\xi}\right|d\lambda_{nj}\right) \\
&\leq\frac{2}{(1-|\xi|)^2}\limsup_{n\to\infty}
\left(\big|\lambda_{nj}(\mathscr{U}_r)-a_{jj}/2\big|+\int_{\mathscr{U}_r}|1-s_j|\,d\lambda_{nj}\right) \\
&=\frac{2}{(1-|\xi|)^2}\left(
\big|\lambda_j(\mathscr{U}_r)-a_{jj}/2\big|+\int_{\mathscr{U}_r}|1-s_j|\,d\lambda_j\right)\mathop{\longrightarrow}\limits_{r\to0}0.
\end{align*} Similarly, one can show that
\[\lim_{r\to0}\lim_{n\to\infty}\int_{\mathbb{T}^2\backslash\mathscr{U}_r}\frac{(1-s_1)(1-s_2)}{(1-zs_1)(1-ws_2)}\,d\rho_n
=\int_{\mathbb{T}^2}\frac{(1-s_1)(1-s_2)}{(1-zs_1)(1-ws_2)}\,d\rho\] and
\[\lim_{r\to0}\limsup_{n\to\infty}\left|\int_{\mathscr{U}_r}\frac{(1-s_1)(1-s_2)}{(1-zs_1)(1-ws_2)}\,d\rho_n
+\frac{a_{12}}{(1-z)(1-w)}\right|=0.\] Next, we shall make use of the equation (\ref{decomeq}).
After some algebraic manipulations, we come to the result
\[\lim_{n\to\infty}\int_{\mathbb{T}^2}f(z,w,{\boldsymbol s})\,d\rho_n({\boldsymbol s})=-N(z,w)+(1-z)(1-w)\int_{\mathbb{T}^2}\widetilde{f}(z,w,{\boldsymbol s})\,d\rho({\boldsymbol s}),\] where
\[\widetilde{f}(z,w,\boldsymbol s)=\frac{1}{(1-zs_1)(1-ws_2)}-\frac{1}{(1-z)(1-w)}-
\frac{iz\Im s_1}{(1-z)^2(1-w)}-
\frac{iw\Im s_2}{(1-z)(1-w)^2}.\]

Lastly, the use of the power series expansion \[\frac{\xi_j}{(1-\xi_1)^2(1-\xi_2)}=\sum_{\boldsymbol p\geq\boldsymbol0}p_j\xi_1^{p_1}\xi_2^{p_2},\;\;\;\;\;\xi_1,\xi_2\in\mathbb{D},\]
allows us to get
\[\int_{\mathbb{T}^2}\widetilde{f}(z,w,{\boldsymbol s})\,d\rho({\boldsymbol s})=\int_{\mathbb{T}^2}\sum_{\boldsymbol p\geq\boldsymbol0}
\big({\boldsymbol s}^{\boldsymbol p}-1-i\langle\boldsymbol p,\Im{\boldsymbol s}\rangle\big)z^{p_1}w^{p_2}\,d\rho({\boldsymbol s}).\]
The operations of integration and summation performed above are interchangeable due to Lemma \ref{angleineq}. Indeed, one can utilize the uniform convergence of the summands to obtain
\begin{align*}
\int\sum_{\boldsymbol p\geq\boldsymbol0}(s_j^{p_j}-1-ip_j\Im s_j)z^{p_1}w^{p_2}\,d\rho&=\sum_{\boldsymbol p\geq\boldsymbol0}\int\mathcal{K}_{p_j}(s)\,d\lambda_j\;z^{p_1}w^{p_2} \\
&=\sum_{\boldsymbol p\geq\boldsymbol0}\int(s_j^{p_j}-1-ip_j\Im s_j)\,d\rho\;z^{p_1}w^{p_2}
\end{align*} and
\[\int\sum_{\boldsymbol p\geq\boldsymbol0}(s_1^{p_1}-1)(s_2^{p_2}-1)z^{p_1}w^{p_2}\,d\rho=
\sum_{\boldsymbol p\geq\boldsymbol0}\int(s_1^{p_1}-1)(s_2^{p_2}-1)\,d\rho\;z^{p_1}w^{p_2}.\]
Putting all these findings together yields the desired result.

According to the definition of $\widetilde{\nu}$, which is characterized by $(\widetilde{\nu})^\star=\nu_{\bitimes}^{(\boldsymbol\gamma^\star,\mathbf{A}\!^{\mathrm{op}},\rho^\star)}$, the last assertion follows from the definition of $U_\nu^\mathrm{op}$.
\end{proof}

Performing the power series expansion to $N_\nu(z,w)$ in Theorem \ref{revisedthm} further yields that
\[\frac{U_\nu(z,w)}{(1-z)(1-w)} = \sum_{\boldsymbol p=\boldsymbol0}^{\boldsymbol\infty}
\left[ i\langle \boldsymbol p, \arg \boldsymbol \gamma\rangle -\frac{1}{2}\langle \mathbf A \boldsymbol p,  \boldsymbol p\rangle+ \int_{\mathbb{T}^2}\big({\boldsymbol s}^{\boldsymbol p}-1-i\langle\boldsymbol p,\Im{\boldsymbol s}\rangle\big)\,d\rho({\boldsymbol s})\right]\!z^{p_1}w^{p_2},\]
which offers the generating series of the exponent of the characteristic function \eqref{eq:classical_exponent} of a measure in $\mathcal{ID}(\mathbb T^2,\circledast)\cap \mathscr{P}_{\mathbb T^2}^\times$. We shall investigate this consistency in more detail in Section 7.

\subsection{Limit theorems via wrapping transformation}
We next present the limit theorem through the wrapping transformation.

\begin{thm} \label{+implyXthm}
Let $(\boldsymbol v,\mathbf{A},\tau)$ be a triplet satisfying \eqref{Levy_triplet}, and let $\{\mu_{nk}\}\subset\mathscr{P}_{\mathbb{R}^2}$ be an infinitesimal triangular array and
$\{\boldsymbol v_n\}$ a sequence of vectors in $\mathbb{R}^2$.
If the sequence in \emph{(\ref{munbifree})} converges weakly to $\mu_{\biplus}^{(\boldsymbol v,\mathbf{A},\tau)}$, then the sequences in \emph{(\ref{limitmul})} and \emph{(\ref{limitmulop})} generated by $\nu_{nk}=\mu_{nk}W^{-1}$ and $\boldsymbol\xi_n=e^{i\boldsymbol v_n}$ converge weakly to $\nu_{\bitimes}^{(\boldsymbol\gamma,\mathbf{A},\rho)}$ and
$\nu_{\bitimes^{\mathrm{op}}}^{(\boldsymbol\gamma,\mathbf{A},\rho)}$, respectively, where
\begin{equation} \label{rhotau}
\rho=1_{\mathbb{T}^2\backslash\{\mathbf{1}\}}(\tau W^{-1})
\end{equation} and
\begin{equation} \label{gammav}
\boldsymbol\gamma=\exp\left[
i\boldsymbol v+i\int_{\mathbb{R}^2}\left(\sin(\boldsymbol x)-\frac{\boldsymbol x}{1+\|\boldsymbol x\|^2}\right)d\tau(\boldsymbol x)\right].
\end{equation}
\end{thm}

\begin{proof} Before carrying out the main proof, let us record some properties instantly inferred from the hypotheses for the later utilization. Because the index $n$ goes to infinity ultimately, it is always big enough whenever mentioned in the proof.

Firstly, observe that $\nu_{nk}$ belongs to $\mathscr{P}_{\mathbb{T}^2}^\times$ and the vector
\begin{equation} \label{newtheta}
\boldsymbol\theta_{nk}=\sum_{\boldsymbol p\in\mathbb{Z}^d\backslash\{\mathbf{0}\}}\int_{\mathscr{V}_\theta}\boldsymbol x\;d\mu_{nk}(\boldsymbol x+2\pi\boldsymbol p)
\end{equation}
satisfies $\lim_{n\to\infty}\max_k\|\boldsymbol\theta_{nk}\|=0$ by the infinitesimality of $\{\mu_{nk}\}$. Secondly, following the notation in (\ref{Vtilde}), an application of (\ref{wrapping2}) gives that
\begin{align*}
{\boldsymbol v}_{nk}+\boldsymbol\theta_{nk}&=
\sum_{\boldsymbol p\in\mathbb{Z}^d}\int_{\mathscr{V}_\theta}\boldsymbol x\;d\mu_{nk}(\boldsymbol x+2\pi\boldsymbol p) \\
&=\int_{\mathbb{R}^d}1_{\{e^{i\boldsymbol x}:\boldsymbol x\in\widetilde{\mathscr{V}}_\theta\}}(e^{i\boldsymbol x})\arg\big(e^{i\boldsymbol x}\big)
\,d\mu_{nk}(\boldsymbol x) \\
&=\int_{\mathbb{T}^d}1_{\{{\boldsymbol s}:\|\arg{\boldsymbol s}\|<\theta\}}({\boldsymbol s})\arg({\boldsymbol s})\,d\nu_{nk}(\boldsymbol s) =\int_{\mathscr{U}_\theta}\arg({\boldsymbol s})\,d\nu_{nk}({\boldsymbol s}).
\end{align*} This result provides us with the relations
$\arg{\boldsymbol b}_{nk}={\boldsymbol v}_{nk}+\boldsymbol\theta_{nk}$ and
\[d\mathring\nu_{nk}({\boldsymbol s})=d(\mathring\mu_{nk}W^{-1})(e^{i\boldsymbol\theta_{nk}}{\boldsymbol s})\] as for any $B\in\mathscr{B}_{\mathbb{T}^2}$, we have
$(\mathring\mu_{nk}W^{-1})(B)=\mu_{nk}(\{e^{i\boldsymbol x}\in e^{i{\boldsymbol v}_{nk}}B\})=\nu_{nk}(e^{i{\boldsymbol v}_{nk}}B)=\mathring\nu_{nk}(e^{-\boldsymbol\theta_{nk}}B)$.

Except for the beforehand mentioned results, the array $\{\boldsymbol\theta_{nk}\}$ in (\ref{newtheta}) also fulfills the condition in (\ref{perturbation}), which will play a dominant role in our arguments. Its proof, provided below, is based on the convergence
$\tau_n=\sum_k\mathring\mu_{nk}\Rightarrow_\mathbf{0}\tau$ and some estimates. For convenience, denote $\boldsymbol\theta_{nk}=(\theta_{nk1},\theta_{nk2})$ and ${\boldsymbol v}_{nk}=(v_{nk1},v_{nk2})$, and view
\[\varrho_{nk}(\cdot)=\sum_{\boldsymbol p\in\mathbb{Z}^d\backslash\{\mathbf{0}\}}\mathring\mu_{nk}(\cdot+2\pi\boldsymbol p)\] as a positive Borel measure on the closure of $\mathscr{V}_{2\theta}$. The infinitesimality of $\{\mathring\mu_{nk}\}$ indicates that we have $\lim_{n\to\infty}\max_{1\leq k\leq k_n}\varrho_{nk}(\mathscr{V}_{2\theta})=0$ and (recall from (\ref{thetarange}) that $\theta\in(0,1)$ is fixed throughout our discussions)
\[\varrho_{nk}(\mathscr{V}_\theta-\boldsymbol v_{nk})\leq\sum_{\boldsymbol p\in\mathbb{Z}^d\backslash\{\mathbf{0}\}}\mathring\mu_{nk}(\mathscr{V}_{2\theta}+2\pi\boldsymbol p)=\mathring\mu_{nk}(\widetilde{\mathscr{V}}_{2\theta}\backslash\mathscr{V}_{2\theta}).\]
This, together with Cauchy-Schwarz inequality, enables us to obtain
\begin{align*}
\sum_{k=1}^{k_n}\theta_{nkj}^2&=\sum_{k=1}^{k_n}\left(\int_{\mathscr{V}_\theta-\boldsymbol v_{nk}}(x_j+v_{nkj})
\;d\varrho_{nk}(\boldsymbol x)\right)^2 \\
&\leq\sum_{k=1}^{k_n}\varrho_{nk}(\mathscr{V}_\theta-\boldsymbol v_{nk})\cdot\int_{\mathscr{V}_\theta-\boldsymbol v_{nk}}(x_j+v_{nkj})^2
\;d\varrho_{nk}(\boldsymbol x) \\
&\leq\theta^2\left(\max_{1\leq k\leq k_n}\varrho_{nk}(\mathscr{V}_{2\theta})\right)\tau_n
\left(\widetilde{\mathscr{V}}_{2\theta}\backslash\mathscr{V}_{2\theta}
\right).
\end{align*} Since $\widetilde{\mathscr{V}}_{2\theta}\backslash\mathscr{V}_{2\theta}$ is bounded away from the origin $\mathbf{1}$ of $\mathbb{T}^2$, $\tau_n\Rightarrow_\mathbf{0}\tau$ leads to $\limsup_n\tau_n(\widetilde{\mathscr{V}}_{2\theta}\backslash\mathscr{V}_{2\theta})<\infty$. Thus, we are able to conclude that $\sum_{k=1}^{k_n}\theta_{nkj}^2\to0$ as $n\to\infty$, yielding (\ref{perturbation}) by the inequality $1-\cos x\leq x^2/2$ on $\mathbb{R}$.

After these preparations, we are ready to present the proof of the theorem. Since (\ref{munbifree}) converges weakly, $\tau_n$ meets Condition \ref{cond+2}, and thus $\rho_n=\tau_n W^{-1}$ satisfies
Condition \ref{condX1} according to Proposition \ref{+implyX}. Then Proposition \ref{perturbthm} consequently
yields that Condition \ref{condX1} also applies to $\widetilde{\rho}_n=\sum_{k=1}^{k_n}\mathring\nu_{nk}$.

To finish the proof, we just need to verify (\ref{gamma}) due to Theorem \ref{limitthmop}. The existence of the limit in (\ref{limitv}) implies that the vector
\[\mathbf{E}_n=i\left[\boldsymbol v_n+\sum_{k=1}^{k_n}\left({\boldsymbol v}_{nk}
+\int_{\mathbb{R}^2}\sin(\boldsymbol x)\;d\mathring\mu_{nk}(\boldsymbol x)\right)\right]\] also has
a limit when $n\to\infty$. Indeed, the limit $-i\lim_{n\to\infty}\mathbf{E}_n$ disintegrates into the sum of that in (\ref{limitv}) and
\[\lim_{n\to\infty}\sum_{k=1}^{k_n}\int_{\mathbb{R}^2}\left(\sin(\boldsymbol x)-\frac{\boldsymbol x}{1+\|\boldsymbol x\|^2}\right)d\mathring\mu_{nk}(\boldsymbol x)=\int_{\mathbb{R}^2}\left(\sin(\boldsymbol x)-\frac{\boldsymbol x}{1+\|\boldsymbol x\|^2}\right)d\tau(\boldsymbol x).\] The validity of the equality displayed above is just because of that the integrand is $O(\|\boldsymbol x\|^3)$ as $\|\boldsymbol x\|\to0$ and the function $\min\{1,\|\boldsymbol x\|^2\}$ is $\tau$-integrable.

In order to go further, we analyze the difference
\[\left(\arg{\boldsymbol b}_{nk}+\int_{\mathbb{T}^2}(\Im{\boldsymbol s})\,d\mathring\nu_{nk}({\boldsymbol s})\right)-
\left({\boldsymbol v}_{nk}+\int_{\mathbb{R}^2}\sin(\boldsymbol x)\,d\mathring\mu_{nk}(\boldsymbol x)\right),\] which, along with the help of equation
\[\int_{\mathbb{R}^2}\sin(\boldsymbol x)\,d\mathring\mu_{nk}(\boldsymbol x)=
\int_{\mathbb{T}^2}\Im(e^{i\boldsymbol\theta_{nk}}{\boldsymbol s})\,d\mathring\nu_{nk}({\boldsymbol s}),\]
becomes
\begin{equation} \label{difference}
(\boldsymbol\theta_{nk}-\sin\boldsymbol\theta_{nk})+\sin(\boldsymbol\theta_{nk})\int_{\mathbb{T}^2}(1-\Re{\boldsymbol s})\,d\mathring\nu_{nk}({\boldsymbol s})
+(1-\cos\boldsymbol\theta_{nk})\int_{\mathbb{T}^2}(\Im{\boldsymbol s})\,d\mathring\nu_{nk}({\boldsymbol s}).
\end{equation}
Using the elementary inequality
\begin{equation} \label{basicineq}
|x-\sin x|\leq1-\cos x,\;\;\;\;\;|x|\leq\pi/4,
\end{equation} we see from the established result that
\[\lim_{n\to\infty}\sum_{k=1}^{k_n}|\theta_{nkj}-\sin\theta_{nkj}|\leq\lim_{n\to\infty}\sum_{k=1}^{k_n}(1-\cos\theta_{nkj})=0.\] For the second term in (\ref{difference}), $\lambda_{nj}=(1-\Re s_j)\widetilde{\rho}_n\Rightarrow\lambda_j\in\mathscr{M}_{\mathbb{T}^2}$ yields that
\[\sum_{k=1}^{k_n}\left|\sin(\theta_{nkj})\int_{\mathbb{T}^2}(1-\Re s_j)\,d\mathring\nu_{nk}(\boldsymbol s)\right|\leq\left(\max_{1\leq k\leq k_n}|\sin\theta_{nkj}|\right)\lambda_{nj}(\mathbb{T}^2)\mathop{\longrightarrow}\limits_{n\to\infty}0.\] As for the last term, we then have
\begin{align*}
\sum_{k=1}^{k_n}(1-\cos\theta_{nkj})\left|\int_{\mathbb{T}^2}(\Im s_j)\,d\mathring\nu_{nk}({\boldsymbol s})\right|\leq\sum_{k=1}^{k_n}(1-\cos\theta_{nkj})\mathop{\longrightarrow}\limits_{n\to\infty}0.
\end{align*} Consequently, we have arrived at that the limit in (\ref{gamma}) exists and equals the vector in (\ref{gammav}). This finishes the proof.
\end{proof}

The employment of the wrapping limit theorem with $\boldsymbol v_n=\mathbf{0}$ gives the following identically distributed limit theorem, which is the bi-free version of \cite[Theorem 3.9]{Ceb16}.

\begin{cor} \label{identicalX1}
Let $(\boldsymbol v,\mathbf{A},\tau)$ be a triplet satisfying \eqref{Levy_triplet}, $\{\mu_n\}$ a sequence in $\mathscr{P}_{\mathbb{R}^2}$, and $\{k_n\}$ a strictly increasing sequence in $\mathbb{N}$.
If $\mu_n^{\biplus k_n}\Rightarrow\mu_{\biplus}^{(\boldsymbol v,\mathbf{A},\tau)}$, then $(\mu_n W^{-1})^{\bitimes k_n}\Rightarrow\nu_{\bitimes}^{(\boldsymbol\gamma,\mathbf{A},\rho)}$ and $(\mu_n W^{-1})^{\bitimes^{\mathrm{op}}k_n}\Rightarrow\nu_{\bitimes^{\mathrm{op}}}^{(\boldsymbol\gamma,\mathbf{A},\rho)}$, where $\boldsymbol\gamma$ and $\rho$ are as in \emph{Theorem \ref{+implyXthm}}.
\end{cor}

\begin{exam} \label{BiGaussian}
Given a $2\times2$ real matrix $\mathbf{A}=(a_{ij})\geq0$ with $a_{11}\geq a_{22}>0$, consider planar probability measures
\[\mu_n=\frac{1}{4}(\delta_{\boldsymbol\alpha_n}+\delta_{-\boldsymbol\alpha_n}+\delta_{\boldsymbol\beta_n}+\delta_{-\boldsymbol\beta_n}),\]
where $\boldsymbol\alpha_n=(\sqrt{2\det\mathbf{A}},0)/\sqrt{na_{22}}$ and $\boldsymbol\beta_n=(\sqrt{2}a_{12},\sqrt{2}a_{22})/\sqrt{na_{22}}$. Clearly, $\mathring\mu_n=\mu_n$ for all $n$ and $\tau_n:=n\mu_n\Rightarrow_\mathbf{0}0$ as $n\to\infty$. Furthermore, for any $\theta>0$, if $n$ is large enough, then
\[\int_{\mathscr{V}_\theta}x_j^2\,d\tau_n(\boldsymbol x)=a_{jj}\;\;\;\;\;\;\mathrm{and}\;\;\;\;\;
\int_{\mathscr{V}_\theta}x_1x_2\,d\tau_n(\boldsymbol x)=a_{12}.\] Hence the identically distributed limit theorem introduced in Section \ref{LTcond} indicates that $\mu_n^{\biplus n}$ converges weakly to $\mu_{\biplus}^{(\mathbf{0},\mathbf{A},0)}$, which is known as the bi-free Gaussian distribution with bi-free L\'{e}vy triplet $(\mathbf{0},\mathbf{A},0)$. For the measures
\[\nu_n=\mu_n W^{-1}=\frac{1}{4}(\delta_{e^{i\boldsymbol\alpha_n}}+\delta_{e^{-i\boldsymbol\alpha_n}}+
\delta_{e^{i\boldsymbol\beta_n}}+\delta_{e^{-i\boldsymbol\beta_n}})\in\mathscr{P}_{\mathbb{T}^2},\] a direct verification or an application of Corollary \ref{identicalX1} shows that $\nu_n^{\bitimes n}\Rightarrow\nu_{\bitimes}^{(\mathbf{1},\mathbf{A},0)}$ and
$\nu_n^{\bitimes^{\mathrm{op}}n}\Rightarrow\nu_{\bitimes^{\mathrm{op}}}^{(\mathbf{1},\mathbf{A},0)}$. Analogically, $\nu_{\bitimes}=\nu_{\bitimes}^{(\mathbf{1},\mathbf{A},0)}$ is called the \emph{bi-free multiplicative Gaussian distribution} with L\'{e}vy triplet $(\mathbf{1},\mathbf{A},0)$. Note that in this case the component $P_{\nu_{\bitimes}}$ in the representation (\ref{revisedform}), called the bi-free multiplicative compound Poisson part (see Example \ref{BiPoisson} below), vanishes.
\end{exam}

\begin{exam} \label{BiPoisson}
Given any $r>0$ and $\mu\in\mathscr{P}_{\mathbb{R}^2}$, let
\[\mu_n=(1-r/n)\delta_\mathbf{0}+r/n\mu,\] $\tau_n=n\mu_n$, and $\tau=r1_{\mathbb{R}^2\backslash\{\mathbf{0}\}}\mu$. A straightforward verification reveals that Condition \ref{cond+2} applies to $\tau_n$, $\tau$, and $Q\equiv0$. Hence \cite[Theorem 5.6]{HHW} shows that $\mu_n^{\biplus n}$ converges weakly to the so-called bi-free compound Poisson distribution $\mu_{\biplus}^{(\boldsymbol v,\mathbf{0},\tau)}$ with rate $r$ and jump distribution $\mu$, where
\[\boldsymbol v=r\cdot\int_{\mathbb{R}^2}\frac{\boldsymbol x}{1+\|\boldsymbol x\|^2}\,d\mu(\boldsymbol x).\] Applying Corollary \ref{identicalX1} shows that
\[\big((1-r/n)\delta_\mathbf{1}+r/n(\mu W^{-1})\big)^{\bitimes n}\Rightarrow\nu_{\bitimes}^{(e^{i\boldsymbol u},\mathbf{0},\rho)},\] as well as $(\mu_n W^{-1})^{\bitimes^{\mathrm{op}}n}\Rightarrow\nu_{\bitimes^{\mathrm{op}}}^{(e^{i\boldsymbol u},\mathbf{0},\rho)}$, where $\rho=r1_{\mathbb{T}^2\backslash\{\mathbf{1}\}}(\mu W^{-1})$ and
\[\boldsymbol u=r\int_{\mathbb{R}^2}\sin(\boldsymbol x)\,d\mu(\boldsymbol x).\] Analogous to the planar case, we refer to measures of the form
$\nu_\bitimes=\nu_{\bitimes}^{(e^{i\boldsymbol u},\mathbf{0},r\nu)}$, where $r>0$, $\nu\in\mathscr{P}_{\mathbb{T}^2}$ with $\nu(\{\mathbf{1}\})=0$, and $\boldsymbol u=r\int\Im\boldsymbol sd\nu(\boldsymbol s)$ as
the \emph{bi-free multiplicative compound Poisson distribution} with rate $r$ and jump distribution $\nu$. In (\ref{revisedform}), the bi-free Gaussian component $N_{\nu_\bitimes}\equiv0$.
\end{exam}

\subsection{Limit Theorems for Identically Distributed Case}
The following is a special case of the limit theorem in the context of identically distributed random vectors on the bi-torus.

\begin{prop} \label{identicalX2}
Let $\rho_n=k_n\nu_n$, where $\{\nu_n\}\subset\mathscr{P}_{\mathbb{T}^2}^\times$ and $\{k_n\}\subset\mathbb{N}$ with $k_1<k_2<\cdots$. If $\rho_n$ satisfies \emph{Condition \ref{condX1}} \emph{(}or \emph{Condition \ref{condX2}}\emph{)} and the limit
\[\boldsymbol v=\lim_{n\to\infty}\int_{\mathbb{T}^2}(\Im\boldsymbol\xi)\,d\rho_n(\boldsymbol\xi)\] exists, then
$\nu_n^{\bitimes k_n}\Rightarrow\nu_{\bitimes}^{(e^{i\boldsymbol v},\mathbf{A},\rho)}$ and $\nu_n^{\bitimes^{\mathrm{op}}k_n}\Rightarrow\nu_{\bitimes^{\mathrm{op}}}^{(e^{i\boldsymbol v},\mathbf{A},\rho)}$, where $\rho$ and $\mathbf{A}$ are as in {\rm Condition \ref{condX2}} and {\rm Proposition \ref{equicondX}}, respectively.
\end{prop}

\begin{proof} Let $h\colon\mathbb{T}^2\to(-\pi,\pi]^2$ be the inverse of the wrapping map $W(\boldsymbol x)=e^{i\boldsymbol x}$ restricted to $(-\pi,\pi]^2$, namely,
$h(\boldsymbol\xi)=\arg\boldsymbol\xi$. Further let $\mu_n\in\mathscr{P}_{\mathbb{R}^2}$ and $\tau\in\mathscr{M}_{\mathbb{R}^2}^\mathbf{0}$ be supported on $[-\pi,\pi]^2$ so that $\mu_n=\nu_n h^{-1}$ and $\tau=\rho h^{-1}$. Then $\nu_n=\mu_n W^{-1}$, and $\tau_n=\rho_n h^{-1}\Rightarrow_\mathbf{0}\tau$ by the continuous mapping theorem. Also the property (\ref{LevyTcond}) and the formula (\ref{wrapping2}) show that $\min\{1,\|\boldsymbol x\|^2\}\in L^1(\tau)$. One can utilize (\ref{basicineq}) to justify
\[\lim_{\epsilon\to0}\liminf_{n\to\infty}\int_{\mathscr{U}_\epsilon}(\arg s_j)(\arg s_\ell)\,d\rho_n(\mathbf{s})=
\lim_{\epsilon\to0}\limsup_{n\to\infty}\int_{\mathscr{U}_\epsilon}(\arg s_j)(\arg s_\ell)\,d\rho_n(\mathbf{s}).\] On the other hand, the change-of-variables formula (\ref{wrapping2}) gives the equation
\[\int_{\mathscr{V}_\epsilon}x_jx_\ell\,d\tau_n(\boldsymbol x)=\int_{\mathscr{U}_\epsilon}(\arg s_j)(\arg s_\ell)\,d\rho_n(\mathbf{s}),\] which implies that $\tau_n$ satisfies \eqref{item:condIV} of Condition \ref{cond+2}. Ultimately, observe that
\[\int_{\mathbb{R}^2}\frac{\boldsymbol x}{1+\|\boldsymbol x\|^2}\,d\tau_n(\boldsymbol x)=
\int_{\mathbb{T}^2}(\Im\boldsymbol s)\,d\rho_n(\boldsymbol s)
+\int_{\mathbb{R}^2}\left(\frac{\boldsymbol x}{1+\|\boldsymbol x\|^2}-\sin(\boldsymbol x)\right)d\tau_n(\boldsymbol x)\] has a limit when $n\to\infty$ owing to $\boldsymbol x/(1+\|\boldsymbol x\|^2)-\sin(\boldsymbol x)= O(\|\boldsymbol x\|^3)$ as $\|\boldsymbol x\|\to0$ and $\min\{1,\|\boldsymbol x\|^2\}\in L^1(\tau)$. Thus, $\mu_n^{\boxplus\boxplus k_n}\Rightarrow\mu_{\boxplus\boxplus}^{(\boldsymbol v,\mathbf{A},\tau)}$ by \cite[Theorem 5.6]{HHW}, and so we accomplish the proof by Corollary \ref{identicalX1}.
\end{proof}

We shall also consider the rotated probabilities
\[d\widetilde{\nu}_n(\boldsymbol s)=d\nu_n(\boldsymbol\omega_n\boldsymbol s)\] associated with a sequence $\{\nu_n\}\subset\mathscr{P}_{\mathbb{T}^2}^\times$, where $\boldsymbol\omega_n=(\omega_{n1},\omega_{n2})\in\mathbb{T}^2$ has components
\[\omega_{nj}=\int_{\mathbb{T}^2}s_j\,d\nu_n(\boldsymbol s)\bigg/\left|\int_{\mathbb{T}^2}s_j\,d\nu_n(\boldsymbol s)\right|\]
Through this sort of rotated distributions, we next present the bi-freely identically distributed limit theorem,  which is the bi-free analog of \cite[Proposition 3.6]{Ceb16}.

\begin{thm} \label{identicalX3}
The following are equivalent for a sequence $\{\nu_n\}$ in $\mathscr{P}_{\mathbb{T}^2}^\times$ and a strictly increasing sequence $\{k_n\}$ in $\mathbb{N}$.
\begin{enumerate} [$\quad(1)$]
\item The sequence $\nu_n^{\bitimes k_n}$ converges weakly to
 some $\nu_{\bitimes} \in \mathscr{P}_{\mathbb T^2}^\times$.
\item The sequence $\nu_n^{\bitimes^{\mathrm{op}}k_n}$
converges weakly to some $\nu_{\bitimes^{\mathrm{op}}} \in \mathscr{P}_{\mathbb T^2}^\times$.
\item {\emph{Condition \ref{condX2}} \emph{(}or \emph{Condition \ref{condX1})}
holds for $\rho_n=k_n\widetilde{\nu}_n$ and the following limit exists in $\mathbb T^2$:
\[\boldsymbol\gamma=\lim_{n\to\infty}\big(\omega_{n1}^{k_n},\omega_{n2}^{k_n}\big).\]}
\end{enumerate}
If {\rm(1)--(3)} hold, then $\nu_{\bitimes}=\nu_{\bitimes}^{(\boldsymbol\gamma,\mathbf{A},\rho)}$ and $(\nu_{\bitimes^{\mathrm{op}}})^\star=\nu_{\bitimes}^{(\boldsymbol\gamma^\star,\mathbf{A}^{\!\!\mathrm{op}},\rho^\star)}$, where $\rho$ and $\mathbf{A}$ are respectively as in {\rm Condition \ref{condX2}} and {\rm Proposition \ref{equicondX}}.
\end{thm}

\begin{proof} Only the equivalence (1)$\Leftrightarrow$(3) needs a proof, which relies on Proposition \ref{perturbthm}. First of all, the weak convergence of $\nu_n^{\bitimes k_n}$ to $\nu\in\mathscr{P}_{\mathbb{T}^2}^\times$
yields that $\widetilde{\nu}_n\Rightarrow\delta_{(1,1)}$. Indeed,
$m_{1,0}(\nu_n)^{k_n}=[\Sigma_{\nu_n^{(1)}}(0)]^{-k_n}\to[\Sigma_{\nu^{(1)}}(0)]^{-1}=m_{1,0}(\nu)$ shows that $\omega_{n1}^{k_n}\to m_{1,0}(\nu)/|m_{1,0}(\nu)|=\omega_1$. Since
\[[\Sigma_{\widetilde{\nu}_n^{(1)}}(z)]^{k_n}=\omega_{n1}^{k_n}[
\Sigma_{\nu_n^{(1)}}(z)]^{k_n}\to\omega_1\Sigma_{\nu^{(1)}}(z)=\Sigma_{\widetilde{\nu}^{(1)}}(z)\]
uniformly for $z$ in a neighborhood of zero by \cite[Proposition 2.9]{BerVoicu92}, it follows from $\Sigma_{\widetilde{\nu}^{(1)}}(0)=1$ that $\widetilde{\nu}_n^{(1)}\Rightarrow\delta_1$. In the same vein, one can obtain $\widetilde{\nu}_n^{(2)}\Rightarrow\delta_1$, giving the desired weak convergence.
On other hand, the $\mathscr{M}_{\mathbb{T}^2}^\mathbf{1}$-weak convergence of $\rho_n=k_n\widetilde{\nu}_n$ also implies $\widetilde{\nu}_n\Rightarrow\delta_\mathbf{1}$. In other words, $\widetilde{\nu}_n$ is infinitesimal if assertion (1) or (3) holds.

Write $\nu_n^{\bitimes k_n}=\delta_{\boldsymbol\xi_n}\bitimet\widetilde{\nu}_n^{\bitimes k_n}$
and consider measures
$d\mathring{\widetilde{\nu}}_n(\boldsymbol s)=d\widetilde{\nu}_n(\widetilde{\boldsymbol b}_n\boldsymbol s)$, where $\boldsymbol\xi_n=\boldsymbol\omega_n^{k_n}$ and
\[\widetilde{\boldsymbol b}_n=\exp\left[i\int_{\mathscr{U}_\theta}(\arg\boldsymbol s)\,d\widetilde{\nu}_n(\boldsymbol s)\right].\] Then as indicated in Theorem \ref{limitthmop}, assertion (1) holds if and only if
$\rho_n'=k_n\mathring{\widetilde{\nu}}_n$ satisfies Condition \ref{condX1} and $\boldsymbol\gamma_n=\boldsymbol\xi_n\exp(i\mathbf{E}_n)$ has a finite limit, where
\[\mathbf{E}_n=k_n\left[\arg\widetilde{\boldsymbol b}_n+\int_{\mathbb{T}^2}(\Im\boldsymbol s)\,d\mathring{\widetilde{\nu}}_n(\boldsymbol s)\right].\] The infinitesimality of $\widetilde{\nu}_n$ reveals that $\boldsymbol\theta_n=(\theta_{n1},\theta_{n2})\to\mathbf{0}$ as $n\to\infty$, where
\begin{equation} \label{thetanj}
\theta_{nj}=\arg\widetilde{b}_{nj}=\int_{\mathscr{U}_\theta}\arg s_j\,d\widetilde{\nu}_n(\boldsymbol s).
\end{equation} This simple fact will be often utilized in the following proof, and all the indices $n$ considered below are sufficiently large.
With a view toward applying Proposition \ref{perturbthm} to $\rho_n$ and $\rho_n'$, we shall prove that $\lim_{n\to\infty}k_n\|\boldsymbol\theta_n\|^2=0$.

Now, we argue that $\rho_n'=k_n\mathring{\widetilde{\nu}}_n$ satisfies Condition \ref{condX1} if the same condition (or, equivalently, Condition \ref{condX2} by Proposition \ref{equicondX}) applies to $\rho_n=k_n\widetilde{\nu}_n$. Let $\lambda_{nj}=(1-\Re s_j)\rho_n$. Using the fact
\begin{equation} \label{meanzero}
\int_{\mathbb{T}^2}(\Im s_j)\,d\widetilde{\nu}_n(\boldsymbol s)=0,\;\;\;\;\;j=1,2,
\end{equation} we have
\begin{align*}
k_n\theta_{nj}&=k_n\int_{\mathscr{U}_\theta}(\arg s_j)\,d\widetilde{\nu}_n(\boldsymbol s)-
k_n\int_{\mathbb{T}^2}(\Im s_j)\,d\widetilde{\nu}_n(\boldsymbol s) \\
&=\int_{\mathbb{T}^2}\frac{\arg s_j-\Im s_j}{1-\Re s_j}\,d\lambda_{nj}(\boldsymbol s)-\int_{\mathbb{T}^2\backslash\mathscr{U}_\theta}(\arg s_j)\,d\rho_n(\boldsymbol s).
\end{align*} Then the continuity of the function $s\mapsto(\arg s-\Im s)/(1-\Re s)$ on $\mathbb{T}$ implies that $\limsup_{n\to\infty}k_n|\theta_{nj}|<\infty$, and so $\lim_{n\to\infty}k_n\|\boldsymbol\theta_n\|^2=0$. Thus, $\rho_n'$ meets Condition \ref{condX1} by the relation $d\rho_n'(\boldsymbol s)=d\rho_n(e^{i\boldsymbol\theta_n}\boldsymbol s)$ and Proposition \ref{perturbthm}.

Conversely, suppose that $\rho_n'$ satisfies Condition \ref{condX1}. We first rewrite (\ref{thetanj}) as
\[\arg\widetilde{b}_{nj}=\int_{\widetilde{\boldsymbol b}_n^{-1}\mathscr{U}_\theta}(\arg s_j+\arg\widetilde{b}_{nj})\,d\mathring{\widetilde{\nu}}_n(\boldsymbol s).\] On the other hand, the integral in (\ref{meanzero}) can be decomposed into the sum
\[\Im\widetilde{b}_{nj}-(\Im\widetilde{b}_{nj})\!\!\int_{\mathbb{T}^2}(1-\Re s_j)d\mathring{\widetilde{\nu}}_n(\boldsymbol s)-(1-\Re\widetilde{b}_{nj})\!\!\int_{\mathbb{T}^2}(\Im s_j)d\mathring{\widetilde{\nu}}_n(\boldsymbol s)+\int_{\mathbb{T}^2}(\Im s_j)d\mathring{\widetilde{\nu}}_n(\boldsymbol s).\]
Since $\widetilde{b}_{nj}=\cos\theta_{nj}+i\sin\theta_{nj}$, some simple calculations allow us to obtain
\[\theta_{nj}=\arg \widetilde b_{nj}-\int_{\mathbb{T}^2}(\Im s_j)\,d\widetilde{\nu}_n(\boldsymbol s)=B_{nj}+R_{nj},\] where
\begin{align*}
B_{nj}=-&\theta_{nj}\mathring{\widetilde{\nu}}_n
(\mathbb{T}^2\backslash\widetilde{\boldsymbol b}_n^{-1}\mathscr{U}_\theta)-\int_{\mathbb{T}^2\backslash\widetilde{\boldsymbol b}_n^{-1}\mathscr{U}_\theta}(\arg s_j)\,d\mathring{\widetilde{\nu}}_n(\boldsymbol s) \\
&+\sin(\theta_{nj})\int_{\mathbb{T}^2}
(1-\Re s_j)\,d\mathring{\widetilde{\nu}}_n(\boldsymbol s)+\int_{\mathbb{T}^2}\frac{\arg s_j-\Im s_j}{1-\Re s_j}\,d(1-\Re s_j)\mathring{\widetilde{\nu}}_n(\boldsymbol s)
\end{align*}
and
\[R_{nj}=(\theta_{nj}-\sin\theta_{nj})+(1-\cos\theta_{nj})\int_{\mathbb{T}^2}(\Im s_j)\,d\mathring{\widetilde{\nu}}_n.\] Note that sets $\mathbb{T}^2\backslash\widetilde{\boldsymbol b}_n^{-1}\mathscr{U}_\theta$ are uniformly bounded away from $\mathbf{1}$ for all large $n$, and so $\limsup_{n\to\infty}k_n|B_{nj}|<\infty$ by the $\mathscr{M}_{\mathbb{T}^2}^\mathbf{1}$-convergence assumption of $\rho_n'$ and by the equivalence of Condition \ref{condX1} and Condition \ref{condX2}. That $|R_{nj}|\leq|\theta_{nj}|^3+|\theta_{nj}|^2$ holds for all large $n$ then leads to
\[\limsup_{n\to\infty}k_n|\theta_{nj}|\big[1-|\theta_{nj}|-|\theta_{nj}|^2\big]\leq\limsup_{n\to\infty}
k_n|B_{nj}|<\infty.\] We thus obtain $\limsup_nk_n|\theta_{nj}|<\infty$ and $\lim_nk_n\|\boldsymbol\theta_n\|^2=0$. Consequently, $\rho_n$ satisfies Condition \ref{condX1} by Proposition \ref{perturbthm} again.

Finally, by using (\ref{meanzero}), one can express components of $\boldsymbol E_n=(E_{n1},E_{n2})$ as
\begin{align*}
E_{nj}&=k_n\theta_{nj}+k_n\Im (\widetilde{b}_{nj}^{-1})\int_{\mathbb{T}^2}(\Re s_j)\,d\widetilde{\nu}_n(\boldsymbol s) \\
&=k_n(\theta_{nj}-\sin\theta_{nj})+\sin(\theta_{nj})\int_{\mathbb{T}^2}(1-\Re s_j)\,d\rho_n(\boldsymbol s).
\end{align*} As noted above that $\rho_n$ converges in $\mathscr{M}_{\mathbb{T}^2}^\mathbf{1}$ if and only if so does $\rho_n'$ and that $\lim_{n\to\infty}k_n|\theta_{nj}|^2=0$ in either case. Consequently, we have shown that
$\lim_{n\to\infty}E_{nj}=0$ whenever $\rho_n$ or $\rho_n'$ converges in $\mathscr{M}_{\mathbb{T}^2}^\mathbf{1}$, and so arrived at $\boldsymbol\gamma=\lim\boldsymbol\gamma_n$ if (1) or (3) holds. This completes the proof.
\end{proof}

\begin{remark} In spite of $\delta_{-\mathbf{1}}^{\bitimes2n}=\delta_{\mathbf{1}}$, $2n\delta_{-\mathbf{1}}$ fails to converge in $\mathscr{M}_{\mathbb{T}^2}^\mathbf{1}$. This example demonstrates that in Theorem \ref{identicalX3}, the rotated probabilities $\widetilde{\nu}_n$ are a necessary medium in the convergence criteria of the bi-free multiplicative limit theorem. For the same inference, the converse statement of Proposition \ref{identicalX2} does not hold, yet it does in the additive setting \cite[Theorem 5.6]{HHW}.
\end{remark}

\section{Classical limit theorems on higher dimensional tori} \label{sec6}
This section focuses on the classical limit theorem on the $d$-torus. Recall from \cite{Par67} that a measure $\nu$ in $\mathcal{ID}(\circledast)$ has no non-trivial $\circledast$-idempotent factor if and only if
its characteristic function takes the form
\begin{equation} \label{eq:classical_exponent}
\widehat{\nu}(\boldsymbol p)=\boldsymbol\gamma^{\boldsymbol p}\exp\left(-\frac{1}{2}\langle \mathbf{A}\boldsymbol p,\boldsymbol p\rangle+\int_{\mathbb{T}^d}\big({\boldsymbol s}^{\boldsymbol p}-1-i\langle \boldsymbol p,\Im\boldsymbol s\rangle\big)\,d\rho({\boldsymbol s})\right),\qquad \boldsymbol p \in \mathbb Z^d,
\end{equation} for certain triplet $(\boldsymbol\gamma,\mathbf{A},\rho)$ fulfilling the conditions in (\ref{propTD}).
We shall write $\nu_{\circledast}^{(\boldsymbol\gamma,\mathbf{A},\rho)}$ for this measure, and refer to $\rho$ and $(\boldsymbol\gamma,\mathbf{A},\rho)$ as its multiplicative L\'evy measure and multiplicative L\'evy triplet, respectively, or simply L\'evy measure and L\'evy triplet if no confusion occurs.

A known phenomenon is that a $\circledast$-infinitely divisible distribution on $\mathbb{T}^d$ may have various L\'evy measures. For example, it was pointed out in \cite{Ceb16} that when $d=1$,
\begin{equation} \label{Cebron's}
\nu_{\circledast}^{(1,0,\pi\delta_i)}=\nu_{\circledast}^{(1,0,\pi\delta_{-i})}.
\end{equation} Some examples of non-unique Haar absolutely continuous L\'evy measures were provided in \cite[Proposition 3.8]{CG08}.
In the following, we first characterize L\'{e}vy triplets that cause the same $\circledast$-infinitely divisible distribution. This will be a crucial ingredient in the main theorem of this section, a limit theorem for infinitesimal triangular arrays.

Some terminologies are used in the following lemma: a Borel set $B\subset\mathbb{T}^d$ and a function $f$ defined on $\mathbb{T}^d$ are called \emph{even} if $B^{-1}=B$ and $f(\boldsymbol s^{-1})=f(\boldsymbol s)$ for all $\boldsymbol s\in\mathbb{T}^d$, respectively.

\begin{lem} \label{lem:smooth}
Let $\rho_1$ and $\rho_2$ be L\'evy measures on $\mathbb{T}^d$, that is, non-negative measures satisfying $\rho_k(\{\mathbf{1}\})=0$ and $\mathbf1-\Re \boldsymbol s\in L^1(\rho_k)$ for $k=1,2$. Suppose that
\begin{equation} \label{integralrho}
\Re\int_{\mathbb{T}^d}(1-\boldsymbol s^{\boldsymbol p})\,d\rho_1(\boldsymbol s)=\Re\int_{\mathbb{T}^d}(1-\boldsymbol s^{\boldsymbol p})\,d\rho_2(\boldsymbol s),\;\;\;\;\; \boldsymbol
p\in\mathbb{Z}^d.
\end{equation}
Then $\rho_1(B)=\rho_2(B)$ for any even Borel set $B\subset\mathbb{T}^d$.
\end{lem}

\begin{proof} Let $\rho=\rho_1-\rho_2$. First, the assumption $\mathbf1-\Re \boldsymbol s\in L^1(\rho_k)$ shows that each integral in (\ref{integralrho}) is a finite number. It is enough to prove, by a standard argument of functions' approximations, that the following applies to any even $C^\infty$-functions $f$ vanishing in a neighborhood of $\mathbf{1}\in\mathbb{T}^d$:
\begin{equation} \label{evenf}
\int_{(-\pi,\pi]^d}f(e^{i\boldsymbol x})\,d\rho(e^{i\boldsymbol x})=0.
\end{equation} For notational convenience, we adopt $f(\boldsymbol x)$ with $\boldsymbol x\in(-\pi,\pi]^d$ instead of using $f(\boldsymbol s)$, $\boldsymbol s\in\mathbb{T}^d$. Let $f_n$ be the truncated Fourier cosine series of $f$:
\[f_n(\boldsymbol  x)=\sum_{\boldsymbol p\in\mathbb Z^d,\,\|\boldsymbol p\|\le n}a_{\boldsymbol p} \cos \langle\boldsymbol p, \boldsymbol  x\rangle,
\;\;\;\;\;a_{\boldsymbol p}=\frac{1}{(2\pi)^d}\int_{(-\pi,\pi]^d}f(\boldsymbol x) \cos \langle\boldsymbol p,\boldsymbol x\rangle\,d\boldsymbol x.\]
Since
\[f_n(\boldsymbol x)-f_n(\boldsymbol 0)=\sum_{\substack{\boldsymbol  p\in\mathbb Z^d,\,0<\|\boldsymbol p\|\le n}}a_{\boldsymbol p}(\cos\langle\boldsymbol ,\boldsymbol x\rangle-1),\]
hypothesis (\ref{integralrho}) yields that
\begin{equation} \label{eq:cos}
\int_{(-\pi,\pi]^d\setminus\{\boldsymbol0\}}\frac{f_n(\boldsymbol x)-f_n(\boldsymbol 0)}{\|\boldsymbol  x\|^2}\cdot\|\boldsymbol x\|^2\,d\rho(e^{i\boldsymbol x})=0,\;\;\;\;\;n\in\mathbb{N}.
\end{equation}
On the other hand, integration by parts yields that $a_{\boldsymbol  p}$ is rapidly decreasing in the sense that for every $k\in\mathbb N$, there exists a constant $C_k>0$ such that $|a_{\boldsymbol p}| \le C_k(1+\|\boldsymbol p\|)^{-k}$. This implies $\|\partial^{\boldsymbol \alpha} f_n-\partial^{\boldsymbol \alpha} f\|_{L^\infty}\to0$ as $n\to\infty$ for any $\boldsymbol \alpha\in(\mathbb{N}\cup\{0\})^d$. In particular, we have $f_n(\boldsymbol 0)\to0$ and $f_n(\boldsymbol x)-f_n(\boldsymbol 0)\to f(\boldsymbol x)$ pointwisely as $n\to\infty$. Furthermore, there is a constant $C>0$ neither depending on $n$ nor $\boldsymbol x$ such that
\[|f_n(\boldsymbol x) - f_n(\boldsymbol 0)| \leq C \|\boldsymbol x\|^2, \qquad \boldsymbol x \in  (-\pi,\pi]^d,\]
thanks to Taylor's theorem $f_n(\boldsymbol x)-f_n(\boldsymbol 0) = (1/2)\sum_{i,j=1}^d\partial^{\boldsymbol e_i +\boldsymbol e_j} f_n(c \boldsymbol x) x_i x_j$ for some $c=c_{n,\boldsymbol x} \in (0,1)$ and the fact that $\{\|\partial^{\boldsymbol \alpha} f_n\|_{L^\infty}\}_{n=1}^\infty$ is uniformly bounded for every fixed $\boldsymbol \alpha$. Moreover, $(1-\cos x_j)/x_j^2\to2$ as $x_j\to0$ for any $j$ gives that $\|\boldsymbol x\|\in|\rho|$.  Finally, applying the dominated convergence theorem to \eqref{eq:cos} yields the desired result (\ref{evenf}).
\end{proof}

We utilize the symbol $a\equiv b\mod c$ that means $a-b=nc$ for some $n\in\mathbb{Z}$.

\begin{prop} \label{prop:claim}
The following statements are equivalent for two L\'evy triplets $(\boldsymbol\gamma,\mathbf A,\rho)$ and $(\boldsymbol\gamma',\mathbf A',\rho')$.
\begin{enumerate} [$\qquad(1)$]
\item $\nu_{\circledast}^{(\boldsymbol\gamma,\mathbf A,\rho)}
=\nu_{\circledast}^{(\boldsymbol\gamma',\mathbf A',\rho')}$.

\item\label{item:non-unique} $\boldsymbol\gamma=\boldsymbol\gamma'$, $\mathbf A=\mathbf A'$,
and for all $\boldsymbol p \in\mathbb{Z}^d$,
\[\int_{\mathbb{T}^d} \big(\boldsymbol s^{\boldsymbol p}-1-i \langle \boldsymbol p, \Im(\boldsymbol s)\rangle\big)\,d\rho(\boldsymbol s)\equiv\int_{\mathbb{T}^d}\big({\boldsymbol s}^{\boldsymbol p}-1-i\langle \boldsymbol p,\Im(\boldsymbol s)\rangle\big)
\,d\rho'(\boldsymbol s)\mod 2\pi i.\]

\item $\boldsymbol\gamma=\boldsymbol\gamma'$, $\mathbf A=\mathbf A'$, $\rho(B)=\rho'(B)$ for
all even Borel subsets $B \subset\mathbb{T}^d$, and for all $\boldsymbol p\in\mathbb{Z}^d$,
\begin{equation} \label{eq:claim}
\int_{\mathbb{T}^d}\Im(\boldsymbol s^{\boldsymbol p}-\langle \boldsymbol p, \boldsymbol s\rangle )\,d\rho(\boldsymbol s)\equiv\int_{\mathbb{T}^d}\Im(\boldsymbol s^{\boldsymbol p}-\langle \boldsymbol p, \boldsymbol s\rangle )\,d\rho'(\boldsymbol s)\mod2\pi.
\end{equation}
\end{enumerate}
\end{prop}
\begin{remark}\label{rem:unique}
\begin{enumerate}[(i)]

\item The uniqueness of the Gaussian components is generally valid
on a locally compact abelian group, see \cite[Theorem 8.1]{Par67}. For the sake of completeness, we include a proof for the special group $\mathbb{T}^d$.

\item The uniqueness of the drift term depends on the choice of the compensating
function of the L\'evy-Khintchine representation. For instance, if we choose the alternative representation
$$
\widehat \nu(n) = \gamma^n \exp\left(-\frac{1}{2}a n^2 +\int_{\mathbb{T}}\left(s^n-1-in\left[\Im s + \frac{1}{2}(\Im s)^3\right] \right)\,d\rho(s)\right)
$$
in $d=1$, then the drift term is not unique anymore, since both triplets $(1,0,\pi\delta_i)$ and $(-1,0,\pi\delta_{-i})$ give the same distribution.
\end{enumerate}
\end{remark}

\begin{proof} We adopt notations such as $\boldsymbol\gamma = (\gamma_1\dots, \gamma_d)$ and $\mathbf A = (a_{ij})_{i,j=1}^d$ in the proof. For simplicity, let $\nu$ stand for $\nu_{\circledast}^{(\boldsymbol\gamma,\mathbf A,\rho)}$.

Statement (2) reveals that all moments of $\nu$ and $\nu_{\circledast}^{(\boldsymbol\gamma',\mathbf A',\rho')}$ coincide, and hence (2)$\Rightarrow$(1). Conversely, if (1) holds, then the identity
\[
\widehat{\nu}(\boldsymbol e_k)=\gamma_k\exp\left(-\frac{a_{kk}}{2}+\int_{\mathbb{T}^d}\big(\Re(s_k)-1\big)\,d\rho(\boldsymbol s)\right)
\]
shows that $\gamma_k=\widehat{\nu}(\boldsymbol e_k)/ |\widehat{\nu}(\boldsymbol e_k)|=\gamma_k'$. Furthermore, the inequality $1-\cos (n x)\leq n^2x^2/2$ for any $n\in\mathbb{Z}$ and the dominated convergence theorem allow us to obtain that
\[e^{-a_{kk}/2}=\lim_{n\to\infty}|\widehat{\nu}(n\boldsymbol e_k)|^{1/n^2},\]
which leads to $a_{kk}=a_{kk}'$. Similarly, the result
\[
e^{-a_{jk}}=e^{(a_{jj}+a_{kk})/2}\lim_{n\to\infty}|\widehat{\nu}(n\boldsymbol e_j+n\boldsymbol e_k)|^{1/n^2},\;\;\;\;\;j \neq k,\] implies $a_{jk}=a_{jk}'$. Hence we have established (1)$\Rightarrow$(2). Statements (2) and (3) are equivalent by treating real and imaginary parts separately and by using Lemma \ref{lem:smooth}.
\end{proof}

As we have seen that the drift part $\boldsymbol\gamma$ and Gaussian component $\mathbf{A}$ in the representation (\ref{eq:classical_exponent}) are always unique, namely the non-uniqueness of L\'evy triplets is solely due to the L\'evy measures. This observation leads to the following definition.

\begin{pdef} Let $\rho$ be a L\'evy measure on $\mathbb T^d$.
The symbol $\mathcal{L}(\rho)$ stands for the collection of those measures serving as L\'{e}vy measures for $\nu_{\circledast}^{(\mathbf1,\mathbf0,\rho)}$.
\end{pdef}

Detailed discussions of the uniqueness of the L\'{e}vy measure will be presented in Section \ref{sec8}. For Now, let us focus on limit theorems. Given an infinitesimal triangular array $\{\nu_{nk}\}_{n\geq1,\,1\leq k \leq k_n}\subset\mathscr{P}_{\mathbb{T}^d}$ and a sequence $\{\boldsymbol\xi_n\}\subset\mathbb{T}^d$, it was proved in \cite[Theorem 5.2]{Par67} that the limit of the sequence
\begin{equation} \label{circlelimit}
\nu_n=\delta_{\boldsymbol\xi_n}\circledast\nu_{n1}\circledast\cdots\circledast\nu_{nk_n},
\;\;\;\;\;n=1,2,3,\dots,
\end{equation}
if it does converge weakly, belongs to $\mathcal{ID}(\circledast)$.
Some basic algebraic works then show that for $\boldsymbol p\in\mathbb{Z}^d$ and for large $n$,
\begin{align} \label{prod}
\widehat{\nu}_n(\boldsymbol p)&=\boldsymbol\xi_n^{\boldsymbol p}\prod_{k=1}^{k_n}{\boldsymbol b}_{nk}^{\boldsymbol p}\,\widehat{\mathring\nu_{nk}}(\boldsymbol p) \\
&=\exp\left(i\langle\boldsymbol p,\arg\boldsymbol\xi_n\rangle+i\sum_{k=1}^{k_n}\langle\boldsymbol p,\arg{\boldsymbol b}_{nk}\rangle+\sum_{k=1}^{k_n} \log\widehat{\mathring\nu_{nk}}(\boldsymbol p)\right) \nonumber.
\end{align} Here, $\log\widehat{\mathring\nu_{nk}}(\boldsymbol p)$ is recognized as the principal value of the logarithm function since $\widehat{\mathring\nu_{nk}}(\boldsymbol p)$ is close to $1$ for all sufficiently large $n$. In order to derive the criteria ensuring the weak convergence of $\nu_n$, we consider the quantity
\[z_{nk}(\boldsymbol p)=\widehat{\mathring\nu_{nk}}(\boldsymbol p)-1,\;\;\;\;\;\boldsymbol p\in\mathbb{Z}^d,\] which has a non-positive real part and satisfies $\lim_{n\to\infty}\max_{1\leq k\leq k_n}|z_{nk}(\boldsymbol p)|\to0$. More useful information regarding this array $\{z_{nk}\}$ is recorded below.

\begin{lem} \label{lem:re-im}
Let $\{\boldsymbol e_j\}_{j=1}^d$ be the standard basis of $\mathbb{Z}^d$. Then
there exists a constant $C(\theta,d)>0$, depending on $\theta$ and the dimension $d$ only, such that
for any $\boldsymbol p\in\mathbb{Z}^d\backslash\{\boldsymbol0\}$,
\[\big|\Im z_{nk}(\boldsymbol p)\big|\leq\big|\Re z_{nk}(\boldsymbol p)\big|+C(\theta,d)\|\boldsymbol p\|\sum_{j=1}^d
\big|\Re z_{nk}(\boldsymbol e_j)\big|\]
holds for all sufficiently large $n$ and $1\leq k\leq k_n$.
\end{lem}

\begin{proof} In our proof, $n$ is always large enough. Firstly, the relation $\arg\boldsymbol s=-\arg\boldsymbol b_{nk}+\arg(\boldsymbol b_{nk}\boldsymbol s)$ for ${\boldsymbol s}\in\mathscr{U}_{\theta/2}$ and the inclusion $\boldsymbol b_{nk}\mathscr{U}_{\theta/2}\subset\mathscr{U}_\theta$ show that
\begin{align*}
\left|\int_{\mathscr{U}_{\theta/2}}(\arg s_j)\,d\mathring\nu_{nk}\right|
&=\left|-\mathring\nu_{nk}(\mathscr{U}_{\theta/2})\arg b_{nkj}-\int_{\mathscr{U}_\theta\backslash(\boldsymbol b_{nk}\mathscr U_{\theta/2})}(\arg s_j)\,d\nu_{nk}+\arg b_{nkj}\right| \\
&=\left|\mathring\nu_{nk}(\mathbb{T}^d\backslash\mathscr{U}_{\theta/2})\arg b_{nkj}-\int_{\mathscr{U}_\theta\setminus(\boldsymbol b_{nk}\mathscr U_{\theta/2})}(\arg s_j)\,d\nu_{nk}\right|  \\
&\leq 2\mathring\nu_{nk}(\mathbb{T}^d\backslash\mathscr{U}_{\theta/4}).
\end{align*} Secondly, the validity of the inequality $|\Im{\boldsymbol s}^{\boldsymbol p}-\langle\boldsymbol p,\arg{\boldsymbol s}\rangle|\leq1-\Re{\boldsymbol s}^{\boldsymbol p}$ for ${\boldsymbol s}\in\mathscr{U}_{\theta/2}$ by (\ref{thetarange}) and (\ref{basicineq}) leads to
\begin{align*}
\big|\Im z_{nk}(\boldsymbol p)\big|
&\leq\int_{\mathscr{U}_{\theta/2}}\big|\Im\boldsymbol s^{\boldsymbol p}-\langle\boldsymbol p,\arg{\boldsymbol s}\rangle\big|\,d\mathring\nu_{nk}
+\left|\int_{\mathscr{U}_{\theta/2}}\langle\boldsymbol p,\arg{\boldsymbol s}\rangle\,d\mathring\nu_{nk}\right|+\mathring\nu_{nk}(\mathbb{T}^d\backslash \mathscr{U}_{\theta/2}) \\
&\leq|\Re z_{nk}(\boldsymbol p)|+3\|\boldsymbol p\|\mathring\nu_{nk}(\mathbb{T}^d
\backslash\mathscr{U}_{\theta/4}).
\end{align*}
Finally, combining the inequality offered above and the following one
\begin{align*}
\mathring\nu_{nk}(\mathbb{T}^d\backslash\mathscr{U}_{\theta/4})&\leq\sum_{j=1}^d\int_{\{|\arg s_j|\geq\theta/(4\sqrt{d})\}}1\;d\mathring\nu_{nk}({\boldsymbol s}) \\
&\leq\frac{1}{1-\cos\big(\theta/(4\sqrt{d})\big)}\sum_{j=1}^d\int_{\{|\arg s_j|\geq\theta/(4\sqrt{d})\}}(1-\Re s_j)\,d\mathring\nu_{nk}({\boldsymbol s})\\
&\leq\frac{1}{1-\cos\big(\theta/(4\sqrt{d})\big)}\sum_{j=1}^d\big|\Re z_{nk}(\boldsymbol e_j)\big|
\end{align*} yield the desired result.
\end{proof}

\begin{thm} \label{ClimitthmX}
Let $(\boldsymbol\gamma,\mathbf{A},\rho)$ be a multiplicative L\'evy triplet. Given an infinitesimal triangular array $\{\nu_{nk}\}\subset\mathscr{P}_{\mathbb{T}^d}$, define $\rho_n=\sum_{k=1}^{k_n}\mathring\nu_{nk}$ and $\boldsymbol\gamma_n$ as in \emph{(\ref{gamman})}.
Then the sequence \emph{(\ref{circlelimit})} converges weakly to $\nu_{\circledast}^{(\boldsymbol\gamma,\mathbf{A},\rho)}$ if and only if the following statements hold:
\begin{enumerate} [$\qquad(1)$]
\item \label{item:cond1} $\lim_{n\to\infty}\boldsymbol\gamma_n=\boldsymbol\gamma$,
\item the limit
\[\lim_{\epsilon\to0}\limsup_{n\to\infty}\int_{\mathscr{U}_\epsilon}(\Im s_j)(\Im s_\ell)\,d\rho_n=\lim_{\epsilon\to0}\liminf_{n\to\infty}\int_{\mathscr{U}_\epsilon}(\Im s_j)(\Im s_\ell)\,d\rho_n\]
exists and equals the $(j,\ell)$-entry of $\mathbf{A}$, and
\item\label{item:cond3} $\{1_{\mathbb T^d \setminus\{\mathbf 1\}}\rho_n\}$ is
$\tilde{\mathscr{M}}_{\mathbb T^d}^{\mathbf1}$-relatively compact and all of its limit points are contained in $\mathcal{L}(\rho)$.
\end{enumerate}
\end{thm}

\begin{proof} Suppose that $\nu_n\Rightarrow\nu=\nu_{\circledast}^{(\boldsymbol\gamma,\mathbf{A},\rho)}$, i.e., $\widehat{\nu_n}(\boldsymbol p)\to\widehat{\nu}(\boldsymbol p)\neq0$ at every point $\boldsymbol p\in\mathbb{Z}^d$. We first infer from (\ref{prod}) that
\[\lim_{n\to\infty}\sum_{k=1}^{k_n}\log\big|\widehat{\mathring\nu_{nk}}(\boldsymbol p)\big|\in(-\infty,0].\] Writing $x_{nk}(\boldsymbol p)=\Re z_{nk}(\boldsymbol p)$ and $y_{nk}(\boldsymbol p)=\Im z_{nk}(\boldsymbol p)$, we have the estimate
\begin{align*} \label{eq:log}
\log|\widehat{\mathring\nu_{nk}}(\boldsymbol p)|&=\frac{1}{2}\log\big(1+2x_{nk}(\boldsymbol p)+x_{nk}(\boldsymbol p)^2+y_{nk}(\boldsymbol p)^2\big) \\
&\leq\frac{1}{2}\big(2x_{nk}(\boldsymbol p)+x_{nk}(\boldsymbol p)^2+y_{nk}(\boldsymbol p)^2\big).
\end{align*} Then the employment of Lemma \ref{lem:re-im} shows that
\begin{align*}
\sum_{j=1}^d y_{nk}(\boldsymbol e_j)^2\leq\left(\sum_{j=1}^d|y_{nk}(\boldsymbol e_j)|\right)^2&\leq C(\theta,d)^2\left(\sum_{j=1}^dx_{nk}(\boldsymbol e_j)\right)^2 \\
&\leq dC(\theta,d)^2\sum_{j=1}^dx_{nk}(\boldsymbol e_j)^2.
\end{align*} For $n$ sufficiently large such that $|x_{nk}(\boldsymbol e_j)|\leq 2^{-1}\min\{1,1/(dC(\theta,d)^2)\}$ for $1\leq k\leq k_n$ and $j=1,\ldots,d$, we further obtain that
\[\sum_{j=1}^d \log|\widehat{\mathring\nu_{nk}}(\boldsymbol e_j)|\leq\frac{1}{2}\sum_{j=1}^dx_{nk}(\boldsymbol e_j)\leq0.\]
Thus, $\{\sum_{k=1}^{k_n}x_{nk}(\boldsymbol e_j)\}_{n=1}^\infty$ is a bounded sequence for each $j=1,\ldots,d$. This also says that for each $j$, the sequence of positive measures
$d\lambda_{nj}({\boldsymbol s})=(1-\Re s_j)d\rho_n({\boldsymbol s})$
has uniformly bounded total mass $\lambda_{nj}(\mathbb{T}^d)=-\sum_{k=1}^{k_n}x_{nk}(\boldsymbol e_j)$.

Next, let $\widetilde{\lambda}_j$ be a weak limit point of $\{\lambda_{nj}\}$, and suppose that $\lambda_{n_m j}\Rightarrow\widetilde{\lambda}_j$ for some subsequence $\{\lambda_{n_m j}\}$.
One can quickly see from the first part of the proof of Proposition \ref{equicondX} that
$\rho_{n_m}$ converges to certain $\widetilde{\rho}\in\tilde{\mathscr{M}}_{\mathbb{T}^d}^\mathbf{1}$.
Also, observe that $|y_{n_m k}(\boldsymbol e_j)|\leq(1+C(\theta,d))|x_{n_m k}(\boldsymbol e_j)|+S_{n_m k}^{(j)}$
and $\log\widehat{\mathring\nu_{n_m k}}(\boldsymbol e_j)=z_{n_m k}(\boldsymbol e_j)(1+\epsilon_{n_m k}^{(j)})$, where $S_{n_m k}^{(j)}=C(\theta,d)\sum_{j'\neq j}|x_{n_m k}(\boldsymbol e_{j'})|$ satisfies
$\sup_{m\geq1}\sum_{k=1}^{k_{n_m}}S_{n_m k}^{(j)}<\infty$ and $\sup_{1\leq k\leq k_{n_m}}|\epsilon_{n_m k}^{(j)}|\to0$ as $n_m\to\infty$ for all $j$. Hence we can deduce from \cite[Lemma 2.1]{Wan08} that
\begin{align}
\lim_{m\to\infty}\widehat{\nu_{n_m}}(\boldsymbol e_j)&=\lim_{m\to\infty}\exp\left(i\langle\boldsymbol e_j,\arg\boldsymbol\xi_{n_m}\rangle+i
\sum_{k=1}^{k_{n_m}}\langle\boldsymbol e_j,\arg\boldsymbol b_{n_m k}\rangle+\sum_{k=1}^{k_{n_m}}\log\widehat{\mathring\nu_{n_m k}}(\boldsymbol e_j)\right) \nonumber \\
&=\lim_{m\to\infty}\exp\left(i\langle\boldsymbol e_j,\arg\boldsymbol\xi_{n_m}\rangle+i\sum_{k=1}^{k_{n_m}}\langle\boldsymbol e_j,\arg \nonumber\boldsymbol b_{n_m k}\rangle+\sum_{k=1}^{k_{n_m}}z_{n_m k}(\boldsymbol e_j)\right).
\end{align}
Since $\sum_{k=1}^{k_{n_m}}x_{n_mk}(\boldsymbol e_j)$
converges as $m\to\infty$, we arrive at
\[\tilde\gamma_j:=\lim_{m\to\infty}\exp\left(i\langle\boldsymbol e_j,\arg\boldsymbol\xi_{n_m} \rangle+i\sum_{k=1}^{k_{n_m}}\langle\boldsymbol e_j,\arg\boldsymbol b_{n_m k}\rangle+i\sum_{k=1}^{k_{n_m}}y_{n_m k}(\boldsymbol e_j)\right)\in\mathbb{T}.\]

Next, making use of the fact (\ref{ImRe}) allows us to obtain
\[\lim_{\epsilon\to0}\liminf_{m\to\infty}\int_{\mathscr{U}_\epsilon}(\Im s_j)^2\,d\rho_{n_m}({\boldsymbol s})=2\widetilde\lambda_j(\{\mathbf{1}\})=
\lim_{\epsilon\to0}\limsup_{m\to\infty}\int_{\mathscr{U}_\epsilon}(\Im s_j)^2\,d\rho_{n_m}({\boldsymbol s}).\] This verifies the situation $j=\ell$ and $n=n_m$ in (2); denote the limit by $\widetilde{a}_{jj}$.
We now proceed to the proof of the case when $j\neq\ell$ and $n=n_m$ in (2). First of all, we see from Lemma \ref{lem:re-im} and \cite[Lemma 2.1]{Wan08} again that for any $\boldsymbol p\in\mathbb{Z}^d\backslash\{\boldsymbol0\}$,
\begin{align}
\lim_{m\to\infty}\widehat{\nu_{n_m}}(\boldsymbol p)
&=\lim_{m\to\infty}\exp\left(i\langle\boldsymbol p,\arg\boldsymbol\xi_{n_m}\rangle+i\sum_{k=1}^{k_{n_m}}\langle\boldsymbol p,\arg\boldsymbol b_{n_m k}\rangle+
\sum_{k=1}^{k_{n_m}}z_{n_m k}(\boldsymbol p)\right) \notag \\
&=\lim_{m\to\infty}\exp\left(i\left\langle\boldsymbol p,\arg\boldsymbol\xi_{n_m}+\sum_{k=1}^{k_{n_m}}\arg\boldsymbol b_{n_m k}+\int_{\mathbb{T}^d}(\Im\boldsymbol s)\,d\rho_{n_m}({\boldsymbol s})\right\rangle\right. \label{eq:conv} \\
&\quad\quad\quad\quad\quad\left.+\int_{\mathbb{T}^d}\big(\boldsymbol s^{\boldsymbol p}-1-i\langle\boldsymbol p,\Im\boldsymbol s\rangle\big)\,d\rho_{n_m}({\boldsymbol s}\big)\right) \notag \\
&=\widetilde{\boldsymbol\gamma}^{\boldsymbol p}\lim_{m\to\infty}\exp\left(\int_{\mathbb{T}^d}\big({\boldsymbol s}^{\boldsymbol p}-1-i\langle\boldsymbol p,\Im\boldsymbol s\rangle\big)\,d\rho_{n_m}({\boldsymbol s})\right), \notag
\end{align} where $\widetilde{\boldsymbol\gamma}=(\widetilde{\gamma}_1,\ldots,\widetilde{\gamma}_d)$. Particularly, taking $\boldsymbol p=\boldsymbol e_j+\boldsymbol e_\ell$ with $j\neq\ell$ in \eqref{eq:conv} yield that
\[\lim_{m\to\infty}\int_{\mathbb{T}^d}\big((\Re s_j)(\Re s_\ell)-1-(\Im s_j)(\Im s_\ell)\big)\,d\rho_{n_m}({\boldsymbol s})\] exists.
This, along with $\lambda_{n_m j}\Rightarrow\widetilde\lambda_j$ and the decomposition $(\Re s_j)(\Re s_\ell)-1=(\Re s_j-1)\Re s_\ell +\Re s_\ell-1$, further yields the existence of
\[
\widetilde L_{j\ell}=\lim_{m \to\infty}\int_{\mathbb{T}^d}(\Im s_j)(\Im s_\ell)\,d\rho_{n_m}({\boldsymbol s}).\] Consequently, we infer from $\rho_{n_m}\Rightarrow_\mathbf{1}\tilde\rho$ that the limits in assertion (2) with $n=n_m$ exist; let $\widetilde{a}_{j\ell}$ stand for the limit.

Cutting the domain $\mathbb{T}^d$ in \eqref{eq:conv} into the union of $\mathscr{U}_\epsilon$ and $\mathbb{T}^d\backslash\mathscr{U}_\epsilon$ and then letting $\epsilon\to0$ give
\[\lim_{m\to\infty}\widehat{\nu_{n_m}}(\boldsymbol p)=\widetilde{\boldsymbol\gamma}^{\boldsymbol p}\exp\left(-\frac{1}{2}\langle\widetilde{\mathbf{A}}\boldsymbol p,\boldsymbol p\rangle
+\int_{\mathbb{T}^d}\big({\boldsymbol s}^{\boldsymbol p}-1-i\langle\boldsymbol p,\Im\boldsymbol s\rangle\big)\,d\tilde\rho({\boldsymbol s})\right),\]
where $\widetilde{\mathbf{A}}=(\widetilde{a}_{j\ell})$.
Therefore, we have arrived at $\nu_\circledast^{(\boldsymbol\gamma, \mathbf{A},\rho)}=\nu_\circledast^{(\widetilde{\boldsymbol\gamma},\widetilde{\mathbf{A}},\tilde\rho)}$, proving that  $\widetilde{\boldsymbol\gamma}=\boldsymbol\gamma$ and $\widetilde{\mathbf{A}}=\mathbf{A}$ by Proposition \ref{prop:claim}, and that assertion (3) holds.
The above arguments also prove assertion (1) since the limit $\boldsymbol \gamma$ does not depend on the choice of the subsequence.

To deal with any arbitrary sequence in assertion (2), firstly note that $\rho (B)=\widetilde \rho(B)$ for even Borel sets $B$ by Proposition \ref{prop:claim}, and hence $\int f d\rho = \int f d\widetilde \rho$ is valid for the even function $f(\boldsymbol s) =(\Im s_j)(\Im s_\ell)$. Then
we derive from Proposition \ref{equicondX} that
\begin{equation} \label{limitA1}
\lim_{m\to\infty}\int_{\mathbb{T}^d}(\Im s_j)(\Im s_\ell)\,d\rho_{n_m}({\boldsymbol s}) = \widetilde a_{j\ell} + \int_{\mathbb{T}^d}(\Im s_j)(\Im s_\ell)\,d\widetilde \rho({\boldsymbol s})  = a_{j\ell} + \int_{\mathbb{T}^d}(\Im s_j)(\Im s_\ell)\,d\rho({\boldsymbol s}),
\end{equation} which reveals that the limit displayed above has nothing to do with the choice of the subsequence $\{\rho_{n_m}\}_{m\geq1}$. Thus, (\ref{limitA1}) holds for the whole sequence $\{\rho_n\}$. Similar reasonings show
\begin{equation} \label{limitA2}
\lim_{n\to\infty}\int_{\mathbb T ^d \setminus\mathscr{U}_\epsilon}(\Im s_j)(\Im s_\ell)\,d\rho_{n}({\boldsymbol s}) = \int_{\mathbb T ^d \setminus\mathscr{U}_\epsilon}(\Im s_j)(\Im s_\ell)\,d\rho({\boldsymbol s})
\end{equation}
for every $\epsilon>0$ such that $\rho(\partial \mathscr{U}_\epsilon)=0$. Finally, combining \eqref{limitA1}, \eqref{limitA2}, and a part of the proof of Proposition \ref{equicondX} starting from \eqref{a} concludes assertion (2).

Conversely, if \eqref{item:cond1}-\eqref{item:cond3} hold and if some subsequence of $\{\rho_n\}$ converges to a certain $\tilde\rho\in\mathcal{L}(\rho)$, then the arguments above show that for any $\boldsymbol p\in\mathbb{Z}^d$,
\begin{align*}
L(\boldsymbol p):&=\lim_{n\to\infty}\exp\left(i\langle\boldsymbol p,\arg\boldsymbol\xi_n\rangle+i\sum_{k=1}^{k_n}\langle\boldsymbol p,\arg\boldsymbol b_{nk}\rangle+
\sum_{k=1}^{k_n}z_{nk}(\boldsymbol p)\right) \\
&=\boldsymbol\gamma^{\boldsymbol p}\exp\left(-\frac{1}{2}\langle\mathbf{A}\boldsymbol p,\boldsymbol p\rangle
+\int_{\mathbb{T}^d}\big({\boldsymbol s}^{\boldsymbol p}-1-i\langle\boldsymbol p,\Im{\boldsymbol s}\rangle\big)\,d\tilde\rho({\boldsymbol s})\right).
\end{align*} The last term equals $\widehat{\nu}({\boldsymbol p})$, where $\nu=\nu_\circledast^{(\boldsymbol\gamma,\mathbf{A},\rho)}$, due to $\tilde\rho\in\mathcal{L}(\rho)$. Since the sequence
$\{\sum_{k=1}^{k_n}\mathring\nu_{nk}(\mathbb{T}^d\backslash\mathscr{U}_\epsilon)\}_{n=1}^\infty$
is uniformly bounded for any $\epsilon>0$ by the relative compactness of $\{\rho_n|_{\mathbb T^d \setminus \{\mathbf1\}}\}$ and by \cite[Theorem 2.7]{mapping}, we infer from Lemma \ref{lem:re-im} and \cite[Lemma 2.1]{Wan08} to deduce that $\lim_{n\to\infty}\widehat{\nu}_n(\boldsymbol p)=L(\boldsymbol p)$. Thus, $\nu_n$ converges weakly to $\nu_{\circledast}^{(\boldsymbol\gamma,\mathbf{A},\rho)}$.
\end{proof}

The following corollary, derived from Theorem \ref{limitthmX} and Theorem \ref{classicalmul}, supplies the link between classical and bi-free limit theorems on the bi-torus. The attentive reader can also notice that the hypothesis $\mathcal{L}(\rho)=\{\rho\}$ is redundant in the implication (2)$\Rightarrow$(1).

\begin{cor}\label{cor:limit_free_classical} Let $\{\nu_{nk}\}\subset\mathscr{P}_{\mathbb{T}^2}$ be infinitesimal, $\{\boldsymbol\xi_n\}\subset\mathbb{T}^2$, and $(\boldsymbol\gamma,\mathbf{A},\rho)$ be a multiplicative L\'{e}vy triplet such that $\mathcal{L}(\rho)=\{\rho\}$. With the notations in \emph{(\ref{gamman})} and \emph{(\ref{rhon})} for $d=2$, the following statements are equivalent:
\begin{enumerate} [$\qquad(1)$]
\item\label{LCF1}
{$\delta_{\boldsymbol\xi_n}\circledast\nu_{n1}\circledast\cdots\circledast\nu_{nk_n}
\Rightarrow\nu_\circledast^{(\boldsymbol\gamma,\mathbf{A},\rho)}$;}
\item\label{LCF2} {$\delta_{\boldsymbol\xi_n}\bitimet\nu_{n1}\bitimet\cdots\bitimet\nu_{nk_n}
\Rightarrow\nu_{\bitimes}^{(\boldsymbol\gamma,\mathbf{A},\rho)}$;}
\item $\lim_{n\to\infty}\boldsymbol\gamma_n=\boldsymbol\gamma$,
$\rho_n\Rightarrow_\mathbf{1}\rho$, and
\[\lim_{\epsilon\to0}\limsup_{n\to\infty}\int_{\mathscr{U}_\epsilon}
\big\langle\boldsymbol p,\Im\boldsymbol s\big\rangle^2d\rho_n({\boldsymbol s})=\langle\mathbf{A}\boldsymbol p,\boldsymbol p\rangle =\lim_{\epsilon\to0}\liminf_{n\to\infty}\int_{\mathscr{U}_\epsilon}
\big\langle\boldsymbol p,\Im\boldsymbol s\big\rangle^2d\rho_n({\boldsymbol s}),\;\;\;\;\;\boldsymbol p\in\mathbb{Z}^d.\]
\end{enumerate}
\end{cor}

The one-dimensional multiplicative limit theorem, which was pointed out in the remark to \cite[Corollary 4.2]{Wan08}, is a consequence of Corollary \ref{ClimitthmX}, e.g., by considering product measures.

\begin{cor} \label{1Dcase}
Let $\{\nu_{nk}\}\subset\mathscr{P}_{\mathbb{T}}$ be infinitesimal, $\{\xi_n\}\subset\mathbb{T}$, and $(\gamma,a,\rho)$ be a multiplicative L\'{e}vy triplet such that $\mathcal{L}(\rho)=\{\rho\}$. With the notations in \emph{(\ref{gamman})} and \emph{(\ref{rhon})} for $d=1$, the following statements are equivalent:
\begin{enumerate} [$\qquad(1)$]
\item {$\delta_{\xi_n}\circledast\nu_{n1}\circledast\cdots\circledast\nu_{nk_n}
\Rightarrow\nu_\circledast^{(\gamma,a,\rho)}$;}
\item {$\delta_{\xi_n}\boxtimes\nu_{n1}\boxtimes\cdots\boxtimes\nu_{nk_n}
\Rightarrow\nu_{\boxtimes}^{(\gamma,a,\rho)}$;}
\item $\lim_{n\to\infty}\gamma_n=\gamma$, $\rho_n\Rightarrow_1\rho$, and
\[\lim_{\epsilon\to0}\limsup_{n\to\infty}\int_{\mathscr{U}_\epsilon}
\big(\Im s)^2d\rho_n({ s})=a =\lim_{\epsilon\to0}\liminf_{n\to\infty}\int_{\mathscr{U}_\epsilon}
\big(\Im s)^2d\rho_n({ s}).\]
\end{enumerate}
\end{cor}

Apparently, the non-uniqueness of L\'evy measures is the exclusive obstruction for reaching the equivalence of limit theorems, thus complementing the work of Chistyakov and G\"otze \cite[Theorems 2.3 and 2.4]{CG08}. Corollaries presented above motivate us to scrutinize the uniqueness of L\'evy measures, which will be the principal theme of Section \ref{sec8}.

\section{Homomorphisms between infinitely divisible distributions}\label{sec7}
This section will provide the explanation of the diagram (\ref{diagram}). The bijection $\Lambda\colon\mathcal{ID}(\ast)\to \mathcal{ID}(\biplus)$ was already defined in \cite{HHW}, specifically,
$$
\Lambda(\mu_*^{(\boldsymbol v,\mathbf{A},\tau)})=\mu_\biplus^{(\boldsymbol v,\mathbf{A},\tau)}.
$$

If $\nu=\mu_*^{(\boldsymbol v,\mathbf{A},\tau)}W^{-1}$, then (\ref{LKrepre}) and (\ref{wrapping2}) show that
\begin{align*}
\widehat{\nu}(\boldsymbol p)&=\int_{\mathbb{R}^d}e^{i\langle\boldsymbol p,\boldsymbol x\rangle}\,d\mu_*^{(\boldsymbol v,\mathbf{A},\tau)}(\boldsymbol x) \\
&=\exp\left[i\langle\boldsymbol p,\boldsymbol v\rangle-\frac{1}{2}\langle\mathbf{A}\boldsymbol p,\boldsymbol p\rangle+\int_{\mathbb{R}^d}
\left(e^{i\langle\boldsymbol p,\boldsymbol x\rangle}-1-\frac{i\langle\boldsymbol p,\boldsymbol x\rangle}{1+\|\boldsymbol x\|^2}\right)d\tau(\boldsymbol x)\right] \\
&={\boldsymbol\gamma}^{\boldsymbol p}\exp\left[-\frac{1}{2}\langle\mathbf{A}\boldsymbol p,\boldsymbol p\rangle+\int_{\mathbb{T}^d}
\big(\boldsymbol s^{\boldsymbol p}-1-i\langle\boldsymbol p,\Im\boldsymbol s\rangle\big)d\rho({\boldsymbol s})\right],
\end{align*} where $\rho$ and $\boldsymbol\gamma$ are respectively given in (\ref{rhotau}) and (\ref{gammav}). Putting it differently, the wrapping map induces a homomorphism $W_\ast\colon\mathcal{ID}(*)\to\mathcal{ID}(\circledast)$ satisfying
\begin{equation} \label{W*}
W_\ast(\mu_*^{(\boldsymbol v,\mathbf{A},\tau)})=\nu_\circledast^{(\boldsymbol\gamma,\mathbf{A},\rho)}.
\end{equation}

Motivated by the observation (\ref{W*}), we analogously define $W_{\biplus}:\mathcal{ID}(\biplus)\to\mathcal{ID}(\bitimes)$ as
\[W_{\biplus}(\nu_{\biplus}^{(\boldsymbol v,\mathbf{A},\tau)})=\nu_{\bitimes}^{(\boldsymbol\gamma,\mathbf{A},\rho)},\] where $\boldsymbol\gamma$ and $\rho$ are given as before. It was shown in Theorem \ref{+implyXthm} that the weak convergence of (\ref{munbifree}) to some $\nu_{\biplus}^{(\boldsymbol v,\mathbf{A},\tau)}$ implies that (\ref{bifreemul}) converges weakly to $W_{\biplus}(\nu_{\biplus}^{(\boldsymbol v,\mathbf{A},\tau)})$.

For the last ingredient $\Gamma\colon\mathcal{ID}(\bitimes)\to
\mathcal{ID}(\circledast)$, recall from Proposition \ref{deltaID} that $\bitimes$-idempotent elements also belong to $\mathcal{ID}(\circledast)$, which leads to the following definition:

\begin{pdef} Given a distribution $\nu\in\mathcal{ID}(\bitimes)$, define $\Gamma(\nu)\in\mathscr{P}_{\mathbb{T}^2}$ as follows.
If $\nu=\nu_{\bitimes}^{(\boldsymbol\gamma,\mathbf{A},\rho)}$, then $\Gamma(\nu)=\nu_\circledast^{(\boldsymbol\gamma,\mathbf{A},\rho)}$. For $\nu\in\mathscr{P}_{\mathbb{T}^2}\backslash\mathscr{P}_{\mathbb{T}^2}^\times$, define $\Gamma(\nu)=\nu$ if $\nu=P\bitimes(\kappa_c \times \delta_1)$, and define $\Gamma(\nu)=\mathrm{m} \times \Gamma_1(\nu^{(2)})$ if $\nu = \mathrm{m} \times \nu^{(2)}$
and $\Gamma(\nu)=\Gamma_1(\nu^{(1)}) \times \mathrm{m}$ if $\nu = \nu^{(1)} \times \mathrm{m}$. Here $\Gamma_1\colon\mathcal{ID}(\mathbb{T},\boxtimes)\to\mathcal{ID}(\mathbb{T},\circledast)$ is the homomorphism introduced in \cite[Definition 3.3]{Ceb16} (which was denoted by $\Gamma$ therein).
\end{pdef}

One can check that $\Gamma\colon\mathcal{ID}(\bitimes)\to
\mathcal{ID}(\circledast)$ is a homomorphism and that the diagram (\ref{diagram}) commutes. The latter result comes from the definition, while the former one requires convolution identities in Section \ref{sec3}. For example, if $\mu=P\bitimes(\kappa_c \times \delta_1)$ and $\nu=\mathrm{m}\times\nu^{(2)}$ with $\nu^{(2)} \in \mathcal{ID}(\boxtimes)\cap\mathscr{P}_{\mathbb{T}}^\times$, then
$$
\mu \bitimet \nu=P\bitimet (\mathrm{m} \times \nu^{(2)})=\mathrm{m}\times \mathrm{m},
$$
where the last equality can be confirmed by the use of \eqref{eq:identityP} and computing moments. On the other hand,
$$
\Gamma(\mu) \circledast \Gamma(\nu)=\mu \circledast (\mathrm{m} \times \Gamma_1(\nu^{(2)})) = P \circledast( \mathrm{m} \times \Gamma_1(\nu^{(2)})) = \mathrm{m}\times \mathrm{m},
$$
where the last equality is again obtained by computing moments. Consequently, we arrive at $\Gamma(\mu \bitimet \nu) = \Gamma(\mu) \circledast \Gamma(\nu)$.

This map $\Gamma$ is neither injective nor surjective as we have
$\nu_{\circledast}^{((1,0),\mathbf0,\pi\delta_{(i,0)})} = \nu_{\circledast}^{((1,0),\mathbf0, \pi \delta_{(-i,0)})}$ due to \eqref{Cebron's} and $P \circledast (\mu \times \delta_1)$ lies in $\mathcal{ID}(\circledast)\backslash\Gamma(\mathcal{ID}(\bitimes))$ for any $\mu \in \mathcal{ID}(\mathbb T,\circledast) \backslash\{\kappa_c:c \in \mathbb D \cup \mathbb T\}$. Further, $\Gamma$ is not weakly continuous. More strongly, we prove the following.

\begin{prop}
\begin{enumerate}[\rm(1)]
\item\label{not-cont1} The restriction of $\Gamma_1$ to the set
$\mathcal{ID}(\boxtimes)\cap\mathscr{P}_{{\mathbb T}}^\times$ has no weakly continuous extension to $\mathcal{ID}(\boxtimes)$.

\item\label{not-cont2} The restriction of $\Gamma$
to the set $\mathcal{ID}(\boxtimes\boxtimes)\cap\mathscr{P}_{{\mathbb T}^2}^\times$ has no weakly continuous extension to $\mathcal{ID}(\boxtimes\boxtimes)$.
\end{enumerate}
\end{prop}

\begin{proof} Since $\Gamma(\mu^{(1)} \times \mu^{(2)}) = \Gamma_1(\mu^{(1)}) \times \Gamma_1(\mu^{(2)})$ for $\mu^{(1)},\mu^{(2)}\in \mathcal{ID}(\boxtimes)\cap\mathscr{P}_{\mathbb T}^\times$, assertion \eqref{not-cont2} follows immediately from \eqref{not-cont1}.

Suppose that $\Gamma_1^0:=\Gamma_1|_{\mathcal{ID}(\boxtimes)\cap\mathscr{P}_{{\mathbb T}}^\times}$ has a weakly continuous extension $\tilde \Gamma_1$ to $\mathcal{ID}(\boxtimes)$.
Observe that $\kappa_c\in\mathcal{ID}(\boxtimes)\cap \mathscr{P}_{{\mathbb T}}^\times$ and $\Gamma_1^0(\kappa_c)=\kappa_c$ for any $c\in(\mathbb D \cup\mathbb T)\backslash\{0\}$. The latter identity is shown below.
From the moments $m_p(\kappa_c)=c^p$ for $p\in \mathbb N$, the formula
$$
\Sigma_{\kappa_c}(z) = \frac{1}{c} = \frac{1}{c/|c|} \exp\left[(-\log |c|)\int_{\mathbb T^\times} \frac{1+sz}{1-sz}(1-\Re s)\, \frac{d\mathrm{m}(s)}{1-\Re s} \right]
$$
yields that $\kappa_c$ has $(c/|c|,0, \rho)$, where $\rho(d s)= [-\log|c|/(1-\Re s)]\mathrm{m}(ds)$ on $\mathbb{T}^\times$, as its free multiplicative L\'evy triplet (also known as $\boxtimes$-characteristic triplet in \cite[p.2437]{Ceb16}). On the other hand, Lemma \ref{angleineq} says that the same triplet $(c/|c|,0,\rho)$ also serves as the classical multiplicative L\'evy triplet of $\kappa_c$. Thus we have shown that $\Gamma_1^0(\kappa_c)=\kappa_c$. That $\kappa_c\Rightarrow \mathrm{m}$ as $c\to0$ allows us to further obtain $\tilde\Gamma_1(\mathrm{m})=\mathrm{m}$.

Next, denote by $\nu_n$ the probability distribution in $\mathcal{ID}(\boxtimes)\cap \mathscr{P}_{{\mathbb T}}^\times$ having the free multiplicative L\'evy triplet $(1,0,n\delta_{-1})$, and let $\mu_n=\Gamma_1^0(\nu_n)$. Then (\ref{eq:classical_exponent}) shows that for any $p\in\mathbb{Z}$,
\[\widehat\mu_n(p)=\exp\big[n((-1)^{p}-1)\big)]=\begin{cases} 1, & \text{$p$ is even}, \\
e^{-2n}, & \text{$p$ is odd},
\end{cases}\] which readily implies that $\mu_n\Rightarrow(\delta_{-1}+\delta_{1})/2$. However, we will explain in the next paragraph that $\nu_n\Rightarrow\mathrm{m}$, which apparently leads to a contradiction.

To see why $\nu_n\Rightarrow\mathrm{m}$, select a weakly convergent subsequence of $\{\nu_n\}$ (still denoted by $\{\nu_n\}$ in the remaining arguments) and denote the weak limit by $\nu$. Let $\nu_n'$ be the probability measure having the free multiplicative L\'evy triplet $(1,0,(n/2)\delta_{-1})$. Passing to a further subsequence we may assume that $\nu_n'$ weakly converge to $\nu'$. Then letting $n\to\infty$ in the identity
$\nu_n=\nu_n'\boxtimes\nu_n'$ gives $\nu=\nu'\boxtimes\nu'$. On the other hand, we see from (\ref{bi-freeLK}) or from \cite[Section 2.5]{Ceb16} that $\Sigma_{\nu_n'}(0)=e^n$, i.e., $m_1(\nu_n')=e^{-n}\to0$ as $n\to\infty$ by Remark \ref{urepre}, whence $m_{1}(\nu')=0$. By the definition of freeness, we can further conclude that $m_{p}(\nu)=0$ for all $p\in \mathbb Z\backslash\{0\}$ or, equivalently, $\nu=\mathrm{m}$.
\end{proof}

\section{Uniqueness of L\'{e}vy-Khintchine Representation on $\mathbb{T}$} \label{sec8}

We investigate in which circumstances the L\'{e}vy triplet of an infinitely divisible distribution on the circle is unique. For simplicity, the discussions will be limited to dimension one.

We first establish two results which will not be used in this paper, but are helpful for understanding the non-uniqueness of L\'evy measures. Note that those two results can be easily generalized to higher dimensions.

\begin{prop} Let $\rho_1$ and $\rho_2$ be L\'evy measures on $\mathbb T$ such that $\nu_{\circledast}^{(1,0,\rho_1)}=\nu_{\circledast}^{(1,0,\rho_2)}$. If $\rho_1$ is Haar absolutely continuous, Haar singular continuous, or discrete, then $\rho_2$ has the same property.
\end{prop}

\begin{proof} This is a direct consequence of the fact that $\rho_1(A)=\rho_2(A)$ for any even Borel set $A$, established in Proposition \ref{prop:claim}. If $\rho_1$, for example, is Haar singular continuous, i.e., supported on a Borel subset $B$ with zero Haar measure, then the support of $\rho_2$ is contained in $B\cup B^{-1}$, which has zero Haar measure. Further, $\rho_2$ is atomless since $\rho_1(\{\xi,\xi^{-1}\})=\rho_2(\{\xi,\xi^{-1}\})=0$ for any $\xi\in\mathbb{T}$, showing that $\rho_2$ is singular continuous as well. We leave the verification of the other two cases to the reader.
\end{proof}

\begin{prop}\label{prop:unique}
Suppose that $\rho_1$ and $\rho_2$ are L\'evy measures. Then $\rho_1=\rho_2$ if and only if
\[
\int_\mathbb{T}\big(s^n-1-in\Im(s)\big)\,d\rho_1(s) = \int_\mathbb{T}\big(s^n-1-in\Im(s)\big) \,d\rho_2(s),\qquad n \in \mathbb Z.
\]
\end{prop}
\begin{proof} We sketch the proof of the ``if'' part. Let $\rho=\rho_1-\rho_2$. The proof of Lemma \ref{lem:smooth} says that $\int_{\mathbb T} f d\rho=0$ for even $f \in C^\infty$ vanishing near 0, so we only need to prove that the same holds for odd $g \in C^\infty$ vanishing near 0. Now $g$ can be approximated by the truncated Fourier sine series $g_N(x)=\sum_{n=1}^N b_n \sin nx$. Since $g_N'(0)$ converges to $g'(0)=0$, function $g_N(x)-g_N'(0)\sin x= \sum_{n=1}^N b_n(\sin nx -n \sin x)$ converges to $g_N(x)$ pointwisely. We can prove that $|g_N(x)-g_N'(0)\sin x| \leq C x^2$ on $(-\pi,\pi]$ for some constant $C>0$ independent of $N$ and $x$. The remaining arguments are similar to the corresponding proof of Lemma \ref{lem:smooth}.
\end{proof}
\begin{remark}
In view of Propositions \ref{prop:claim} and \ref{prop:unique}, two distinct L\'evy measures $\rho_1$ and $\rho_2$ give the same $\circledast$-infinitely divisible distribution if and only if there is a \emph{non-zero} function $f\colon \mathbb Z \to \mathbb Z$ such that for any $n\in \mathbb Z$,
\[
\int_\mathbb{T}\big(s^n-1-in\Im(s)\big)\,d\rho_1(s) = \int_\mathbb{T}\big(s^n-1-in\Im(s)\big) \,d\rho_2(s) +2\pi i f(n).
\]
\end{remark}

Some $\circledast$-infinitely divisible distributions have unique L\'evy measures, such as the wrapped normal distributions on the circle and delta measures; see the proof of   \cite[Corollary 4.2]{Wan08}. We are concerned with discovering other measures in $\mathcal{ID}(\circledast)$ having this feature, namely, seeking for L\'evy measures $\rho$ for which the set
\begin{equation*}
\mathcal{L}(\rho)=\big\{\widetilde{\rho}:\nu_{\circledast}^{(1,0,\rho)} =\nu_{\circledast}^{(1,0,\widetilde{\rho})}\big\},
\end{equation*}
introduced in Section \ref{sec6}, is a singleton:

\begin{pdef}
A L\'evy measure $\rho$ is said to be L-unique if $\mathcal{L}(\rho)=\{\rho\}$.
\end{pdef}

\begin{thm} \label{L-unique}
Let $\varphi\in(0,\pi]$, $c,d\geq0$, and $\alpha=e^{i\varphi}$. Then $c\delta_\alpha+d\delta_{\bar{\alpha}}$ is L-unique if $\cos\varphi=-1$ or $\cos\varphi$ does not belong to the set of dyadic rational numbers
\begin{equation} \label{dyadic}
\mathbf{D}=\left\{\frac{a}{2^b}:a,b\in\mathbb{Z},\,b\geq0\right\}.
\end{equation}
\end{thm}

\begin{proof} Denote $\rho=c\delta_\alpha+d\delta_{\bar{\alpha}}$ and let $\widetilde{\rho}\in\mathcal{L}(\rho)$. According to Proposition \ref{prop:claim}, $\widetilde{\rho}=c'\delta_\alpha+d'\delta_{\bar{\alpha}}$ with $c,d'\geq0$ and $c'+d'=c+d$. If $\varphi=\pi$, then $\alpha=\bar\alpha$ shows that $\widetilde{\rho}=\rho$. From now on, we assume that $\varphi\in(0,\pi)$.

Suppose first that $\cos\varphi$ is irrational.
Taking $n=2$ and $n=3$ in \eqref{eq:claim} with $\rho_1=\rho$ and $\rho_2=\widetilde{\rho}$ implies that
\[(c-d-c'+d')\Im(\alpha^2-2\alpha)\equiv0\mod2\pi\] and
\[(c-d-c'+d')\Im(\alpha^3-3\alpha)\equiv0\mod2\pi.\] Note that $\Im(\alpha^2-2\alpha)\neq0\neq\Im(\alpha^2-3\alpha)$. Combining these observations with $c+d=c'+d'$ results in
\begin{equation} \label{c=c'}
\frac{\ell}{\Im(\alpha^2-2\alpha)}=\frac{c-c'}{\pi}=\frac{m}{\Im(\alpha^3-3\alpha)}
\end{equation}
for some $\ell,m\in\mathbb{Z}$. Then after some simple computations we arrive at
$2\ell(\cos \varphi+1)=m$, which, along with the irrationality of $\cos\varphi$, yields that we must have $\ell=0=m$. Hence $c'=c$ and $d'=d$ by (\ref{c=c'}), i.e., $\widetilde{\rho}=\rho$.

Ultimately, assume that $\cos\varphi$ is rational but not dyadic. Then $\cos\varphi=p/q$, where $p,q$ are coprime integers and $q=2^kq'$ with $q'\geq3$ odd and $k\geq0$ an integer. Rewrite the equation \eqref{eq:claim} as that for all $n\in\mathbb{N}$,
\begin{equation} \label{eq:L-unique1}
(c-c')\big(\sin(n\varphi)-n\sin\varphi\big)\equiv0\mod\pi.
\end{equation} We shall make use of the useful formula $\sin(n\varphi)=\sin(\varphi)U_{n-1}(\cos\varphi)$ to continue the proof. Here, $U_n(x)$ is the Chebyshev polynomial, which is an integer-coefficient polynomial of degree $n$ with the leading coefficient $2^n$. Hence we have
\begin{equation} \label{eq:L-unique2}
U_{n-1}(\cos\varphi)-n=\frac{p_n}{2^{k_n}(q')^{n-1}},
\end{equation}
where $k_n\geq0$ is an integer, and $p_n$ is an integer coprime from $2^{k_n}(q')^{n-1}$. Choosing $n=2$ in \eqref{eq:L-unique1} gives that $(c-c')\sin(\varphi)/\pi\in\mathbb{Q}$. Unless $c'=c$, we infer from \eqref{eq:L-unique2} that
\[\big(U_{n-1}(\cos\varphi)-n\big)\cdot\frac{(c-c')\sin\varphi}{\pi}\]
cannot be an integer for all sufficiently large $n$. Therefore $\widetilde{\rho}=\rho$.
\end{proof}

In some exceptional angles, one can directly compute and identify the set $\mathcal{L}(c\delta_\alpha+d\delta_{\bar\alpha})$.

\begin{prop} \label{prop:exceptions}
For $c,d\geq0$ and $\varphi$, let $\alpha =e^{i\varphi}$ and $\rho=c\delta_\alpha+d\delta_{\bar{\alpha}}$.
\begin{enumerate} [$\qquad(1)$]
\item Then
$\mathcal{L}(\rho)=\left\{(c-\frac{2\pi \ell}{\sqrt{3}})\delta_{\alpha} +(d+\frac{2\pi \ell}{\sqrt{3}})\delta_{\bar{\alpha}}: \ell\in\mathbb{Z},  -[\frac{\sqrt{3}d}{2\pi}] \leq \ell \leq [\frac{\sqrt{3}c}{2\pi}]\right\}$ if $\varphi=\pi/3$.
\item Then $\mathcal{L}(\rho)=\left\{(c-\frac{\pi \ell}{2})\delta_{\alpha} +(d+\frac{\pi \ell}{2})\delta_{\bar{\alpha}}:
\ell\in\mathbb{Z},
-[\frac{2d}{\pi}]\leq \ell \leq  [\frac{2c}{\pi}]\right\}$ if $\varphi= \pi/2$.
\item Then $\mathcal{L}(\rho)=\left\{(c-\frac{2\pi \ell}{3\sqrt{3}})\delta_{\alpha} +(d+\frac{2\pi
\ell}{3\sqrt{3}})\delta_{\bar{\alpha}}: \ell \in\mathbb{Z},
-[\frac{3\sqrt{3}d}{2\pi}] \leq \ell\leq  [\frac{3\sqrt{3}c}{2\pi}]\right\}$ if $\varphi= 2\pi/3$.
\end{enumerate}
\end{prop}

\begin{remark} For these three angles, $\mathcal{L}(\rho)$ is a finite set. Moreover, the L-uniqueness of $c\delta_\alpha+d\delta_{\bar\alpha}$ depends on $c$ and $d$. For example $c\delta_i + d\delta_{-i}$ is L-unique if and only if $c,d\in[0,\pi/2)$.
\end{remark}

\begin{proof}  Let $\widetilde{\rho}\in \mathcal{L}(\rho)$.
By Proposition \ref{prop:claim}, $\widetilde{\rho}= c' \delta_\alpha + d' \delta_{\bar{\alpha}}$ with $c'+d'=c+d$. Then we can directly analyze the equations \eqref{eq:claim} to get the all possible $\widetilde{\rho}$. The details are left to the reader.
\end{proof}

We now present a statement for L\'evy measures supported on a finite set, which is slightly weaker than Theorem \ref{L-unique} when the number of atoms is two or one.

\begin{thm} \label{L-unique2}
Let $\varphi_k \in (0,\pi)$ and let $c_k,d_k \geq0$ for $k=1,2,\dots, m$, and set $\alpha_k=e^{i\varphi_k}$. If $\cos\varphi_1,\dots,\cos\varphi_m$ are algebraically independent over $\mathbb{Q}$, then the L\'evy measure $\sum_{k=1}^m(c_k\delta_{\alpha_k}+d_k\delta_{\bar{\alpha}_k})$ is L-unique.
\end{thm}

\begin{proof} First, we infer from Proposition \ref{prop:claim} that any element in $\mathcal{L}(\rho)$, where  $\rho=c_k\delta_{\alpha_k}+d_k\delta_{\bar{\alpha}_k}$, is of the form
\[\widetilde{\rho}=\sum_{k=1}^m (c_k'\delta_{\alpha_k}+d_k'\delta_{\bar{\alpha}_k}),\qquad c_k+d_k= c_k'+d_k'.\] Also, the equation \eqref{eq:claim} reads that for any $n\in\mathbb{N}$,
\begin{equation} \label{L-unique2eq}
\sum_{k=1}^m\big[U_{n-1}(\cos\varphi_k)-n\big]v_k\in\mathbb{Z},
\end{equation}
where $v_k=\pi^{-1}(c_k-c_k')\sin\varphi_k$ and $U_n$, as before, is the Chebyshev polynomial of degree $n$.
Then equation (\ref{L-unique2eq}) with $n=2,3,\dots,m+1$ tells us that the $m$-tuple $(\cos\varphi_1,\dots,\cos\varphi_m)$ is a solution to the equation
\begin{equation} \label{eq:matrix1}
\begin{pmatrix}
U_1(x_1)-2 & \cdots  &U_1(x_m)-2 \\
\vdots & \ddots & \vdots \\
U_m(x_1)- (m+1) & \cdots  &U_m(x_m)-(m+1)
\end{pmatrix}\boldsymbol v
=
\begin{pmatrix}
\ell_1 \\
\vdots\\
\ell_m
\end{pmatrix},
\end{equation} where $\boldsymbol v=(v_1,\ldots,v_m)^t$ and $(\ell_1,\dots,\ell_m)\in\mathbb{Z}^m$. Note that $U_n(1)=n+1$. Hence the determinant of the coefficient matrix $A$ in \eqref{eq:matrix1} is of the form
\begin{equation}\label{eq:det1}
\det(A) = c_m \prod_{i >j} (x_i - x_j) \prod_{i=1}^m (x_i-1)
\end{equation}
as it is a polynomial of degree $m(m+1)/2$, and it vanishes
when $x_i=x_j$ for any $i\neq j$ and when $x_i=1$ for any $i$. Further, the coefficient of $x_1x_2^2\cdots x_m^{m}$ is $c_m=2^{1+2+\cdots+m}$. Consequently, if all $x_i$'s are mutually distinct, then $A$ is invertible. Similarly, letting $n=2,3,\dots,m,m+2$ in \eqref{L-unique2eq} gives
\begin{equation}\label{eq:matrix2}
\begin{pmatrix}
U_1(x_1)-2 & \cdots  &U_1(x_m)-2 \\
\vdots & \vdots & \vdots \\
U_{m-1}(x_1)- m & \cdots  &U_{m-1}(x_m)-m  \\
U_{m+1}(x_1)- (m+2) & \cdots  &U_{m+1}(x_m)-(m+2)
\end{pmatrix}\boldsymbol v
=
\begin{pmatrix}
\ell_1 \\
\vdots\\
\ell_{m-1} \\
\ell_{m+1}
\end{pmatrix},
\end{equation}
where $\ell_{m+1}\in\mathbb{Z}$. Similar to the previous case, the determinant of the coefficient matrix $B$ in \eqref{eq:matrix2} admits the form
\begin{equation}
\det(B)=2c_m(x_1+\cdots+x_m+a)\prod_{i >j}(x_i-x_j)\prod_{i=1}^m(x_i-1),
\end{equation}
where $a$ is a real number. The additional linear term $x_1+\cdots+x_m+a$ appears since now the degree of the determinant increases by $1$. To determine the unknown $a$ (in fact, it is enough and easier to show that $a$ is rational), it is useful to introduce a new variable $y_k=x_k-1$. More precisely, using the expression
\[U_n(x)-(n+1)=\sum_{j=1}^n \frac{(n+j+1)!}{(n-j)!(2j+1)!}(2y)^j,\qquad y=x-1,\]
and comparing the coefficients of $y_1y_2^2\cdots y_m^m$ in $\det(B)$ deduces that $a=1$.

Cramer's rule implies that $(\det A)^{-1}\det A_m=v_m=(\det B)^{-1}\det(B_m)$, where $A_m$ and $B_m$ are matrices obtained by replacing the $m$-th columns of $A$ and $B$ with the vectors $(\ell_1,\dots,\ell_{m-1},\ell_m)^t$ and $(\ell_1,\dots,\ell_{m-1},\ell_{m+1})^t$, respectively. After some cancellations, we obtain the equation
\begin{align*}
&2(x_1 + \cdots + x_m+1) \left|\begin{matrix}
U_1(x_1)-2 & \cdots & U_1(x_{m-1})-2&\ell_1 \\
\vdots & \ddots & \vdots \\
U_{m-1}(x_1)- m & \cdots  &U_{m-1}(x_{m-1})- m &\ell_{m-1}  \\
U_m(x_1)- (m+1) & \cdots &U_m(x_{m-1})- (m+1) &\ell_m
\end{matrix}\right|  \\
&\qquad =
\left|\begin{matrix}
U_1(x_1)-2 & \cdots &U_1(x_{m-1})-2 &\ell_1 \\
\vdots & \ddots & \vdots \\
U_{m-1}(x_1)- m & \cdots &U_{m-1}(x_{m-1})- m  &\ell_{m-1}  \\
U_{m+1}(x_1)- (m+2) & \cdots &U_{m+1}(x_{m-1})- (m+2)  &\ell_{m+1}
\end{matrix}\right|.
\end{align*}
Since $\cos \varphi_1,\dots,\cos\varphi_m$ are algebraically independent and satisfy this equation, the polynomials of both sides must be identical. Comparing the coefficients of $x_1^2 x_2^3\cdots x_{m-1}^{m}x_m$ results in $\ell_1=0$. Calculating the coefficients of $x_1x_2^3x_3^4\cdots x_{m-1}^{m}x_m$ yields that $\ell_2=0$. Similarly, one can prove that $\ell_3=\cdots=\ell_m=0$. These findings and \eqref{eq:matrix1} imply that $\boldsymbol v=\boldsymbol0$, proving $c_k'=c_k$ and $d_k'=d_k$ for all $k=1,\dots,m$.
\end{proof}

Denote by $\mathscr{P}_\mathbb{T}^\times$ the set of probability measures on $\mathbb{T}$ having non-zero mean.

\begin{cor} The set of $\circledast$-infinitely divisible distributions on $\mathbb{T}$ having L-unique L\'evy measures is weakly dense in $\mathcal{ID}(\circledast)\cap\mathscr{P}_\mathbb{T}^\times$.
\end{cor}

\begin{proof} The set $A$ of tuples $\boldsymbol\alpha=(\alpha_1,\alpha_2,\dots)\in\bigcup_{m\geq1}\mathbb{T}^{m}$ of finite length (denoted by $|\boldsymbol\alpha|$), for which the tuple $(\cos (\arg \alpha_1),  \cos( \arg \alpha_2),\dots)$ is algebraically dependent, is countable since the set of polynomials over $\mathbb{Q}$ is countable. Any L\'evy measure can be approximated in the $\mathscr{M}_{\mathbb T}^1$-topology by discrete ones, which further can be approximated by discrete measures $\sum_{k=1}^{|\boldsymbol \alpha|}(c_k\delta_{\alpha_k}+d_k\delta_{\bar{\alpha}_k})$ with $\boldsymbol\alpha\in(\bigcup_{m\geq1}\mathbb{T}^{m})\setminus A$.
\end{proof}

\section*{Acknowledgments} The first-named author is granted by JSPS kakenhi (B) 15K17549 and 19K14546, while the second-named author is supported by the Ministry of Science and Technology of Taiwan under the research grant MOST 106-2628-M-110-002-MY4. This research is an outcome of Joint Seminar supported by JSPS and CNRS under the Japan-France Research Cooperative Program.

\end{document}